\documentclass[11pt]{article}

\usepackage[a4paper,margin=1in]{geometry}
\usepackage{amsmath,amssymb,amsthm,mathtools}
\usepackage{eucal}
\usepackage{microtype}
\usepackage{fancyhdr}
\usepackage{lastpage}
\usepackage[
    bookmarksnumbered=true,
    bookmarksopen=true,
    bookmarksopenlevel=2,
    pdfpagemode=UseOutlines
]{hyperref}

\hypersetup{
    colorlinks=true,
    linkcolor=blue,
    citecolor=blue,
    urlcolor=blue
}

\fancypagestyle{main}{%
    \fancyhf{}%
    \cfoot{Page \thepage\ of \pageref*{LastPage}}%
}
\fancypagestyle{plain}{%
    \fancyhf{}%
    \cfoot{Page \thepage\ of \pageref*{LastPage}}%
}

\allowdisplaybreaks

\newcommand{\R}{\mathbb R}
\newcommand{\E}{\mathbb E}
\newcommand{\Var}{\operatorname{Var}}
\newcommand{\PP}{\mathbb P}
\newcommand{\norm}[1]{\left\lvert #1 \right\rvert}
\newcommand{\Norm}[1]{\left\lVert #1 \right\rVert}

\newcommand{\ip}[2]{\left\langle #1,#2\right\rangle}
\newcommand{\N}{\mathbb N}
\DeclareMathOperator{\Cov}{Cov}
\DeclareMathOperator{\diag}{diag}
\DeclareMathOperator{\tr}{Tr}
\DeclareMathOperator{\argmin}{argmin}
\DeclareMathOperator{\op}{op}

\theoremstyle{plain}
\newtheorem{theorem}{Theorem}[section]
\newtheorem{lemma}[theorem]{Lemma}
\newtheorem{corollary}[theorem]{Corollary}
\newtheorem{proposition}[theorem]{Proposition}

\theoremstyle{definition}
\newtheorem{assumption}{Assumption}
\newtheorem{remark}[theorem]{Remark}
\newtheorem{example}[theorem]{Example}

\numberwithin{equation}{section}
\title{A Central Limit Theorem for Regularized M-Estimators}
\author{Cosme Louart}
\date{}

\begin{document}

\maketitle

\begin{abstract}
We prove a quantitative central limit theorem for linear functionals of
regularized empirical-risk minimizers in the proportional-dimensional regime
\(p=O(n)\). The data columns are independent, not necessarily identically
distributed, and satisfy a uniform columnwise Poincar\'e inequality. Under
uniform curvature and smoothness assumptions, and for a quadratic regularizer,
we show that every nondegenerate statistic
\(\sqrt n\,u^\top\hat\theta\), centered by its expectation and normalized by
its standard deviation, converges to a standard normal random variable in
Wasserstein distance, with rate \(O((\log n)^7n^{-1/4})\). The proof is based
on moment and stability bounds for the minimizer, a second-order leave-one-out
expansion, and a perturbative normal-approximation argument for functions of
independent variables. We also prove the variance upper bound
\(\Var(u^\top\hat\theta)\le C\norm{u}_2^2/n\), identifying the \(\sqrt n\)
fluctuation scale.
\end{abstract}
\section{Introduction}

We consider a regularized empirical-risk minimizer
\[
    \hat\theta
    =
    \argmin_{\theta\in\R^p}
    \left\{
        \frac1n\sum_{i=1}^n L(x_i^\top\theta)+\rho(\theta)
    \right\}.
\]
Here \(X=(x_1,\ldots,x_n)\in\R^{p\times n}\) has independent columns, the loss
\(L:\R\to\R\) is smooth and uniformly convex, and the regularizer
\(\rho:\R^p\to\R\) is strongly convex. Problems of this form are central in
classical \(M\)-estimation, robust statistics, and convex learning theory. In
fixed dimension, consistency and asymptotic normality for smooth
\(M\)-estimators are classical; see Huber~\cite{Huber1967} and
van der Vaart~\cite{vanDerVaart1998}.

When the dimension grows with the sample size, the fixed-dimensional theory no
longer applies directly. Classical increasing-dimension normal approximation
results, such as those of Portnoy~\cite{Portnoy1985},
Mammen~\cite{Mammen1989}, and He--Shao~\cite{HeShao2000}, require the
dimension to grow sufficiently slowly compared with the sample size. In the
proportional regime, where \(p/n\) is of constant order, central limit theorems
and exact-asymptotic descriptions have been obtained in more structured
settings. Donoho--Montanari~\cite{DonohoMontanari2016},
Sur--Cand\`es~\cite{SurCandes2019},
Zhao--Sur--Cand\`es~\cite{ZhaoSurCandes2022}, and
Bellec--Shen--Zhang~\cite{BellecShenZhang2022} work under Gaussian-design
assumptions, possibly with general covariance. Lei--Bickel--El
Karoui~\cite{LeiBickelElKaroui2018} prove coordinate-wise normality under
fixed-design regularity conditions. El Karoui and collaborators
\cite{ElKarouiBeanBickelLimYu2013,ElKaroui2013} allow non-Gaussian designs, but
the assumptions still impose iid, isotropic, or independent-entry structure
together with strong concentration properties.

The present paper proves a quantitative central limit theorem under weaker
distributional assumptions on the design. We work in the regime \(p=O(n)\), and
the columns \(x_i\) are independent but need not be identically distributed,
Gaussian, centered, isotropic, or have iid entries. The main concentration
assumption is instead a uniform columnwise Poincar\'e inequality. Under the
smoothness and curvature assumptions stated below, and for a quadratic
regularizer, we prove a Wasserstein CLT for arbitrary deterministic linear
projections of the minimizer. More precisely, for deterministic directions
\(u=u_n\in\R^p\) with \(\norm{u}_2\le1\), set
\[
    \sigma_n^2
    :=
    \Var\!\big(\sqrt n\,u^\top\hat\theta\big).
\]
If \(\inf_n\sigma_n^2>0\), then
\[
    \frac{
        \sqrt n\,u^\top\hat\theta
        -
        \E[\sqrt n\,u^\top\hat\theta]
    }{\sigma_n}
    \Rightarrow
    \mathcal N(0,1)
\]
quantitatively in Wasserstein distance, with rate
\(O((\log n)^7n^{-1/4})\). We also prove the variance upper bound
\[
    \Var(u^\top\hat\theta)
    \lesssim
    \frac{\norm{u}_2^2}{n},
\]
so that the only variance assumption in the CLT is the matching lower bound.

The proof combines leave-one-out expansions with the perturbative
normal-approximation method of Chatterjee~\cite{Chatterjee2008}; see also the
quantitative Berry--Esseen theory for nonlinear statistics and \(M\)-estimators
of Shao--Zhang~\cite{ShaoZhang2022}. The key step is a second-order
leave-one-out approximation for the increments of
\(\sqrt n\,u^\top\hat\theta\). This expansion uses an averaged-leverage score
to remove the leading dependence of the active score on the active column. Once
this approximation is established, moment estimates and conditional-covariance
bounds verify a perturbative Wasserstein criterion for asymptotic normality.

\tableofcontents

\section{Setting, assumptions, and main result}
\label{sec:assumptions}

\subsection{Notation}

We use the following notation. Given $i\in \mathbb N$, denote:
\[
    [i]:=\{1,\ldots,i\}
    \qquad\text{and}\qquad
    [0]:=\emptyset .
\]
For a measurable function
\(f:\mathbb R^p\to\mathbb R\), a finite-dimensional vector \(a=(a_i)\), and
\(A\in\mathcal M_{p,n}\), set
\begin{itemize}
    \item \(\norm{f}_{\infty}:=\sup_{x\in\mathbb R^p}|f(x)|\);
    \item \(\norm{a}_2:=\big(\sum_i a_i^2\big)^{1/2}\) and
    \(\norm{a}_\infty:=\sup_i |a_i|\);
    \item \(\norm{A}_{\op}:=\sup_{\norm{x}_2\le1}\norm{Ax}_2\) and
    \(\norm{A}_F:=\sqrt{\tr(AA^\top)}\).
\end{itemize}

Given \(f\in\mathcal C^2(\mathbb R^p)\) and \(z\in\mathbb R^p\),
\[
    \nabla f(z)
    =
    \begin{pmatrix}
        \frac{\partial f(z)}{\partial z_1},\ldots,
        \frac{\partial f(z)}{\partial z_p}
    \end{pmatrix}^{\!\top},
    \qquad
    \nabla^2 f(z)
    =
    \left(
        \frac{\partial^2 f(z)}{\partial z_a\partial z_b}
    \right)_{a,b=1}^p .
\]

For \(\ell\) indexing a column of
\(Z=(z_1,\ldots,z_n)\in\mathcal M_{p,n}\), we write \(\mathbb D_\ell\)
for the Fr\'echet derivative with respect to \(z_\ell\). If \(T=T(Z)\) is
real-, vector-, or matrix-valued, then \(\mathbb D_\ell[T(Z)][h]\) denotes
the derivative in the direction \(h\in\mathbb R^p\). For instance, when
\(T\) is vector-valued,
\[
    \norm{
        T(Z+h e_\ell^\top)-T(Z)-\mathbb D_\ell[T(Z)][h]
    }_2
    =
    o(\norm{h}_2).
\]

When \(T\) is matrix-valued, we write
\[
    \norm{\mathbb D_\ell[T]}_{\op}^{*}
    :=
    \sup_{\norm{h}_2\le1}
    \norm{\mathbb D_\ell[T][h]}_{\op},
\]
with analogous notation for the other norms and for real- or vector-valued
maps. For example, given \(u\in\mathbb R^p\),
\[
    \norm{u^\top\mathbb D_\ell[T]}_2^{*}
    :=
    \sup_{\norm{h}_2\le1}
    \norm{u^\top\mathbb D_\ell[T][h]}_2 .
\]

We will keep the $\|\cdot\|$ notation for the integrated norms defined for any $k>0$ and any random variable $Z\in \mathbb R$ as:
\begin{align*}
    \|Z\|_{L^k}:= \mathbb E[|Z|^k]^{\frac{1}{k}}.
\end{align*}
Given a statistic $g((z_1,\ldots, z_n))$ depending on $n$ independent variables $z_1,\ldots, z_n$, we denote:
\begin{align*}
    \|g((z_1,\ldots, z_n))\|_{L_i^k}:=\mathbb E[|g((z_1,\ldots, z_n))|^k \mid (z_j)_{j\in[n]\setminus\{i\}}]^{\frac{1}{k}}.
\end{align*}
Similarly, \(\E_i\), \(\Var_i\), and \(\Cov_i\) denote conditional expectation,
variance, and covariance with respect to \(z_i\) only, conditionally on all
other variables.

For two nonnegative sequences $(a_n)$ and $(b_n)$, we write $a_n=O(b_n)$ if there exists $C>0$ such that $a_n\le Cb_n$ for all large enough $n$, and $a_n=o(b_n)$ if $a_n/b_n\to0$. All implicit constants are deterministic and independent of $n$ and $p$, unless explicitly stated otherwise.

\subsection{Setting}
Let $X=(x_1,\dots,x_n)\in\R^{p\times n}$ have independent columns
$x_1,\dots,x_n\in\R^p$. For $\theta\in\R^p$, set
\[
    \Psi_X(\theta)
    :=
    \frac1n\sum_{i=1}^n L(x_i^\top\theta)+\rho(\theta),
    \qquad
    \hat\theta
    :=
    \argmin_{\theta\in\R^p}\Psi_X(\theta).
\]
The assumptions below imply that $\Psi_X$ is strongly convex, so this minimizer is unique. We view $(p,X,L,\rho)=(p_n,X_n,L_n,\rho_n)$ as a sequence indexed by $n$, and we suppress the dependence on $n$ in the notation.

\begin{assumption}
\label{ass:1dim}
$p=O(n)$.
\end{assumption}

\begin{assumption}
\label{ass:2origin}
$|L'(0)|=O(1)$ and $\norm{\nabla\rho(0)}_2=O(1)$.
\end{assumption}

\begin{assumption}
\label{ass:3bounded-means}
$\sup_{i\in[n]}\norm{\E[x_i]}_2=O(1)$.
\end{assumption}

\begin{assumption}[Poincar\'e inequality for the columns]
\label{ass:poincare-data}
There exists $C_P=O(1)$ such that, for every $i\in[n]$ and every $C^1$ function
$f:\R^p\to\R$ for which the right-hand side is finite,
\[
    \Var(f(x_i))
    \le
    C_P\,\E\norm{\nabla f(x_i)}_2^2.
\]
\end{assumption}

\begin{assumption}
\label{ass:5regularity}
The loss $L:\R\to\R$ is $C^2$, $\norm{L''}_\infty=O(1)$, and there exists a
constant \(\kappa_L>0\), with \(\kappa_L^{-1}=O(1)\), such that, for all
\(t\in\R\),
\[
    L''(t)\ge \kappa_L .
\]
The regularizer \(\rho:\R^p\to\R\) is \(C^2\) and \(\kappa\)-strongly convex
for some constant \(\kappa>0\), with \(\kappa^{-1}=O(1)\), namely
\[
    \nabla^2\rho(\theta)\succeq \kappa I_p,
    \qquad
    \theta\in\R^p.
\]
\end{assumption}

These assumptions ensure well-posedness and provide the basic moment and concentration bounds for $\hat\theta$. To derive a first-order expansion of the leave-one-out perturbation $\hat\theta-\hat\theta_{-i}$, we also require the following smoothness condition.

\begin{assumption}
\label{ass:6smooth-curvature}
The second derivative $L''$ is globally Lipschitz with constant $C_L=O(1)$:
\[
    |L''(t)-L''(t')|
    \le
    C_L |t-t'|,
    \qquad t,t'\in\R.
\]
Moreover, $\nabla^2\rho$ is globally Lipschitz in operator norm with constant
$C_\rho=O(1)$:
\[
    \norm{\nabla^2\rho(\theta)-\nabla^2\rho(\theta')}_{\op}
    \le
    C_\rho\norm{\theta-\theta'}_{2},
    \qquad
    \theta,\theta'\in\R^p.
\]
\end{assumption}

Finally, the CLT proof uses a second-order expansion of the leave-one-out perturbation. For this step we impose the following additional curvature condition.

\begin{assumption}[Third-order curvature]
\label{ass:7third-order-curvature}
The loss $L$ is $C^4$ with\footnote{The bound $\norm{L''''}_{\infty}=O(1)$ will actually only be used in Section~\ref{sec:sensitivity-covariances} but to stay simple we did not isolate it in an independent assumption.}:
\[
    \norm{L'''}_{\infty}= O(1),
    \qquad
    \text{and}
    \qquad
    \norm{L''''}_{\infty}= O(1).
\]
The regularizer $\rho$ is $C^3$. We denote its third-order differential tensor by $\nabla^3\rho(\theta)$, so that, for any $h\in\R^p$, $\nabla^3\rho(\theta)[h]$ is a symmetric matrix. There exists a constant $\gamma_\rho=O(1)$ such that
\[
    \sup_{\theta\in \mathbb R^p}\sup_{\norm{h}_2\leq 1}\norm{\nabla^3\rho(\theta)[h]}_{\op}
    \le
    \gamma_\rho,
\]
and for all $\theta,\theta',h\in\R^p$:
\[
    \norm{\nabla^3\rho(\theta)[h]-\nabla^3\rho(\theta')[h]}_{\op}
    \le
    \gamma_\rho\norm{\theta-\theta'}_2\,\norm{h}_2.
\]
\end{assumption}

Fix a deterministic sequence of directions $u=u_n\in\R^p$ such that
\[
    \norm{u}_2\le1.
\]
Define
\begin{align}\label{eq:def_f_n}
    f_n(X)
    :=
    \sqrt n\,u^\top\hat\theta(X),
    \qquad
    \sigma_n^2
    :=
    \Var\!\big(f_n(X)\big).
\end{align}
Whenever \(\sigma_n>0\), set
\[
    W_n
    :=
    \frac{f_n(X)-\E f_n(X)}{\sigma_n}.
\]
\subsection{Main result}

\begin{theorem}[CLT for linear functionals of the minimizer]
\label{thm:main-clt}
Under Assumptions~\ref{ass:1dim}--\ref{ass:7third-order-curvature}, assume in addition that \(\rho\) is quadratic, in the sense that
\[
    \nabla^3\rho\equiv0,
    \qquad\text{and}\qquad
    \sigma_n^{-1}=O(1).
\]
Then
\[
    d_W(W_n,\mathcal N)
    =
    O\left( \frac{(\log n)^{13}}{\sqrt n}\right),
\]
where $\mathcal N$ denotes a standard normal random variable and
\[
    d_W(Y,Z)
    :=
    \sup_{\norm{h}_{\mathrm{Lip}}\le1}
    \left|
        \E h(Y)-\E h(Z)
    \right|
\]
is the Wasserstein distance.
\end{theorem}

The proof is based on Chatterjee's normal-approximation method \cite{Chatterjee2008} that was further simplified by Shao and Zhang in \cite{ShaoZhang2025Functionals}. We provide below a very similar version adapted to our needs.

Let $Z=(z_1,\ldots,z_n)$ have independent coordinates, and let
$Z'=(z_1',\ldots,z_n')$ be an independent copy. For $A\subset[n]$, let
$Z^A$ be obtained from $Z$ by replacing $z_j$ by $z_j'$ for every
$j\in A$. With this notation:
\begin{align*}
    Z^{[i]} = (z_1',\ldots, z_i', z_{i+1},\ldots, z_n).
\end{align*}
For a random variable $Y=Y(Z,Z')$, let us denote the expectation obtained by integrating on $z_i$ by
\[
    \E_i[Y]:=\E\left[Y \mid \{z_j,z_j'\}_{j\in [n]} \setminus \{z_i\}\right].
\]
The notation $\Cov_i$ has the analogous meaning.

\begin{theorem}
\label{the:conditional-loo-perturbative}
Let $Z=(Z_1,\ldots,Z_n)$ have independent entries in $\mathcal X$, and let
$Z'=(Z_1',\ldots,Z_n')$ be an independent copy of $Z$.
Let $g_n:\mathcal X^n\to\mathbb R$ be measurable and assume that
\[
    \E[g_n(Z)]=0,
    \qquad
    \Var(g_n(Z))=1.
\]
Then, denoting the intrinsic local covariances and increments by
\[
    c_i
    :=
    \Cov_i\!\left(g_n(Z),g_n(Z^{[i-1]})\right)
    \qquad \text{and}\qquad
    \Delta_i:= g_n(Z)-\E_i[g_n(Z)],
\]
we have $\mathbb E \left[ \sum_{i=1}^n c_i \right] = 1$ and
\[
\begin{aligned}
    d_W(g_n(Z),\mathcal N(0,1))
    &\le \sqrt{\frac{2}{\pi}}
        \mathbb E\left|
        1-
        \sum_{i=1}^n
            c_i
        \right|
    +4\sum_{i=1}^n
    \E\left[
        |\Delta_i|^3
    \right].
\end{aligned}
\]
\end{theorem}
The result could not be directly deduced from~\cite{ShaoZhang2025Functionals}, we thus provide the full proof in Appendix~\ref{subsec:proof_of_clt-perturbative} for completeness.

\begin{remark}\label{rem:simplification_wassertein_conv}
    In the setting of Theorem~\ref{the:conditional-loo-perturbative} we can rely on the weaker formulation:
    \begin{align*}
        \begin{aligned}
    d_W(g_n(Z),\mathcal N(0,1))
    &\le
    \sqrt{\frac{2}{\pi}}
    \sqrt{\Var\!\left(
        \sum_{i=1}^n c_i
    \right)}
    +4\sum_{i=1}^n
    \E\left[
        |\Delta_i|^3
    \right] \\
    &\le
    \sqrt{\frac{2}{\pi}}
    \sum_{i=1}^n
        \sqrt{\Var(c_i)}
    +4\sum_{i=1}^n
    \E\left[
        |\Delta_i|^3
    \right],
\end{aligned}
    \end{align*}
    thanks to Cauchy--Schwarz and the fact that $\mathbb E[\sum_{i=1}^nc_i] = 1$.
\end{remark}

\begin{example}[Why the cubic local term is needed]
\label{ex:single-coordinate-cubic-term}
Let $z_1,\ldots,z_n$ be independent Rademacher variables and set $g_n(Z):=z_1$.
Trivially, $\E g_n(Z)=0$ and $\Var(g_n(Z))=1$, but $g_n(Z)$ is Rademacher for
all $n$, and hence does not converge to a standard Gaussian. The covariance part of the bound provided by Theorem~\ref{the:conditional-loo-perturbative} alone would
see no obstruction since:
\begin{align*}
    \sum_{i=1}^n c_i = c_1 + \sum_{i=2}^n c_i = \Cov_1(z_1,z_1) =1,
\end{align*}
deterministically. The obstruction is exactly the large local increment,
\[
    \sum_{i=1}^n \E |\Delta_i|^3 = \E |z_1|^3 + \sum_{i=2}^n \E |\Delta_i|^3 =1 \not \to 0.
\]
\end{example}

\begin{example}[Why the intrinsic local covariances terms and the ordered resampling is needed]
\label{ex:folded-rademacher-ordered-resampling}
Let $n$ be even and let $z_1,\ldots,z_n$ be independent Rademacher
variables. Set
\[
    S_n := \frac1{\sqrt n}\sum_{j=1}^n z_j,
    \qquad
    a_n := \E |S_n|,
    \qquad
    b_n^2 := \Var(|S_n|),
\]
and define the centered and normalized statistic
\[
    g_n(Z) := \frac{|S_n|-a_n}{b_n}.
\]
Then $\E g_n(Z)=0$ and $\Var(g_n(Z))=1$. Moreover, by the classical
central limit theorem,
\[
    g_n(Z)
    \Longrightarrow
    \frac{|G|-\sqrt{2/\pi}}{\sqrt{1-2/\pi}},
    \qquad G\sim\mathcal N(0,1),
\]
which is not Gaussian. 

We now compare the ordered covariance appearing in
Theorem~\ref{the:conditional-loo-perturbative} with the diagonal analogue
where $g_n(Z^{[i-1]})$ is replaced by $g_n(Z)$. Put
\[
    R_i := \sum_{j\ne i} z_j,
    \qquad
    R_i^{[i-1]} := \sum_{j<i}z_j' + \sum_{j>i}z_j .
\]
Assuming, without loss of generality that $n$ is even, we can deduce that $R_i$ and $R_i^{[i-1]}$ are almost surely non-zero as sums of an odd number of Rademacher variables. Then $ \E_i |S_n| = \mathbb E[|z_i|] \frac{|R_i|}{\sqrt n} = \frac{|R_i|}{\sqrt n}$ and therefore
\[
    \Delta_i
    = g_n(Z)-\E_i g_n(Z)
    = \frac{1}{b_n\sqrt n} \left( |z_i + R_i| - |R_i| \right)  = \frac{z_i\,\operatorname{sgn}(R_i)}{b_n\sqrt n}.
\]
In particular
\[
    \sum_{i=1}^n \E |\Delta_i|^3
    =\sum_{i=1}^n \E \left[ \left( \frac{|z_i|}{b_n^3\sqrt n} \right)^3 \right]
    = \frac1{b_n^3\sqrt n}
    \longrightarrow 0,
\]
because $b_n^2\to 1-2/\pi>0$.

The ordered covariance of
Theorem~\ref{the:conditional-loo-perturbative} is
\[
    c_i
    := \Cov_i\bigl(g_n(Z),g_n(Z^{[i-1]})\bigr)
    =\Cov_i\bigl(\Delta_i,\Delta_i^{[i-1]}\bigr)
    = \frac{\operatorname{sgn}(R_i)\operatorname{sgn}(R_i^{[i-1]})}
           {b_n^2 n}.
\]
One can show that, uniformly for $i/n\in[1/3,2/3]$, $\Var\bigl(
        \operatorname{sgn}(R_i)\operatorname{sgn}(R_i^{[i-1]})
    \bigr)
    \ge \kappa$, for some numerical constant $\kappa>0$ and all large $n$. Therefore
\[
    \sum_{i=1}^n \sqrt{\Var(c_i)}
    \ge
    \sum_{n/3\le i\le 2n/3}
        \frac{\sqrt\kappa}{b_n^2 n}
    \ge c>0
\]
for all large $n$. Thus the ordered covariance term does detect the
failure of asymptotic normality. The role of $Z^{[i-1]}$ is precisely to
compare the local derivative in two environments which share $z_i$ but
have different past coordinates; this reveals the random orientation
$\operatorname{sgn}(R_i)$ that is invisible to the diagonal variance
$\operatorname{sgn}(R_i)^2$.

The random sign in $c_i$ is the obstruction that the diagonal proxy
\[
    \tilde c_i
    := \Cov_i\bigl(g_n(Z),g_n(Z)\bigr)
    = \Var_i(g_n(Z)).
\]
squares away. We get indeed:
\[
    \tilde c_i
    = \Var_i(\Delta_i)
    = \frac1{b_n^2 n},
\]
which is deterministic. Consequently $\sum_{i=1}^n \sqrt{\Var(\tilde c_i)} = 0$.
The diagonal fluctuation criterion sees no obstruction, although
$g_n(Z)$ does not converge to a Gaussian. Notice also that the diagonal
quantity is not the correct Stein variance proxy: in this example
\[
    \E \left[ \sum_{i=1}^n \tilde c_i \right] = \frac1{b_n^2}\ne 1.
\]

\end{example}

With the notation of \eqref{eq:def_f_n}, we apply
Theorem~\ref{the:conditional-loo-perturbative} with \(Z=X\) and \(g_n=\frac{1}{\sigma_n}(f_n-\E f_n)\) with the nottions of~\eqref{eq:def_f_n}.
Our approach consists in introducing a leave-one-out statistic
\[
    f_n^{-i}(X_{-i})
    :=
    \sqrt n\,u^\top\hat\theta_{-i},
\]
where
\begin{align}\label{eq:def_loo}
    \hat\theta_{-i}
    :=
    \argmin_{\theta\in\R^p}\Psi_{X_{-i}}(\theta),
    \qquad
    \Psi_{X_{-i}}(\theta)
    :=
    \frac1n\sum_{j\ne i}L(x_j^\top\theta)+\rho(\theta).
\end{align}
The normalization is kept equal to \(1/n\), so that \(\Psi_{X_{-i}}\) is the original objective with the \(i\)th loss term removed. We then replace the appearance of $f_n(X)-\mathbb E_i[f_n(X)]$ with
\begin{align*}
    \delta_if_n(X):= f_n(X) - f_n^{-i}(X_{-i}) = \sqrt n u^\top(\hat\theta-\hat\theta_{-i}).
\end{align*}
That is made possible thanks to the identities
\begin{align*}
    \Cov_i\!\left( f_n(X), f_n(X^{[i-1]})  \right)
    =
    \Cov_i\!\left( \delta_if_n(X), \delta_if_n(X^{[i-1]}) \right)
\end{align*}
and
\begin{align*}
    \left\Vert f_n(X)-\mathbb E_i[f_n(X)] \right\Vert_{L^3}
    &=\left\Vert \delta_if_n(X)-\mathbb E_i[\delta_if_n(X)] \right\Vert_{L^3}\\
    &\leq \left\Vert \delta_if_n(X)\right\Vert_{L^3}+\left\Vert\mathbb E_i[\delta_if_n(X)] \right\Vert_{L^3}
    \leq 2 \left\Vert \delta_if_n(X)\right\Vert_{L^3}.
\end{align*}

Then one can deduce from Theorem~\ref{the:conditional-loo-perturbative} and more specifically from Remark~\ref{rem:simplification_wassertein_conv} that, introducing the numerical constant $C= \max (\sqrt{2/\pi}, 32)$, one can bound
\begin{align}\label{eq:loo_bound_wasserstein_bound_overview}
    d_W \left( \frac{f_n(X)- \mathbb E[f_n(X)]}{\sigma_n}, \mathcal N \right)
    \leq C \left( \frac{n}{\sigma_n^2} \omega_c + \frac{n}{\sigma_n^3}\omega_\Delta^3 \right),
\end{align}
where we denoted
\begin{align*}
    \omega_c := \sup_{i\in[n]}\sqrt{\Var\!\left(\Cov_i\!\left(\delta_if_n(X), \delta_if_n(X^{[i-1]})\right)\right)}
    &&\text{and}&&
    \omega_{\Delta}:= \sup_{i\in[n]}\|\delta_if_n(X)\|_{L^3}.
\end{align*}

To prove Theorem~\ref{thm:main-clt}, we establish the following estimates:
\begin{enumerate}
    \item $\sigma_n=O(1)$ in Section~\ref{sec:preliminary_concentration_properties}, Remark~\ref{rem:sigma-upper};
    \item $\omega_\Delta
    =
    O\!\left(\frac{\log n}{\sqrt n}\right)$ in Subsection~\ref{sse:first_order},~\eqref{eq:borne_omega_delta};
    \item $\omega_c
    =
    O\!\left(
        \frac{(\log n)^{13}}{n^{3/2}}
    \right)$ in Subsection~\ref{subsec:contracted-active-background}, Corollary~\ref{cor:bound-cn-bar-x}.
\end{enumerate}
Together with the nondegeneracy assumption $\sigma_n^{-1}=O(1)$, these bounds imply the corresponding Wasserstein rate.

The argument proceeds in three main stages. First, we prove moment, stability,
and variance bounds for the minimizer and its leave-one-out analogues. Second,
we derive first- and second-order leave-one-out expansions for
\(u^\top(\hat\theta-\hat\theta_{-i})\). Third, we insert these expansions into
the perturbative normal-approximation bound and verify the required sensitivity
and covariance estimates.

\section{Preliminary concentration properties}
\label{sec:preliminary_concentration_properties}

\subsection{Basic concentration results on the data}
We first collect the concentration estimates for the data and the minimizer that
will be used throughout the proof. The general Poincar\'e facts are recalled in
Appendix~\ref{sec:preliminaries_data}; here we only record their consequences
for the independent columns \(x_1,\ldots,x_n\).

\begin{lemma}[Tensorized Poincar\'e inequality for the data]
\label{lem:concentration_X}
Under Assumption~\ref{ass:poincare-data}, for every sufficiently smooth
\(f:\mathcal M_{p,n}\to\R\),
\[
    \Var(f(X))
    \le
    C_P\sum_{i=1}^n \|\norm{\mathbb D_i [f(X)]}_2^{*}\|_{L^ 2}^2 .
\]
\end{lemma}

\begin{proof}
This is Proposition~\ref{prop:tensorization} applied to the independent
columns \(x_1,\ldots,x_n\).
\end{proof}

The next lemma gives the data moment bounds used below. It is a direct
consequence of Lemma~\ref{lem:linear-moments-poincare},
Lemma~\ref{lem:column-euclidean-moments}, and the boundedness of the column
means.

\begin{lemma}[Column moment bounds]
\label{lem:moments_of_x_i}
Under Assumptions~\ref{ass:3bounded-means} and~\ref{ass:poincare-data}, for any
sequence of positive integers \(k_n\geq 1\), and any deterministic sequence
\(u=u_n\in\R^p\) satisfying \(\norm{u}\le1\),
\[
    \sup_{i\in[n]}
    \Norm{u^\top x_i }_{L^{k_n}}
    =
    O(k_n).
\]
If, in addition, Assumption~\ref{ass:1dim} holds, then
\[
    \sup_{i\in[n]}
    \Norm{\norm{x_i}_2 }_{L^{k_n}}
    =
    O(k_n\sqrt n).
\]
\end{lemma}

\begin{proof}
Lemma~\ref{lem:linear-moments-poincare} gives
\[
    \sup_{i\in[n]}
    \Norm{u^\top(x_i-\E x_i) }_{L^{k_n}}
    =
    O(k_n),
    \qquad \norm{u}\le1.
\]
Since
\[
    |u^\top\E x_i|
    \le
    \norm{\E x_i}_2
    =
    O(1)
\]
uniformly in \(i\) by Assumption~\ref{ass:3bounded-means}, the first bound
follows by the triangle inequality. The second bound follows from
Lemma~\ref{lem:column-euclidean-moments} and \(p=O(n)\).
\end{proof}

Combining Lemma~\ref{lem:moments_of_x_i} with
Lemma~\ref{lem:max-polynomial-moment-growth} also gives the logarithmic control of
columnwise maxima that will be used later. For instance, if
\((u_i)_{i\in[n]}\) is deterministic and \(\norm{u_i}\le1\), then, for every
\(k\ge1\),
\[
    \Norm{\max_{i\in[n]} |u_i^\top x_i|  }_{L^k}
    =
    O\big(\max(k,\log(en))\big).
\]

\begin{lemma}[Conditional quadratic-form bound]
\label{lem:conditional-quadratic-form-poincare}
Let
\[
    \Sigma_i:=\E[x_i x_i^\top]
\]
denote the second-moment matrix of \(x_i\). Under
Assumptions~\ref{ass:3bounded-means} and~\ref{ass:poincare-data}, for every
fixed \(k\ge1\), every deterministic symmetric matrix \(A\), and every
\(i\in[n]\),
\[
    \Norm{
        x_i^\top A x_i-\tr(\Sigma_iA)
    }_{L_i^k}
    \le
    O_k(\norm{A}_F).
\]
The same bound holds conditionally whenever \(A\) is independent of \(x_i\).
\end{lemma}

\begin{proof}
It is enough to prove the result for deterministic \(A\), since the conditional
version follows by freezing the variables on which \(A\) depends. Write
\(x_i=m_i+y_i\), where \(m_i=\E x_i\) and \(\E y_i=0\). Then
\[
    x_i^\top A x_i-\tr(\Sigma_iA)
    =
    \big(y_i^\top A y_i-\E[y_i^\top A y_i]\big)
    +
    2m_i^\top A y_i .
\]
The linear term satisfies
\[
    \Norm{m_i^\top A y_i}_{L^k}
    \le
    O_k(\norm{A m_i}_2)
    \le
    O_k(\norm{A}_{\op})
    \le
    O_k(\norm{A}_F),
\]
by Lemma~\ref{lem:linear-moments-poincare} and
Assumption~\ref{ass:3bounded-means}. For the centered quadratic term, we use
the standard \(L^k\)-form of the Poincar\'e inequality, obtained by applying
the Poincar\'e inequality to powers of the centered function:
\[
    \Norm{F-\E F}_{L^k}
    \le
    O_k\!\left(
        \Norm{\norm{\nabla F}_2}_{L^k}
    \right).
\]
Applying this to \(F(y)=y^\top A y\), whose gradient is \(2Ay\), gives
\[
    \Norm{
        y_i^\top A y_i-\E[y_i^\top A y_i]
    }_{L^k}
    \le
    O_k\!\left(
        \Norm{\norm{Ay_i}_2}_{L^k}
    \right).
\]
If \(A=\sum_a \sigma_a v_aw_a^\top\) is a singular-value decomposition, then
Minkowski's inequality and Lemma~\ref{lem:linear-moments-poincare} yield
\[
    \Norm{\norm{Ay_i}_2}_{L^k}
    \le
    \left(
        \sum_a \sigma_a^2
        \Norm{w_a^\top y_i}_{L^k}^2
    \right)^{1/2}
    \le
    O_k(\norm{A}_F).
\]
Combining the two estimates proves the lemma.
\end{proof}

Let us finally mention, as a direct consequence of~Lemma~\ref{lem:operator-norm-poincare-matrix} that will only be used once to set Lemma~\ref{lem:deterministic-gamma-concentration}:
\begin{lemma}\label{lem:op_norm_X}
    Under Assumptions~\ref{ass:1dim},\ref{ass:3bounded-means} and~\ref{ass:poincare-data}:
    \begin{align*}
        \Norm{|X|_{\op}}=O(\sqrt n).
    \end{align*}
\end{lemma}
\subsection{Control of the minimizer and its leave-one-out displacement}
\label{sec:minimizer_control}

We start with a deterministic stability estimate for strongly convex functions.

\begin{lemma}[Stability of minimizers]
\label{lem:stability-direct}
Let \(\phi:\R^p\to\R\) be differentiable and \(\kappa\)-strongly convex, and let
\[
    \mu_\phi:=\argmin_{x\in\R^p}\phi(x).
\]
Then, for every \(x\in\R^p\),
\[
    \norm{\mu_\phi-x}_2
    \le
    \frac1\kappa\norm{\nabla\phi(x)}_2.
\]
\end{lemma}

\begin{proof}
Strong convexity implies the \(\kappa\)-strong monotonicity of \(\nabla\phi\):
\[
    \ip{\nabla\phi(x)-\nabla\phi(y)}{x-y}
    \ge
    \kappa\norm{x-y}_2^2.
\]
Taking \(y=\mu_\phi\) and using \(\nabla\phi(\mu_\phi)=0\), we obtain
\[
    \kappa\norm{x-\mu_\phi}^2
    \le
    \ip{\nabla\phi(x)}{x-\mu_\phi}
    \le
    \norm{\nabla\phi(x)}_2\,\norm{x-\mu_\phi}_2.
\]
The claim follows by canceling one factor \(\norm{x-\mu_\phi}_2\), the case
\(x=\mu_\phi\) being immediate.
\end{proof}

\begin{lemma}[Empirical gradient at the origin]
\label{lem:gradient-zero}
Under Assumptions~\ref{ass:1dim}--\ref{ass:poincare-data}, for any
sequence \(k=(k_n)_{n\in\N}\),
\[
    \Norm{\norm{\nabla\Psi_X(0)}_2  }_{L^{k_n}}
    =
    O(k_n).
\]
\end{lemma}

\begin{proof}
We have
\[
    \nabla\Psi_X(0)
    =
    \frac{L'(0)}{n}\sum_{i=1}^n x_i
    +
    \nabla\rho(0).
\]
Thus
\[
    \Norm{\norm{\nabla\Psi_X(0)}_2  }_{L^{k_n}}
    \le
    \norm{\nabla\rho(0)}_2
    +
    |L'(0)|
    \left|
        \frac1n\sum_{i=1}^n \E x_i
    \right|_2
    +
    |L'(0)|
    \left\|
        \norm{\frac1n\sum_{i=1}^n(x_i-\E x_i)}_2
    \right\|_{L^{k_n}}.
\]
The first two terms are \(O(1)\) by Assumptions~\ref{ass:2origin} and
\ref{ass:3bounded-means}.

Set
\[
    S_n:=\frac1n\sum_{i=1}^n(x_i-\E x_i).
\]
For each deterministic \(u\in\R^p\) with \(\norm{u}\le1\), the map
\[
    X\mapsto u^\top S_n
\]
is a centered linear functional of \(X\) with Euclidean coefficient norm
at most \(n^{-1/2}\). By Lemma~\ref{lem:concentration_X} and
Lemma~\ref{lem:linear-moments-poincare},
\[
    \sup_{\norm{u}_2\le1}\Norm{u^\top S_n  }_{L^{k_n}}
    =
    O(k_nn^{-1/2}).
\]
Lemma~\ref{lem:column-euclidean-moments} and \(p=O(n)\) then give
\[
    \Norm{\norm{S_n}_2 }_{L^{k_n}}
    \le
    \sqrt p \sup_{\norm{u}_2\le1}\Norm{u^\top S_n  }_{L^{k_n}}
    =
    O(k_n).
\]
Since \(|L'(0)|=O(1)\), the claim follows.
\end{proof}

\begin{theorem}[Moment bound for the minimizer]
\label{thm:direct-hat-bound}
Under Assumptions~\ref{ass:1dim}--\ref{ass:5regularity}, for every sequence of
positive integers \(k_n\geq 1\)
\[
    \left\Vert |\hat\theta|_2\right\Vert_{L^{k_n}} + \sup_{i\in[n]}
    \Norm{\norm{\hat\theta_{-i}}  }_{L^{k_n}}
    =
    O(k_n).
\]
\end{theorem}

\begin{proof}
By Assumption~\ref{ass:5regularity}, the objective \(\Psi_X\) is
\(\kappa\)-strongly convex. Lemma~\ref{lem:stability-direct}, applied to
\(\phi=\Psi_X\) and \(x=0\), yields
\[
    \vert\hat\theta\vert_2
    \le
    \kappa^{-1}\norm{\nabla\Psi_X(0)}_2.
\]
Taking \(L^{k_n}\)-norms and applying Lemma~\ref{lem:gradient-zero} proves the
claim.
The bound for \(\hat\theta_{-i}\) is obtained in the same way, applied to
\[
    \nabla\Psi_{X_{-i}}(0)
    =
    \frac{L'(0)}{n}\sum_{j\ne i}x_j+\nabla\rho(0),
\]
which satisfies the same bound as \(\nabla\Psi_X(0)\), uniformly in \(i\).
\end{proof}

We next control the leave-one-out displacement
\[
    \delta_i\hat\theta
    :=
    \hat\theta-\hat\theta_{-i},
    \qquad i\in[n].
\]
We require two preliminary results. For \(i\in[n]\), define the test score
\[
    s_i^-:=L'(x_i^\top\hat\theta_{-i}).
\]
\begin{lemma}[Test score control]
\label{lem:test-score-control}
Under Assumptions~\ref{ass:1dim}--\ref{ass:5regularity}, for every sequence of
positive integers \(k_n\),
\[
    \sup_{i\in[n]}
    \Norm{s_i^-}_{L^{k_n}}
    =
    O(k_n^2).
\]
\end{lemma}

\begin{proof}
Conditionally on \(X_{-i}\), the vector \(\hat\theta_{-i}\) is deterministic
and independent of \(x_i\). Lemma~\ref{lem:moments_of_x_i} gives
\[
    \Norm{x_i^\top\hat\theta_{-i}}_{L_i^{k_n}}
    \le
    O(k_n)|\hat\theta_{-i}|_2 .
\]
Taking the \(L^{k_n}\)-norm over \(X_{-i}\) and using
Theorem~\ref{thm:direct-hat-bound}, we obtain
\[
    \Norm{x_i^\top\hat\theta_{-i}}_{L^{k_n}}
    =
    O(k_n^2).
\]
The claim then follows from Assumptions~\ref{ass:2origin} and~\ref{ass:5regularity} that provide $|L'(t)|
    \le
    |L'(0)|+\norm{L''}_\infty |t|
    \le
    O(1)(1+|t|)$.
\end{proof}
\begin{lemma}[Strong monotonicity with prediction curvature]
\label{lem:curvature_bound_t_hat_theta}
Under Assumption~\ref{ass:5regularity}, for every \(t\in\R^p\),
\[
    \frac{\kappa_L}{n}
    \norm{X^\top(\hat\theta-t)}_2^2
    +
    \kappa\norm{\hat\theta-t}_2^2
    \le
    -\nabla\Psi_X(t)^\top(\hat\theta-t).
\]
\end{lemma}

\begin{proof}
For all \(\theta,\theta'\in\R^p\),
\[
\begin{aligned}
    &\big\langle
        \nabla\Psi_X(\theta)-\nabla\Psi_X(\theta'),
        \theta-\theta'
    \big\rangle                                      \\
    &\qquad =
    \frac1n\sum_{j=1}^n
    \big(
        L'(x_j^\top\theta)-L'(x_j^\top\theta')
    \big)
    x_j^\top(\theta-\theta')
    +
    \big\langle
        \nabla\rho(\theta)-\nabla\rho(\theta'),
        \theta-\theta'
    \big\rangle .
\end{aligned}
\]
Since \(L''\ge\kappa_L\),
\[
    \big(
        L'(x_j^\top\theta)-L'(x_j^\top\theta')
    \big)
    x_j^\top(\theta-\theta')
    \ge
    \kappa_L
    \big(x_j^\top(\theta-\theta')\big)^2.
\]
Since \(\rho\) is \(\kappa\)-strongly convex,
\[
    \big\langle
        \nabla\rho(\theta)-\nabla\rho(\theta'),
        \theta-\theta'
    \big\rangle
    \ge
    \kappa\norm{\theta-\theta'}_2^2 .
\]
As a consequence, for all
\(\theta,\theta'\in\R^p\),
\[
    \big\langle
        \nabla\Psi_X(\theta)-\nabla\Psi_X(\theta'),
        \theta-\theta'
    \big\rangle
    \ge
    \frac{\kappa_L}{n}
    \norm{X^\top(\theta-\theta')}_2^2
    +
    \kappa\norm{\theta-\theta'}_2^2 .
\]
Taking \(\theta=\hat\theta\),
\(\theta'=t\), and using \(\nabla\Psi_X(\hat\theta)=0\), gives the result.
\end{proof}

\begin{lemma}[Prediction stability of the leave-one-out displacement]
\label{lem:loo-prediction-stability}
Under Assumptions~\ref{ass:1dim}--\ref{ass:5regularity}, for every sequence of
positive integers \(k_n\),
\[
    \sup_{i\in[n]}
    \Norm{\norm{\delta_i\hat\theta}_2}_{L^{k_n}}
    =
    O\!\left(\frac{k_n^2}{\sqrt n}\right),
\]
and
\[
    \sup_{i\in[n]}
    \Norm{\norm{X^\top\delta_i\hat\theta}_2}_{L^{k_n}}
    +
    \sup_{i\in[n]}
    \Norm{\norm{X_{-i}^\top\delta_i\hat\theta}_2}_{L^{k_n}}
    =
    O(k_n^2).
\]
\end{lemma}

\begin{proof}
Since
\[
    \nabla\Psi_X(\hat\theta_{-i})
    =
    \frac1n s_i^-x_i,
\]
Lemma~\ref{lem:curvature_bound_t_hat_theta}, applied with
\(t=\hat\theta_{-i}\), gives
\[
    \frac{\kappa_L}{n}
    \norm{X^\top\delta_i\hat\theta}_2^2
    +
    \kappa\norm{\delta_i\hat\theta}_2^2
    \le
    -\frac1n s_i^-\,x_i^\top\delta_i\hat\theta .
\]
In particular,
\[
    \frac{\kappa_L}{n}
    \norm{X^\top\delta_i\hat\theta}_2^2
    \le
    \frac1n |s_i^-|\,|x_i^\top\delta_i\hat\theta|
    \le
    \frac1n |s_i^-|\,
    \norm{X^\top\delta_i\hat\theta}_2,
\]
because \(x_i^\top\delta_i\hat\theta\) is one coordinate of
\(X^\top\delta_i\hat\theta\). If
\(\norm{X^\top\delta_i\hat\theta}_2=0\), the prediction bound is immediate.
Otherwise, dividing by \(\norm{X^\top\delta_i\hat\theta}_2\) yields
\[
    \norm{X^\top\delta_i\hat\theta}_2
    \le
    \kappa_L^{-1}|s_i^-|.
\]

Returning to the same inequality and dropping the nonnegative prediction term,
\[
    \kappa\norm{\delta_i\hat\theta}_2^2
    \le
    \frac1n |s_i^-|\,|x_i^\top\delta_i\hat\theta|
    \le
    \frac1n |s_i^-|\,
    \norm{X^\top\delta_i\hat\theta}_2
    \le
    \frac{|s_i^-|^2}{\kappa_L n}.
\]
Taking \(L^{k_n}\)-norms and using Lemma~\ref{lem:test-score-control} gives
the first two estimates. Finally,
\[
    \norm{X_{-i}^\top\delta_i\hat\theta}_2
    \le
    \norm{X^\top\delta_i\hat\theta}_2,
\]
which gives the last one.
\end{proof}

\subsection{Concentration of the minimizer}
\label{sec:variance}

We next derive the sensitivity estimates for \(\hat\theta\) with respect to the
columns of \(X\). These estimates are the input needed to apply the tensorized
Poincar\'e inequality to functions of \(\hat\theta\).

Fix \(i\in[n]\), freeze all columns except \(x_i\), and perturb \(x_i\) in the
direction \(h\in\R^p\):
\[
    x_i(t)=x_i+th,
    \qquad
    x_j(t)=x_j\quad(j\ne i),
    \qquad
    \theta(t)=\hat\theta(X(t)).
\]
The first-order condition is
\[
    0
    =
    \nabla\Psi_{X(t)}(\theta(t))
    =
    \frac1n\sum_{j=1}^n
    L'(x_j(t)^\top\theta(t))x_j(t)
    +
    \nabla\rho(\theta(t)).
\]
Since \(\Psi_X\) is strongly convex and \(C^2\), the implicit function theorem
allows us to differentiate this identity at \(t=0\). Define
\[
    G
    :=
    \nabla^2\Psi_X(\hat\theta)^{-1}.
\]
Writing \(\theta'(0)=\mathbb D_i[\hat\theta][h]\), we obtain
\[
    G^{-1}\mathbb D_i[\hat\theta][h]
    +
    \frac1n
    \left[
        L''(x_i^\top\hat\theta)(h^\top\hat\theta)x_i
        +
        L'(x_i^\top\hat\theta)h
    \right]
    =
    0.
\]
Equivalently,
\begin{align}\label{eq:derivative_hat_theta}
    \mathbb D_i[\hat\theta]
    =
    -\frac1n
    \left[
        L''(x_i^\top\hat\theta)Gx_i\hat\theta^\top
        +
        s_iG
    \right].
\end{align}
where we introduced the following notation for the train score:
\begin{align*}
    s_i
    :=
    L'(x_i^\top\hat\theta)
\end{align*}
\begin{lemma}[Uniform prediction moments]
\label{lem:uniform-prediction-moments}
Under Assumptions~\ref{ass:1dim}--\ref{ass:5regularity}, for every sequence of
positive integers \(k_n\),
\[
    \sup_{i\in[n]}
    \Norm{x_i^\top\hat\theta}_{L^{k_n}}
    =
    O(k_n^2),
    \qquad
    \sup_{i\in[n]}
    \Norm{s_i}_{L^{k_n}}
    =
    O(k_n^2).
\]
\end{lemma}
\begin{proof}
Fix \(i\in[n]\). Since \(\hat\theta_{-i}\) is independent of \(x_i\), conditioning
on \(\hat\theta_{-i}\), applying Lemma~\ref{lem:moments_of_x_i}, and then using
Theorem~\ref{thm:direct-hat-bound} give
\[
    \Norm{x_i^\top\hat\theta_{-i}  }_{L^{k_n}}
    =
    O(k_n^2).
\]
One can then deduce from Lemma~\ref{lem:loo-prediction-stability} that:
\begin{align*}
    \Norm{x_i^\top\hat\theta  }_{L^{k_n}}
    &\le
    \Norm{x_i^\top\hat\theta_{-i}  }_{L^{k_n}}
    +
    \Norm{x_i^\top(\hat\theta-\hat\theta_{-i}) }_{L^{k_n}}        \\
    &\le
    O(k_n^2)
    +
    \Norm{\norm{X^\top(\hat\theta-\hat\theta_{-i})}_2}_{L^{k_n}}
    =
    O(k_n^2),
\end{align*}
Finally, since
\[
    |L'(t)|\le O(1)(1+|t|),
\]
the same bound holds for \(s_i=L'(x_i^\top\hat\theta)\).
\end{proof}

\begin{lemma}[Differential bound for the minimizer]
\label{lem:differential-identity-minimizer}
Under Assumptions~\ref{ass:1dim}--\ref{ass:5regularity}, for every \(i\in[n]\),
every random vector \(z=z(X)\), and any
sequence \(k=(k_n)_{n\in\N}\),
\[
    \Norm{\norm{z^\top \mathbb D_i[\hat\theta] }^* }_{L^{k_n}}
    \le
    O \left( \frac{k_n^2}{n} \right)
    \left(
        \Norm{z^\top Gx_i }_{L^{2k_n}}
        +
        \Norm{\norm{z}_2 }_{L^{2k_n}}
    \right).
\]
\end{lemma}

\begin{proof}
Multiplying the differential identity~\eqref{eq:derivative_hat_theta} by \(z^\top\), we get, for every
\(h\in\R^p\),
\[
    z^\top \mathbb D_i[\hat\theta]\,[h]
    =
    -\frac1n h^\top
    \left[
        L''(x_i^\top\hat\theta)(x_i^\top Gz)\hat\theta
        +
        s_iGz
    \right].
\]
Since \(\norm{G}_{\op}\le\kappa^{-1}\) and \(\norm{L''}_\infty=O(1)\),
\[
    \norm{z^\top \mathbb D_i[\hat\theta]}^*
    \le
    \frac{O(1)}{n}
    \left(
        |x_i^\top Gz|\,\norm{\hat\theta}_2
        +
        |s_i|\,\norm{z}_2
    \right).
\]
Taking \(L^{k_n}\)-norms and applying H\"older's inequality, together with
Theorem~\ref{thm:direct-hat-bound} and
Lemma~\ref{lem:uniform-prediction-moments}, proves the claim.
\end{proof}

\begin{lemma}[Empirical prediction energy]
\label{lem:empirical-prediction-energy}
Under Assumptions~\ref{ass:1dim}--\ref{ass:5regularity}, for every sequence of
positive integers \(k_n\),
\[
    \left\|
        \frac1n\sum_{i=1}^n (x_i^\top\hat\theta)^2
    \right\|_{L^{k_n}}
    =
    O(k_n^2),
    \qquad
    \left\|
        \frac1n\sum_{i=1}^n s_i^2
    \right\|_{L^{k_n}}
    =
    O(k_n^2).
\]
\end{lemma}

\begin{proof}
Applying Lemma~\ref{lem:curvature_bound_t_hat_theta} with \(t=0\), we get
\[
    \frac{\kappa_L}{n}\sum_{i=1}^n (x_i^\top\hat\theta)^2
    +
    \kappa\norm{\hat\theta}_2^2
    \le
    -\nabla\Psi_X(0)^\top\hat\theta .
\]
By Young's inequality,
\[
    -\nabla\Psi_X(0)^\top\hat\theta
    \le
    \frac1{2\kappa}\norm{\nabla\Psi_X(0)}_2^2
    +
    \frac{\kappa}{2}\norm{\hat\theta}_2^2 .
\]
After cancellation,
\[
    \frac1n\sum_{i=1}^n (x_i^\top\hat\theta)^2
    \le
    O(1)\norm{\nabla\Psi_X(0)}_2^2 .
\]
Taking \(L^{k_n}\)-norms and using Lemma~\ref{lem:gradient-zero} at exponent
\(2k_n\), we obtain the first bound.

Since \(\norm{L''}_\infty=O(1)\) and \(|L'(0)|=O(1)\), $|L'(t)|^2\le O(1)(1+t^2)$, and therefore:
\[
    \frac1n\sum_{i=1}^n s_i^2
    \le
    O(1)\left(
        1+\frac1n\sum_{i=1}^n (x_i^\top\hat\theta)^2
    \right),
\]
which gives the second bound.
\end{proof}
\begin{lemma}[Square-summed background sensitivity]
\label{lem:square-summed-background-sensitivity}
Under Assumptions~\ref{ass:1dim}--\ref{ass:5regularity}, for every random vector
\(z=z(X)\) and any sequence of integers \(k_n\ge1\),
\[
    \left\|
        \sum_{i=1}^n
        \left( \norm{z^\top \mathbb D_i[\hat\theta]}^* \right)^2
    \right\|_{L^{k_n}}
    \le
    O \left( \frac{k_n^2}{n} \right)
    \Norm{\norm{z}_2}_{L^{4k_n}}^2 .
\]
\end{lemma}

\begin{proof}
The pointwise bound in the proof of
Lemma~\ref{lem:differential-identity-minimizer} gives
\[
   \left( \norm{z^\top \mathbb D_i[\hat\theta]}^* \right)^2
    \le
    \frac{O(1)}{n^2}
    \left[
        (x_i^\top Gz)^2 |\hat\theta|_2^2
        +
        s_i^2\norm{z}_2^2
    \right].
\]
By Assumption~\ref{ass:5regularity},
\[
    G^{-1}
    =
    \frac1n\sum_{i=1}^n L''(x_i^\top\hat\theta)x_ix_i^\top
    +
    \nabla^2\rho(\hat\theta)
    \succeq
    \frac{\kappa_L}{n}\sum_{i=1}^n x_ix_i^\top .
\]
Hence
\[
    G\left(\frac1n\sum_{i=1}^n x_ix_i^\top\right)G
    \preceq
    \kappa_L^{-1}G
    \preceq
    O(1)I_p,
\]
and therefore
\[
    \sum_{i=1}^n (x_i^\top Gz)^2
    \le
    O(n)\norm{z}_2^2.
\]
Summing the pointwise bound over \(i\) yields
\[
    \sum_{i=1}^n
    \left( \norm{z^\top \mathbb D_i[\hat\theta]}^* \right)^2
    \le
    \frac{O(1)}{n}
    \left[
        |\hat\theta|_2^2
        +
        \frac1n\sum_{i=1}^n s_i^2
    \right]
    \norm{z}^2.
\]
Taking \(L^{k_n}\)-norms and using H\"older's inequality gives
\[
\begin{aligned}
    &\left\|
        \sum_{i=1}^n
        \left( \norm{z^\top \mathbb D_i[\hat\theta]}^* \right)^2
    \right\|_{L^{k_n}}  
    \le
    \frac{O(1)}{n}
    \left\|
        \norm{\hat\theta}_2^2
        +
        \frac1n\sum_{i=1}^n s_i^2
    \right\|_{L^{2k_n}}
    \Norm{\norm{z}_2^2}_{L^{2k_n}} .
\end{aligned}
\]
By Theorem~\ref{thm:direct-hat-bound} and
Lemma~\ref{lem:empirical-prediction-energy}, applied at exponent \(2k_n\),
the first factor is \(O(k_n^2)\). Since
\[
    \Norm{\norm{z}_2^2}_{L^{2k_n}}
    =
    \Norm{\norm{z}_2}_{L^{4k_n}}^2,
\]
the claim follows.
\end{proof}

\begin{corollary}[Variance bounds for smooth functions of the minimizer]
\label{cor:variance-linear-poincare}
Under Assumptions~\ref{ass:1dim}--\ref{ass:5regularity}, for every sufficiently
smooth \(f:\mathbb R^p\to\mathbb R\),
\[
    \Var(f(\hat\theta))
    \le
    O\left(
        \frac1n
        \Norm{\norm{\nabla f(\hat\theta)}_2}_{L^4}^2
    \right).
\]
\end{corollary}


\begin{proof}
By the chain rule,
\[
    \norm{\mathbb D_i[f(\hat\theta)]}_2^*
    =
    \norm{\nabla f(\hat\theta)^\top\mathbb D_i[\hat\theta]}^* .
\]
Applying Lemma~\ref{lem:concentration_X} and then
Lemma~\ref{lem:square-summed-background-sensitivity} with
\(z=\nabla f(\hat\theta)\) and \(k_n=1\) gives
\[
\begin{aligned}
    \Var(f(\hat\theta))
    &\le
    C_P\,\E\sum_{i=1}^n
    \left(
        \norm{
            \nabla f(\hat\theta)^\top\mathbb D_i[\hat\theta]
        }^*
    \right)^2 
    \le
    O\left(
        \frac1n
        \Norm{\norm{\nabla f(\hat\theta)}_2}_{L^4}^2
    \right).
\end{aligned}
\]
\end{proof}

\begin{remark}
\label{rem:sigma-upper}
Corollary~\ref{cor:variance-linear-poincare} gives, in the setting of
Theorem~\ref{thm:main-clt},
\[
    \sigma_n^2
    =
    n\,\Var(u^\top\hat\theta)
    =
    O(1).
\]
Thus the additional variance assumption in Theorem~\ref{thm:main-clt} is only
the lower bound \(\inf_n\sigma_n^2>0\).
\end{remark}

Until now, we were able to show in Lemma~\ref{lem:loo-prediction-stability} that $\||\delta_i\hat \theta|_2\|_{L^k}\leq O(1/\sqrt n)$ however that do not allow us to reach already the bound $\|\sqrt n u^\top \delta_i\hat \theta|\|_{L^k}\leq O(1/\sqrt n)$ that is needed to be able to use the Wasserstein bound of theorem~\ref{the:conditional-loo-perturbative}. This estimate will be though provided through a precise estimation of the leave-one-out displacement in next section

\section{Leave-one-out perturbation expansion}
\label{sec:loo}

The goal of this section is to obtain quantitative expansions of the
leave-one-out displacement
\[
    \delta_i\hat\theta
    :=
    \hat\theta-\hat\theta_{-i},
    \qquad i\in[n].
\]
The prediction stability estimates proved below give direct control of
\(X^\top\delta_i\hat\theta\), while the Wasserstein bound requires a more
precise description of deterministic projections \(u^\top\delta_i\hat\theta\).
We first derive a first-order expansion, then introduce an averaged-leverage
score which removes the leading dependence of the score on the random leverage
\(n^{-1}x_i^\top G_{-i}x_i\). Finally, under the third-order curvature
assumptions and the condition \(\nabla^3\rho\equiv0\), we derive a projected
second-order expansion based on this averaged score.

For \(i\in[n]\), write
\[
    \Psi_X^{-i}(\theta)
    :=
    \frac1n\sum_{\substack{j\in[n]\\j\ne i}}
    L(x_j^\top\theta)+\rho(\theta),
    \qquad
    \hat\theta_{-i}
    :=
    \argmin_{\theta\in\R^p}\Psi_X^{-i}(\theta),
\]
and
\[
    G_{-i}
    :=
    \left(\nabla^2\Psi_X^{-i}(\hat\theta_{-i})\right)^{-1},
    \qquad
    s_i
    :=
    L'(x_i^\top\hat\theta).
\]
By Assumption~\ref{ass:5regularity}, \(G_{-i}\) is well-defined and
\[
    \norm{G_{-i}}_{\op}\le \kappa^{-1}.
\]

\subsection{First-order expansion}
\label{sse:first_order}

\begin{lemma}[Exact leave-one-out identity]
\label{lem:loo-exact}
Under Assumption~\ref{ass:5regularity}, for every \(i\in[n]\),
\[
    \delta_i\hat\theta
    =
    -\frac{s_i}{n}\bar G_i x_i,
    \qquad
    \bar G_i
    :=
    \left[
        \int_0^1
        \nabla^2\Psi_X^{-i}
        \big(
            \hat\theta_{-i}+t\delta_i\hat\theta
        \big)\,dt
    \right]^{-1}.
\]
\end{lemma}

\begin{proof}
Since \(\nabla\Psi_X(\hat\theta)=0\) and
\[
    \Psi_X(\theta)
    =
    \Psi_X^{-i}(\theta)+\frac1nL(x_i^\top\theta),
\]
we have
\[
    \nabla\Psi_X^{-i}(\hat\theta)
    =
    -\frac1nL'(x_i^\top\hat\theta)x_i
    =
    -\frac{s_i}{n}x_i.
\]
Moreover, \(\nabla\Psi_X^{-i}(\hat\theta_{-i})=0\). Hence
\[
    -\frac{s_i}{n}x_i
    =
    \nabla\Psi_X^{-i}(\hat\theta)
    -
    \nabla\Psi_X^{-i}(\hat\theta_{-i})
    =
    \left[
        \int_0^1
        \nabla^2\Psi_X^{-i}
        \big(
            \hat\theta_{-i}+t\delta_i\hat\theta
        \big)\,dt
    \right]\delta_i\hat\theta .
\]
The integrated Hessian is invertible by Assumption~\ref{ass:5regularity}, and
the claim follows.
\end{proof}

Write
\[
    \bar G_i^{-1}
    =
    G_{-i}^{-1}+R_i,
    \qquad
    R_i
    =
    R_i^\rho+R_i^L,
\]
where
\begin{align}
    R_i^\rho
    &:=
    \int_0^1
    \left[
        \nabla^2\rho(\hat\theta_{-i}+t\delta_i\hat\theta)
        -
        \nabla^2\rho(\hat\theta_{-i})
    \right]\,dt,
    \label{eq:def_Rrho}
\end{align}
and
\begin{align}
    R_i^L
    &:=
    \frac1n X_{-i}\Gamma_iX_{-i}^\top .
    \label{eq:def_RL}
\end{align}
Here \(\Gamma_i\) is the diagonal matrix, indexed by \(j\ne i\), with entries
\[
    (\Gamma_i)_j
    :=
    \int_0^1
    \left[
        L''\!\left(
            x_j^\top(\hat\theta_{-i}+t\delta_i\hat\theta)
        \right)
        -
        L''(x_j^\top\hat\theta_{-i})
    \right]\,dt .
\]
The resolvent identity gives
\begin{align}
    G_{-i}-\bar G_i
    =
    \bar G_iR_iG_{-i}.
    \label{eq:resolvent-bar-G}
\end{align}

\begin{lemma}[Curvature bounds for the resolvents]
\label{lem:curvature-resolvent}
Under Assumption~\ref{ass:5regularity}, for every \(i\in[n]\) and every
realization,
\[
    \norm{\bar G_i}_{\op}+\norm{G_{-i}}_{\op}=O(1),
\]
\[
    \norm{\bar G_iX_{-i}}_{\op}
    +
    \norm{G_{-i}X_{-i}}_{\op}
    =
    O(\sqrt n),
\]
and
\[
    \norm{X_{-i}^\top\bar G_iX_{-i}}_{\op}
    +
    \norm{X_{-i}^\top G_{-i}X_{-i}}_{\op}
    =
    O(n).
\]
\end{lemma}

\begin{proof}
For every \(\theta\),
\[
    \nabla^2\Psi_X^{-i}(\theta)
    \succeq
    \kappa I_p+\frac{\kappa_L}{n}X_{-i}X_{-i}^\top .
\]
The same lower bound holds after averaging in the definition of
\(\bar G_i^{-1}\). Thus \(G_{-i}^{-1}\) and \(\bar G_i^{-1}\) are both bounded
below by the right-hand side. This first gives
\[
    \norm{G_{-i}}_{\op}+\norm{\bar G_i}_{\op}=O(1).
\]
If \(M\in\{G_{-i},\bar G_i\}\), then
\[
    M^{-1}
    \succeq
    \frac{\kappa_L}{n}X_{-i}X_{-i}^\top.
\]
Therefore
\[
    M^{1/2}X_{-i}X_{-i}^\top M^{1/2}
    \preceq
    \frac{n}{\kappa_L}I_p.
\]
It follows that
\[
    \norm{MX_{-i}}_{\op}^2
    =
    \norm{MX_{-i}X_{-i}^\top M}_{\op}
    \le
    \frac{n}{\kappa_L}\norm{M}_{\op}
    =
    O(n),
\]
and
\[
    \norm{X_{-i}^\top MX_{-i}}_{\op}
    =
    \norm{M^{1/2}X_{-i}X_{-i}^\top M^{1/2}}_{\op}
    \le
    \frac{n}{\kappa_L}.
\]
This proves the claim.
\end{proof}

\begin{lemma}[Score comparison and score moments]
\label{lem:score-moments}
Under Assumptions~\ref{ass:1dim}--\ref{ass:5regularity}, define
\[
    s_i^-:=L'(x_i^\top\hat\theta_{-i}).
\]
Then, for every \(i\in[n]\),
\[
    |s_i|\le |s_i^-|.
\]
Moreover, for every fixed \(k\ge1\),
\[
    \sup_{i\in[n]}\Norm{s_i}_{L^k}
    +
    \sup_{i\in[n]}\Norm{s_i^-}_{L^k}
    =
    O_k(1).
\]
\end{lemma}

\begin{proof}
Let
\[
    \alpha_i
    :=
    \int_0^1
    L''\!\left(
        x_i^\top(\hat\theta_{-i}+t\delta_i\hat\theta)
    \right)\,dt .
\]
Then
\[
    s_i
    =
    s_i^-+\alpha_i x_i^\top\delta_i\hat\theta .
\]
By Lemma~\ref{lem:loo-exact},
\[
    x_i^\top\delta_i\hat\theta
    =
    -\frac{s_i}{n}x_i^\top\bar G_i x_i.
\]
Thus
\[
    s_i
    \left(
        1+\frac{\alpha_i}{n}x_i^\top\bar G_i x_i
    \right)
    =
    s_i^-.
\]
Since \(\alpha_i\ge0\) and \(\bar G_i\succeq0\), we obtain
\[
    |s_i|\le |s_i^-|.
\]

It remains to bound \(s_i^-\). By Assumptions~\ref{ass:2origin} and
\ref{ass:5regularity},
\[
    |L'(t)|
    \le
    |L'(0)|+\norm{L''}_\infty |t|
    \le
    O(1)(1+|t|).
\]
Conditionally on \(X_{-i}\), the vector \(\hat\theta_{-i}\) is deterministic
and independent of \(x_i\). Hence Lemma~\ref{lem:moments_of_x_i} gives
\[
    \Norm{x_i^\top\hat\theta_{-i}}_{L_i^k}
    \le
    O_k(1)\norm{\hat\theta_{-i}}_2 .
\]
Taking the \(L^k\)-norm over \(X_{-i}\) and using
Theorem~\ref{thm:direct-hat-bound}, we get
\[
    \sup_{i\in[n]}
    \Norm{x_i^\top\hat\theta_{-i}}_{L^k}
    =
    O_k(1).
\]
Therefore
\[
    \sup_{i\in[n]}\Norm{s_i^-}_{L^k}=O_k(1),
\]
and the comparison \(|s_i|\le |s_i^-|\) gives the same bound for \(s_i\).
\end{proof}

\begin{lemma}[Cross-leverage bounds]
\label{lem:cross-leverage-poincare}
Under Assumptions~\ref{ass:1dim}--\ref{ass:5regularity}, for every fixed
\(k\ge1\),
\[
    \sup_{i\in[n]}
    \Norm{
        \norm{X_{-i}^\top G_{-i}x_i}_\infty
    }_{L^k}
    =
    O_k(\sqrt n\,\log n),
\]
and
\[
    \sup_{i\in[n]}
    \Norm{
        \norm{X_{-i}^\top G_{-i}x_i}_2
    }_{L^k}
    =
    O_k(n).
\]
\end{lemma}

\begin{proof}
Condition on \(X_{-i}\). Then \(G_{-i}\) and \((x_j)_{j\ne i}\) are fixed and
independent of \(x_i\). For every \(j\ne i\),
Lemma~\ref{lem:curvature-resolvent} gives
\[
    \norm{G_{-i}x_j}_2
    \le
    \norm{G_{-i}X_{-i}}_{\op}
    =
    O(\sqrt n).
\]
Thus, conditionally on \(X_{-i}\), the linear moment bound under the Poincar\'e
assumption gives, for every integer \(q\ge1\),
\[
    \max_{j\ne i}
    \Norm{x_i^\top G_{-i}x_j}_{L_i^q}
    \le
    O(q\sqrt n).
\]
Applying Lemma~\ref{lem:max-polynomial-moment-growth} conditionally on \(X_{-i}\) gives
\[
    \Norm{
        \max_{j\ne i}|x_i^\top G_{-i}x_j|
    }_{L_i^k}
    \le
    O_k(\sqrt n\,\log n).
\]
Taking the \(L^k\)-norm over \(X_{-i}\) proves the first estimate.

For the second estimate, Lemma~\ref{lem:curvature-resolvent} gives
\[
    \norm{X_{-i}^\top G_{-i}x_i}_2
    \le
    \norm{G_{-i}X_{-i}}_{\op}\norm{x_i}_2
    =
    O(\sqrt n)\norm{x_i}_2.
\]
The claim follows from Lemma~\ref{lem:moments_of_x_i}.
\end{proof}

\begin{lemma}[Frobenius bound for the loss-curvature drift]
\label{lem:gamma-frobenius-bound}
Under Assumptions~\ref{ass:1dim}--\ref{ass:6smooth-curvature}, for every fixed
\(k\ge1\),
\[
    \sup_{i\in[n]}
    \Norm{\norm{\Gamma_i}_F}_{L^k}
    =
    O_k(1).
\]
\end{lemma}

\begin{proof}
By Assumption~\ref{ass:6smooth-curvature}, for \(j\ne i\),
\[
    |(\Gamma_i)_j|
    \le
    O(1)|x_j^\top\delta_i\hat\theta|.
\]
Therefore
\[
    \norm{\Gamma_i}_F
    \le
    O(1)\norm{X_{-i}^\top\delta_i\hat\theta}_2.
\]
The result follows from Lemma~\ref{lem:loo-prediction-stability}.
\end{proof}

\begin{remark}\label{rem:why_frob_norm}
At this stage we only need a Frobenius bound on \(\Gamma_i\). The sharper
entrywise estimate
\[
    \sup_{i\in[n]}
    \Norm{\norm{\Gamma_i}_{\op}}_{L^k}
    =
    O_k\!\left(\frac{\log n}{\sqrt n}\right)
\]
will follow later from the projected first-order estimate (see Lemma~\ref{lem:second-order-curvature-remainders-xi}). The Frobenius bound
is sufficient for the first-order expansion because it is used through
\[
    \norm{\Gamma_i v}_2
    \le
    \norm{\Gamma_i}_F\norm{v}_\infty ,
    \qquad v\in\R^{n-1}.
\]
\end{remark}

\begin{lemma}[Regularizer drift term]
\label{lem:rho-drift}
Under Assumptions~\ref{ass:1dim}--\ref{ass:6smooth-curvature}, for every fixed
\(k\ge1\),
\[
    \sup_{i\in[n]}
    \Norm{
        \norm{
            \frac{s_i}{n}\bar G_iR_i^\rho G_{-i}x_i
        }_2
    }_{L^k}
    =
    O_k\!\left(\frac1n\right).
\]
\end{lemma}

\begin{proof}
By Assumption~\ref{ass:6smooth-curvature} and
Lemma~\ref{lem:loo-prediction-stability},
\begin{align}
    \norm{R_i^\rho}_{\op}
    \le
    O(1)\norm{\delta_i\hat\theta}_2
    \le
    O(1)\frac{|s_i^-|}{\sqrt n}.
    \label{eq:bound_R_rho}
\end{align}
Consequently,
\[
    \norm{
        \frac{s_i}{n}\bar G_iR_i^\rho G_{-i}x_i
    }_2
    \le
    \frac{O(1)}{n\sqrt n}
    |s_i|\,|s_i^-|\,\norm{x_i}_2.
\]
The claim follows from Lemmas~\ref{lem:score-moments} and
\ref{lem:moments_of_x_i}, together with H\"older's inequality.
\end{proof}

\begin{lemma}[Loss drift term]
\label{lem:loss-drift}
Under Assumptions~\ref{ass:1dim}--\ref{ass:6smooth-curvature}, for every fixed
\(k\ge1\),
\[
    \sup_{i\in[n]}
    \Norm{
        \norm{
            \frac{s_i}{n}\bar G_iR_i^LG_{-i}x_i
        }_2
    }_{L^k}
    =
    O_k\!\left(\frac{\log n}{n}\right).
\]
\end{lemma}

\begin{proof}
Since \(R_i^L=n^{-1}X_{-i}\Gamma_iX_{-i}^\top\),
\[
    \norm{
        \frac{s_i}{n}\bar G_iR_i^LG_{-i}x_i
    }_2
    \le
    \frac{|s_i|}{n^2}
    \norm{\bar G_iX_{-i}}_{\op}\,
    \norm{\Gamma_i}_F\,
    \norm{X_{-i}^\top G_{-i}x_i}_\infty .
\]
By Lemmas~\ref{lem:score-moments},
\ref{lem:curvature-resolvent},
\ref{lem:gamma-frobenius-bound}, and
\ref{lem:cross-leverage-poincare}, H\"older's inequality gives
\[
    \sup_{i\in[n]}
    \Norm{
        \norm{
            \frac{s_i}{n}\bar G_iR_i^LG_{-i}x_i
        }_2
    }_{L^k}
    =
    O_k\!\left(\frac{\log n}{n}\right).
\]
\end{proof}

Define the first-order leave-one-out displacement based on the train score by
\[
    \delta_{i,\mathrm{tr}}^{(0)}\hat\theta
    :=
    -\frac{s_i}{n}G_{-i}x_i.
\]

\begin{theorem}[First-order vector leave-one-out expansion]
\label{thm:loo-remainder-tilde}
Under Assumptions~\ref{ass:1dim}--\ref{ass:6smooth-curvature}, for every fixed
\(k\ge1\),
\[
    \sup_{i\in[n]}
    \Norm{
        \norm{
            \delta_i\hat\theta-\delta_{i,\mathrm{tr}}^{(0)}\hat\theta
        }_2
    }_{L^k}
    =
    O_k\!\left(\frac{\log n}{n}\right),
\]
and
\[
    \sup_{i\in[n]}
    \Norm{
        \norm{
            X_{-i}^\top
            (\delta_i\hat\theta-\delta_{i,\mathrm{tr}}^{(0)}\hat\theta)
        }_2
    }_{L^k}
    =
    O_k\!\left(\frac{\log n}{\sqrt n}\right).
\]
\end{theorem}

\begin{proof}
Using Lemma~\ref{lem:loo-exact} and \eqref{eq:resolvent-bar-G},
\[
    \delta_i\hat\theta-\delta_{i,\mathrm{tr}}^{(0)}\hat\theta
    =
    \frac{s_i}{n}\bar G_iR_iG_{-i}x_i
    =
    \frac{s_i}{n}\bar G_iR_i^\rho G_{-i}x_i
    +
    \frac{s_i}{n}\bar G_iR_i^LG_{-i}x_i.
\]
The first estimate follows from Lemmas~\ref{lem:rho-drift} and
\ref{lem:loss-drift}.

For the projected estimate, multiply the preceding identity by
\(X_{-i}^\top\). For the regularizer contribution, using
\eqref{eq:bound_R_rho}, Lemma~\ref{lem:curvature-resolvent}, and
\(\norm{G_{-i}x_i}_2\le O(1)\norm{x_i}_2\), we get
\[
    \Norm{
        \norm{
            X_{-i}^\top
            \frac{s_i}{n}\bar G_iR_i^\rho G_{-i}x_i
        }_2
    }_{L^k}
    =
    O_k\!\left(\frac1{\sqrt n}\right).
\]
For the loss contribution,
\[
    \norm{
        X_{-i}^\top
        \frac{s_i}{n}\bar G_iR_i^LG_{-i}x_i
    }_2
    \le
    \frac{|s_i|}{n^2}
    \norm{X_{-i}^\top\bar G_iX_{-i}}_{\op}\,
    \norm{\Gamma_i}_F\,
    \norm{X_{-i}^\top G_{-i}x_i}_\infty .
\]
Using Lemmas~\ref{lem:score-moments},
\ref{lem:curvature-resolvent},
\ref{lem:gamma-frobenius-bound}, and
\ref{lem:cross-leverage-poincare}, we obtain
\[
    \sup_{i\in[n]}
    \Norm{
        \norm{
            X_{-i}^\top
            \frac{s_i}{n}\bar G_iR_i^LG_{-i}x_i
        }_2
    }_{L^k}
    =
    O_k\!\left(\frac{\log n}{\sqrt n}\right).
\]
Combining the two contributions proves the projected estimate.
\end{proof}

Recall that the CLT will be proved using the following simplification of
Theorem~\ref{the:conditional-loo-perturbative}:
\begin{align}\label{eq:loo_bound_wasserstein_bound}
    d_W \left( \frac{f_n(X)- \mathbb E[f_n(X)]}{\sigma_n}, \mathcal N \right)
    \leq C \left( \frac{n}{\sigma_n^2} \omega_c + \frac{n}{\sigma_n^3}\omega_\Delta^3 \right).
\end{align}
Recall the notation for the leave-one-out increment of
\(f_n(X)=\sqrt n\,u^\top \hat \theta\):
\[
    \delta_i f_n(X)
    =
    f_n(X)-f_n^{-i}(X_{-i})
    =
    \sqrt n\,u^\top\delta_i\hat\theta(X).
\]
One can bound the intrinsic increment term appearing
in~\eqref{eq:loo_bound_wasserstein_bound}:
\begin{align}\label{eq:borne_omega_delta}
    \omega_\Delta
    =
    \sup_{i\in[n]}
    \Norm{\delta_i f_n(X)}_{L^3}
    =
    O\!\left(\frac{\log n}{\sqrt n}\right).
\end{align}
This follows from Theorem~\ref{thm:loo-remainder-tilde} and the following
lemma.

\begin{lemma}\label{lem:borne_delta_tr_0}
Under Assumptions~\ref{ass:1dim}--\ref{ass:6smooth-curvature},
\[
    \sup_{i\in[n]}
    \Norm{
        u^\top\delta_{i,\mathrm{tr}}^{(0)}\hat\theta
    }_{L^k}
    =
    O_k\!\left(\frac1n\right).
\]
\end{lemma}

\begin{proof}
For the first scalar bound,
\[
    |u^\top\delta_{i,\mathrm{tr}}^{(0)}\hat\theta|
    \le
    \frac{|s_i|}{n}|u^\top G_{-i}x_i|.
\]
Conditionally on \(X_{-i}\), the vector \(G_{-i}u\) is deterministic and has
bounded Euclidean norm. Hence Lemma~\ref{lem:moments_of_x_i} gives
\[
    \Norm{u^\top G_{-i}x_i}_{L^k}
    =
    O_k(1).
\]
Together with Lemma~\ref{lem:score-moments} and H\"older's inequality, this
proves the result.
\end{proof}

\begin{remark}[Necessity of a second-order expansion of the leave-one-out displacement]
\label{rem:csqce_wasserstein_bound_one_order_leaveone_out_expension}

The bound
\begin{align}\label{eq:bound_delata_i_1st_expansion}
    \sup_{i\in[n]}
    \Norm{
        \norm{
            \delta_i\hat\theta-\delta_{i,\mathrm{tr}}^{(0)}\hat\theta
        }_2
    }_{L^k}
    =
    O_k\!\left(\frac{\log n}{n}\right),
\end{align}
given by Theorem~\ref{thm:loo-remainder-tilde}, is not sufficient to control
the intrinsic covariance quantity appearing in~\eqref{eq:loo_bound_wasserstein_bound}:
\[
    \omega_c
    =
    n\sup_{j\in[n]}
    \sqrt{
        \Var\!\left(
            \Cov_j\!\left(
                u^\top\delta_j\hat\theta(X),
                u^\top\delta_j\hat\theta(X^{A_j})
            \right)
        \right)
    }.
\]
A direct but sharp computation, which we omit for conciseness, only gives from
\eqref{eq:bound_delata_i_1st_expansion} the bound
\begin{align*}
    & n^2\left\vert \Var\!\left(
            \Cov_j\!\left(
                u^\top\delta_j\hat\theta(X),
                u^\top\delta_j\hat\theta(X^{A_j})
            \right)
        \right)
        -
        \Var\!\left(
            \Cov_j\!\left(
                u^\top\delta_{j,\mathrm{tr}}^{(0)}\hat\theta(X),
                u^\top\delta_{j,\mathrm{tr}}^{(0)}\hat\theta(X^{A_j})
            \right)
        \right) \right\vert  \\
    &\qquad
    =
    O\!\left(\frac{(\log n)^4}{n^2}\right).
\end{align*}
This is too large for the target bound \(\omega_c=o(1/n)\), as can be checked
from~\eqref{eq:loo_bound_wasserstein_bound}. This motivates the second-order
expansions below. Before doing so, and to better disentangle the dependence of
\(s_i\) on \(x_i\), we introduce an approximation of \(s_i\) which can be
expressed as a Lipschitz functional of \(x_i^\top \hat \theta_{-i}\), namely
\(s_i^\zeta\). This will be crucial for efficiently bounding \(\omega_c\).
More specifically, we need to work with \(s_i^\zeta\) in order to apply
Lemma~\ref{lem:differentiated-active-score-contractions} with sufficiently
sharp bounds.
\end{remark}

\subsection{Averaged-leverage score}
\label{subsec:averaged-leverage-score}

We now replace the train score \(s_i=L'(x_i^\top\hat\theta)\) by a
leave-one-out score depending on \(x_i\) only through
\(x_i^\top\hat\theta_{-i}\). Let us introduce the averaged leverage:
\[
    \gamma_i
    :=
    \E\!\left[
        \frac1n x_i^\top G_{-i}x_i
    \right].
\]
Since \(x_i\) is independent of \(G_{-i}\), this can also be written as
\[
    \gamma_i
    =
    \E\!\left[
        \frac1n\tr(\Sigma_iG_{-i})
    \right],
    \qquad
    \Sigma_i:=\E[x_ix_i^\top].
\]
By Assumptions~\ref{ass:poincare-data} and~\ref{ass:3bounded-means},
\[
    \norm{\Sigma_i}_{\op}=O(1),
\]
and since \(\norm{G_{-i}}_{\op}=O(1)\) and \(p=O(n)\),
\[
    0\le \gamma_i=O(1).
\]

The following lemma allows to replace the
conditional averaged leverage by its full expectation.
\begin{lemma}[Concentration of the averaged leverage]
\label{lem:deterministic-gamma-concentration}
Under Assumptions~\ref{ass:1dim}--\ref{ass:7third-order-curvature}, for every
fixed \(k\ge1\),
\[
    \sup_{i\in[n]}
    \Norm{
        \frac1n x_i^\top G_{-i}x_i-\gamma_i
    }_{L^k}
    =
    O_k\!\left(\frac1{\sqrt n}\right).
\]
\end{lemma}

\begin{proof}
We decompose
\[
\begin{aligned}
    \frac1n x_i^\top G_{-i}x_i-\gamma_i
    &=
    \frac1n
    \left(
        x_i^\top G_{-i}x_i-\tr(\Sigma_iG_{-i})
    \right)
    +
    \left(
        \frac1n\tr(\Sigma_iG_{-i})
        -
        \E\!\left[\frac1n\tr(\Sigma_iG_{-i})\right]
    \right).
\end{aligned}
\]

We first treat the conditional quadratic fluctuation. Conditionally on
\(X_{-i}\), the matrix \(G_{-i}\) is deterministic and independent of \(x_i\).
Thus Lemma~\ref{lem:conditional-quadratic-form-poincare} gives
\[
    \sup_{i\in[n]}
    \left\|
        \frac1n
        \left(
            x_i^\top G_{-i}x_i-\tr(\Sigma_iG_{-i})
        \right)
    \right\|_{L^k}
        \le \sup_{i\in[n]}
    O_k\!\left(\frac{\norm{G_{-i}}_F}{n}\right)
    =
    O_k\!\left(\frac1{\sqrt n}\right),
\]
since \(\sup_{i\in[n]}\norm{G_{-i}}_{\op}=O(1)\) and \(p=O(n)\).

It remains to control
\[
    \frac1n\tr(\Sigma_iG_{-i})
    -
    \E\!\left[\frac1n\tr(\Sigma_iG_{-i})\right].
\]
By Lemma~\ref{lem:tensorized-Lk-poincare}, it is enough to prove
\[
    \sup_{i\in[n]}
    \left\|
        \left(
            \sum_{j\ne i}
            \left(
                \norm{
                    \mathbb D_j\!\left[
                        \frac1n\tr(\Sigma_iG_{-i})
                    \right]
                }_2^*
            \right)^2
        \right)^{1/2}
    \right\|_{L^k}
    =
    O_k\!\left(\frac1{\sqrt n}\right).
\]

Fix \(j\ne i\). By the resolvent differential identity,
\[
\begin{aligned}
    \mathbb D_j\!\left[
        \frac1n\tr(\Sigma_iG_{-i})
    \right][h]
    &=
    \frac1n\tr\!\left(
        \Sigma_i\mathbb D_j[G_{-i}][h]
    \right) 
    =
    -\frac1n
    \tr\!\left(
        G_{-i}\Sigma_iG_{-i}\,
        \mathbb D_j[G_{-i}^{-1}][h]
    \right).
\end{aligned}
\]

We now write explicitly the derivative of \(G_{-i}^{-1}\). Since
\[
    G_{-i}^{-1}
    =
    \frac1n\sum_{\ell\ne i}
    L''(x_\ell^\top\hat\theta_{-i})x_\ell x_\ell^\top
    +
    \nabla^2\rho(\hat\theta_{-i}),
\]
we decompose
\[
    \mathbb D_j[G_{-i}^{-1}][h]
    =
    \mathcal L_{ij}[h]+\mathcal B_{ij}[h],
\]
where the local part, coming from the direct differentiation of the column
\(x_j\), is
\[
\begin{aligned}
    \mathcal L_{ij}[h]
    &:=
    \frac1n
    L''(x_j^\top\hat\theta_{-i})
    \left(
        h x_j^\top+x_jh^\top
    \right) 
    +
    \frac1n
    L'''(x_j^\top\hat\theta_{-i})
    (h^\top\hat\theta_{-i})
    x_jx_j^\top,
\end{aligned}
\]
and the background part, coming from the variation of
\(\hat\theta_{-i}\), is
\[
\begin{aligned}
    \mathcal B_{ij}[h]
    &:=
    \frac1n
    \sum_{\ell\ne i}
    L'''(x_\ell^\top\hat\theta_{-i})
    \left(
        x_\ell^\top\mathbb D_j[\hat\theta_{-i}][h]
    \right)
    x_\ell x_\ell^\top
    +
    \nabla^3\rho(\hat\theta_{-i})
    \big[
        \mathbb D_j[\hat\theta_{-i}][h]
    \big].
\end{aligned}
\]
Thus the derivative of the trace splits into a local and a background
contribution.

We first bound the local contribution. Using
\(\norm{L''}_\infty=O(1)\), \(\norm{L'''}_\infty=O(1)\),
\(\norm{G_{-i}\Sigma_iG_{-i}}_{\op}=O(1)\), and
\(\norm{h}_2\le1\), we have
\[
\begin{aligned}
    &\left|
        \frac1n
        \tr\!\left(
            G_{-i}\Sigma_iG_{-i}\mathcal L_{ij}[h]
        \right)
    \right|
    \le
    \frac{O(1)}{n^2}
    \left(
        \norm{G_{-i}\Sigma_iG_{-i}x_j}_2
        +
        \norm{\hat\theta_{-i}}_2
        \left|
            x_j^\top G_{-i}\Sigma_iG_{-i}x_j
        \right|
    \right).
\end{aligned}
\]
Consequently,
\[
\begin{aligned}
    &\left\|
        \left(
            \sum_{j\ne i}
            \left[
                \sup_{\norm{h}_2\le1}
                \left|
                    \frac1n
                    \tr\!\left(
                        G_{-i}\Sigma_iG_{-i}\mathcal L_{ij}[h]
                    \right)
                \right|
            \right]^2
        \right)^{1/2}
    \right\|_{L^k}                                           \\
    &\qquad\le
    \frac{O(1)}{n^2}
    \left\|
        \left(
            \sum_{j\ne i}
            \norm{G_{-i}\Sigma_iG_{-i}x_j}_2^2
        \right)^{1/2}
    \right\|_{L^k}
    +
    \frac{O(1)}{n^2}
    \left\|
        \norm{\hat\theta_{-i}}_2
        \left(
            \sum_{j\ne i}
            \left|
                x_j^\top G_{-i}\Sigma_iG_{-i}x_j
            \right|^2
        \right)^{1/2}
    \right\|_{L^k}.
\end{aligned}
\]
The first term is \(O_k(n^{-1})\), since
\[
    \left(
        \sum_{j\ne i}
        \norm{G_{-i}\Sigma_iG_{-i}x_j}_2^2
    \right)^{1/2}
    \le
    O(1)\norm{X_{-i}}_F
\]
and \(\Norm{\norm{X_{-i}}_F}_{L^k}=O_k(n)\). For the second term,
\[
    \left|
        x_j^\top G_{-i}\Sigma_iG_{-i}x_j
    \right|
    \le
    \norm{G_{-i}\Sigma_iG_{-i}}_{\op}\norm{x_j}_2^2 
    \le
    \norm{G_{-i}}_{\op}^2\norm{\Sigma_i}_{\op}\norm{x_j}_2^2
    \le
    O(1)\norm{x_j}_2^2,
\]
so
\begin{align*}
        \left\|
        \left(
            \sum_{j\ne i}
            \left|
                x_j^\top G_{-i}\Sigma_iG_{-i}x_j
            \right|^2
        \right)^{1/2}
    \right\|_{L^{2k}}
    &\le
    O(1)
    \left(
        \sum_{j\ne i}
        \Norm{\norm{x_j}_2^2}_{L^{2k}}^2
    \right)^{1/2}\\
    &=
    O(1)
    \left(
        \sum_{j\ne i}
        \Norm{\norm{x_j}_2}_{L^{4k}}^4
    \right)^{1/2}
    =
    O_k(n^{3/2}).
\end{align*}

Together with
\[
    \Norm{\norm{\hat\theta_{-i}}_2}_{L^{2k}}=O_k(1),
\]
this gives
\[
    \left\|
        \left(
            \sum_{j\ne i}
            \left[
                \sup_{\norm{h}_2\le1}
                \left|
                    \frac1n
                    \tr\!\left(
                        G_{-i}\Sigma_iG_{-i}\mathcal L_{ij}[h]
                    \right)
                \right|
            \right]^2
        \right)^{1/2}
    \right\|_{L^k}
    =
    O_k\!\left(\frac1{\sqrt n}\right).
\]

We now treat the background contribution. Define the random vector
\(z_i^\gamma\in\R^p\) by
\[
\begin{aligned}
    z_i^\gamma
    &:=
    \frac1n\sum_{\ell\ne i}
    L'''(x_\ell^\top\hat\theta_{-i})
    \left(
        x_\ell^\top G_{-i}\Sigma_iG_{-i}x_\ell
    \right)
    x_\ell
    +
    z_i^{\gamma,\rho},
\end{aligned}
\]
where \(z_i^{\gamma,\rho}\) is the representing vector of the linear functional
\[
    v
    \mapsto
    \tr\!\left(
        G_{-i}\Sigma_iG_{-i}
        \nabla^3\rho(\hat\theta_{-i})[v]
    \right),
\]
that is,
\[
    (z_i^{\gamma,\rho})^\top v
    =
    \tr\!\left(
        G_{-i}\Sigma_iG_{-i}
        \nabla^3\rho(\hat\theta_{-i})[v]
    \right).
\]
Then, for every \(h\in\R^p\),
\[
    \tr\!\left(
        G_{-i}\Sigma_iG_{-i}\mathcal B_{ij}[h]
    \right)
    =
    (z_i^\gamma)^\top
    \mathbb D_j[\hat\theta_{-i}][h].
\]
Thus the background contribution to the square-summed derivative is
\[
\begin{aligned}
    &\left(
        \sum_{j\ne i}
        \left[
            \sup_{\norm{h}_2\le1}
            \left|
                \frac1n
                \tr\!\left(
                    G_{-i}\Sigma_iG_{-i}\mathcal B_{ij}[h]
                \right)
            \right|
        \right]^2
    \right)^{1/2} 
    =
    \frac1n
    \left(
        \sum_{j\ne i}
        \left(
            \norm{
                (z_i^\gamma)^\top
                \mathbb D_j[\hat\theta_{-i}]
            }^*
        \right)^2
    \right)^{1/2}.
\end{aligned}
\]

We first record the bound
\[
    \Norm{\norm{z_i^\gamma}_2}_{L^{4k}}
    =
    O_k(n).
\]
Indeed, set
\[
    A_i:=G_{-i}\Sigma_iG_{-i}.
\]
For the regularizer contribution, Assumption~\ref{ass:7third-order-curvature}
gives
\[
\begin{aligned}
    \norm{z_i^{\gamma,\rho}}_2
    &=
    \sup_{\norm{v}_2\le1}
    \left|
        \tr\!\left(
            A_i\nabla^3\rho(\hat\theta_{-i})[v]
        \right)
    \right|                                                   \\
    &\le
    \norm{A_i}_{*}
    \sup_{\norm{v}_2\le1}
    \norm{\nabla^3\rho(\hat\theta_{-i})[v]}_{\op}
    \le
    O(1)\norm{A_i}_{*}.
\end{aligned}
\]
Since \(\norm{A_i}_{\op}=O(1)\) and \(\operatorname{rank}(A_i)\le p=O(n)\),
\[
    \norm{z_i^{\gamma,\rho}}_2
    =
    O(n).
\]
This is sufficient for the argument below, since the loss contribution is also
\(O_k(n)\) in \(L^{4k}\).
For the loss part, using \(\norm{L'''}_\infty=O(1)\),
\(\norm{G_{-i}\Sigma_iG_{-i}}_{\op}=O(1)\), and
Lemma~\ref{lem:op_norm_X},
\[
\begin{aligned}
    &\left\|
        \frac1n\sum_{\ell\ne i}
        L'''(x_\ell^\top\hat\theta_{-i})
        \left(
            x_\ell^\top G_{-i}\Sigma_iG_{-i}x_\ell
        \right)
        x_\ell
    \right\|_2
    \le
    \frac{O(1)}{n}
    \norm{X_{-i}}_{\op}
    \left(
        \sum_{\ell\ne i}
        \norm{x_\ell}_2^4
    \right)^{1/2}.
\end{aligned}
\]
Taking the \(L^{4k}\)-norm and using
\[
    \Norm{\norm{X_{-i}}_{\op}}_{L^{8k}}=O_k(\sqrt n),
    \qquad
    \left\|
        \left(
            \sum_{\ell\ne i}\norm{x_\ell}_2^4
        \right)^{1/2}
    \right\|_{L^{8k}}
    =
    O_k(n^{3/2}),
\]
we obtain the claimed \(O_k(n)\) bound for
\(\Norm{\norm{z_i^\gamma}_2}_{L^{4k}}\).

We may now invoke Lemma~\ref{lem:square-summed-background-sensitivity}, applied
to the leave-one-out minimizer \(\hat\theta_{-i}\). The same proof applies
verbatim after replacing \(X,\hat\theta,G\) by
\(X_{-i},\hat\theta_{-i},G_{-i}\), with the normalization still equal to
\(1/n\). Hence
\[
\begin{aligned}
    &\left\|
        \left(
            \sum_{j\ne i}
            \left(
                \norm{
                    (z_i^\gamma)^\top
                    \mathbb D_j[\hat\theta_{-i}]
                }^*
            \right)^2
        \right)^{1/2}
    \right\|_{L^k}  
    \le
    O_k\!\left(\frac1{\sqrt n}\right)
    \Norm{\norm{z_i^\gamma}_2}_{L^{4k}}
    =
    O_k(\sqrt n).
\end{aligned}
\]
Multiplying by the outer factor \(1/n\), we get
\[
\begin{aligned}
    &\left\|
        \left(
            \sum_{j\ne i}
            \left[
                \sup_{\norm{h}_2\le1}
                \left|
                    \frac1n
                    \tr\!\left(
                        G_{-i}\Sigma_iG_{-i}\mathcal B_{ij}[h]
                    \right)
                \right|
            \right]^2
        \right)^{1/2}
    \right\|_{L^k}
    =
    O_k\!\left(\frac1{\sqrt n}\right).
\end{aligned}
\]

Combining the local and background contributions gives
\[
    \sup_{i\in[n]}
    \left\|
        \left(
            \sum_{j\ne i}
            \left(
                \norm{
                    \mathbb D_j\!\left[
                        \frac1n\tr(\Sigma_iG_{-i})
                    \right]
                }_2^*
            \right)^2
        \right)^{1/2}
    \right\|_{L^k}
    =
    O_k\!\left(\frac1{\sqrt n}\right).
\]
The tensorized \(L^k\)-Poincar\'e inequality therefore yields
\[
    \sup_{i\in[n]}
    \left\|
        \frac1n\tr(\Sigma_iG_{-i})
        -
        \E\!\left[
            \frac1n\tr(\Sigma_iG_{-i})
        \right]
    \right\|_{L^k}
    =
    O_k\!\left(\frac1{\sqrt n}\right).
\]
Together with the conditional quadratic-form estimate at the beginning of the
proof, this proves
\[
    \sup_{i\in[n]}
    \Norm{
        \frac1n x_i^\top G_{-i}x_i-\gamma_i
    }_{L^k}
    =
    O_k\!\left(\frac1{\sqrt n}\right),
\]
\end{proof}

For \(t\in\R\), define \(\zeta_i(t)\) as the unique solution of (see Lemma~\ref{lem:def_zeta} for a justification of this definition)
\[
    z+\gamma_iL'(z)=t,
\]
where
\[
    \gamma_i
    :=
    \E\!\left[
        \frac1n x_i^\top G_{-i}x_i
    \right].
\]
We also set
\[
    s_i^\zeta
    :=
    L'\!\left(
        \zeta_i(x_i^\top\hat\theta_{-i})
    \right).
\]

\begin{lemma}[Definition and elementary properties of \(\zeta_i\)]
\label{lem:def_zeta}
Under Assumption~\ref{ass:5regularity}, for every \(i\in[n]\) and every
\(t\in\R\), the equation
\[
    z+\gamma_iL'(z)=t
\]
admits a unique solution \(z=\zeta_i(t)\). Moreover, \(t\mapsto\zeta_i(t)\) is
\(1\)-Lipschitz and
\[
    |\zeta_i(t)|
    \le
    |t|+\gamma_i|L'(0)|.
\]
Consequently, under Assumptions~\ref{ass:1dim}--\ref{ass:5regularity}, for
every fixed \(k\ge1\),
\[
    \sup_{i\in[n]}\Norm{s_i^\zeta}_{L^k}
    =
    O_k(1).
\]
\end{lemma}

\begin{proof}
Let
\[
    H_i(z):=z+\gamma_iL'(z).
\]
Then
\[
    H_i'(z)=1+\gamma_iL''(z)\ge1+\gamma_i\kappa_L\ge1.
\]
Hence \(H_i\) is strictly increasing. Moreover, since \(H_i'(z)\ge1\),
\[
    H_i(z)\to+\infty
    \quad\text{as }z\to+\infty,
    \qquad
    H_i(z)\to-\infty
    \quad\text{as }z\to-\infty.
\]
Thus \(H_i\) is a bijection from \(\R\) to \(\R\), and
\(\zeta_i=H_i^{-1}\) is well-defined. Since \(H_i'\ge1\), its inverse is
\(1\)-Lipschitz. In particular,
\[
    |\zeta_i(t)|
    =
    |\zeta_i(t)-\zeta_i(H_i(0))|
    \le
    |t-H_i(0)|
    \le
    |t|+\gamma_i|L'(0)|.
\]

Taking \(t=x_i^\top\hat\theta_{-i}\), conditionally on \(X_{-i}\), the vector
\(\hat\theta_{-i}\) is deterministic and independent of \(x_i\). Hence
Lemma~\ref{lem:moments_of_x_i} and Theorem~\ref{thm:direct-hat-bound} give
\[
    \sup_{i\in[n]}
    \Norm{x_i^\top\hat\theta_{-i}}_{L^k}
    =
    O_k(1).
\]
Since \(\gamma_i=O(1)\), \(|L'(0)|=O(1)\), and
\[
    |L'(z)|
    \le
    |L'(0)|+\norm{L''}_\infty |z|,
\]
we obtain
\[
    \sup_{i\in[n]}
    \Norm{
        L'\!\left(
            \zeta_i(x_i^\top\hat\theta_{-i})
        \right)
    }_{L^k}
    =
    O_k(1),
\]
which is the desired bound on \(s_i^\zeta\).
\end{proof}

\begin{lemma}[Derivative of the averaged inverse score]
\label{lem:derivative-zeta}
Under Assumption~\ref{ass:5regularity}, for each \(i\in[n]\), the deterministic
map
\[
    t\mapsto\zeta_i(t),
    \qquad
    \zeta_i(t)+\gamma_iL'(\zeta_i(t))=t,
\]
is \(C^1\), and
\[
    \zeta_i'(t)
    =
    \frac{1}{1+\gamma_iL''(\zeta_i(t))}.
\]
In particular, for \(j\ne i\),
\[
    \mathbb D_j\!\left[
        \zeta_i(x_i^\top\hat\theta_{-i})
    \right][h]
    =
    \frac{
        x_i^\top\mathbb D_j[\hat\theta_{-i}][h]
    }
    {
        1+\gamma_iL''\!\left(
            \zeta_i(x_i^\top\hat\theta_{-i})
        \right)
    },
\]
and
\[
    \mathbb D_j[s_i^\zeta][h]
    =
    \frac{
        L''\!\left(
            \zeta_i(x_i^\top\hat\theta_{-i})
        \right)
    }
    {
        1+\gamma_iL''\!\left(
            \zeta_i(x_i^\top\hat\theta_{-i})
        \right)
    }
    x_i^\top\mathbb D_j[\hat\theta_{-i}][h].
\]
\end{lemma}

\begin{proof}
The derivative of \(\zeta_i\) follows from the implicit function theorem
applied to
\[
    F(z,t):=z+\gamma_iL'(z)-t.
\]
Indeed,
\[
    \partial_zF(z,t)=1+\gamma_iL''(z)\ge1,
\]
so the implicit derivative is well-defined and gives
\[
    \zeta_i'(t)
    =
    \frac1{1+\gamma_iL''(\zeta_i(t))}.
\]

Since \(\gamma_i\) is deterministic, differentiating
\[
    \zeta_i(x_i^\top\hat\theta_{-i})
    +
    \gamma_i
    L'\!\left(
        \zeta_i(x_i^\top\hat\theta_{-i})
    \right)
    =
    x_i^\top\hat\theta_{-i}
\]
with respect to the column \(x_j\), \(j\ne i\), gives
\[
    \left[
        1+
        \gamma_iL''\!\left(
            \zeta_i(x_i^\top\hat\theta_{-i})
        \right)
    \right]
    \mathbb D_j\!\left[
        \zeta_i(x_i^\top\hat\theta_{-i})
    \right][h]
    =
    x_i^\top\mathbb D_j[\hat\theta_{-i}][h].
\]
This proves the formula for
\(\mathbb D_j[\zeta_i(x_i^\top\hat\theta_{-i})]\). The formula for
\(\mathbb D_j[s_i^\zeta]\) follows from the chain rule.
\end{proof}

\begin{lemma}[Approximation of the train score]
\label{lem:score-xi-approximation}
Under Assumptions~\ref{ass:1dim}--\ref{ass:7third-order-curvature}, for every
fixed \(k\ge1\),
\[
    \sup_{i\in[n]}
    \Norm{s_i-s_i^\zeta}_{L^k}
    =
    O_k\!\left(\frac{\log n}{\sqrt n}\right).
\]
\end{lemma}

\begin{proof}
Set
\[
    \lambda_i:=\frac1n x_i^\top\bar G_i x_i.
\]
By Lemma~\ref{lem:loo-exact},
\[
    x_i^\top\hat\theta
    =
    x_i^\top\hat\theta_{-i}
    -
    \lambda_iL'(x_i^\top\hat\theta).
\]
On the other hand, by definition of \(\zeta_i(x_i^\top\hat\theta_{-i})\),
\[
    \zeta_i(x_i^\top\hat\theta_{-i})
    =
    x_i^\top\hat\theta_{-i}
    -
    \gamma_i
    L'\!\left(
        \zeta_i(x_i^\top\hat\theta_{-i})
    \right).
\]

Let
\[
    H_i^\lambda(z):=z+\lambda_iL'(z).
\]
Since
\[
    \lambda_i=\frac1n x_i^\top\bar G_i x_i\ge0
\]
and \(L''\ge\kappa_L>0\), we have
\[
    (H_i^\lambda)'(z)
    =
    1+\lambda_iL''(z)
    \ge1.
\]
Hence \(H_i^\lambda\) is increasing with inverse Lipschitz constant at most
one: for all \(z,z'\in\R\),
\[
    |z-z'|
    \le
    |H_i^\lambda(z)-H_i^\lambda(z')|.
\]
By Lemma~\ref{lem:loo-exact},
\[
    H_i^\lambda(x_i^\top\hat\theta)
    =
    x_i^\top\hat\theta_{-i}.
\]
On the other hand, by definition of \(\zeta_i\),
\[
\begin{aligned}
    H_i^\lambda\!\left(
        \zeta_i(x_i^\top\hat\theta_{-i})
    \right)
    &=
    \zeta_i(x_i^\top\hat\theta_{-i})
    +
    \lambda_i
    L'\!\left(
        \zeta_i(x_i^\top\hat\theta_{-i})
    \right)                                      \\
    &=
    x_i^\top\hat\theta_{-i}
    +
    (\lambda_i-\gamma_i)
    L'\!\left(
        \zeta_i(x_i^\top\hat\theta_{-i})
    \right).
\end{aligned}
\]
Consequently,
\[
\begin{aligned}
    &\left|
        x_i^\top\hat\theta
        -
        \zeta_i(x_i^\top\hat\theta_{-i})
    \right| 
    \le
    \left|
        H_i^\lambda(x_i^\top\hat\theta)
        -
        H_i^\lambda\!\left(
            \zeta_i(x_i^\top\hat\theta_{-i})
        \right)
    \right|
    =
    |\lambda_i-\gamma_i|
    \left|
        L'\!\left(
            \zeta_i(x_i^\top\hat\theta_{-i})
        \right)
    \right|.
\end{aligned}
\]
Since the map \(z\mapsto z+\lambda_iL'(z)\) has derivative at least one,
\[
    \left|
        x_i^\top\hat\theta
        -
        \zeta_i(x_i^\top\hat\theta_{-i})
    \right|
    \le
    |\lambda_i-\gamma_i|\,
    \left|
        L'\!\left(
            \zeta_i(x_i^\top\hat\theta_{-i})
        \right)
    \right|.
\]
Therefore,
\[
\begin{aligned}
    |s_i-s_i^\zeta|
    &=
    \left|
        L'(x_i^\top\hat\theta)
        -
        L'\!\left(
            \zeta_i(x_i^\top\hat\theta_{-i})
        \right)
    \right|
    \le
    \norm{L''}_\infty
    \left|
        x_i^\top\hat\theta
        -
        \zeta_i(x_i^\top\hat\theta_{-i})
    \right|
    \le
    O(1)|\lambda_i-\gamma_i|\,|s_i^\zeta|.
\end{aligned}
\]
It remains to bound \(\lambda_i-\gamma_i\). We decompose
\[
    \lambda_i-\gamma_i
    =
    \left(
        \frac1n x_i^\top G_{-i}x_i-\gamma_i
    \right)
    +
    \frac1n x_i^\top(\bar G_i-G_{-i})x_i .
\]
The first term is bounded by Lemma~\ref{lem:deterministic-gamma-concentration}:
\[
    \left\|
        \frac1n x_i^\top G_{-i}x_i-\gamma_i
    \right\|_{L^k}
    =
    O_k\!\left(\frac1{\sqrt n}\right).
\]

For the second term, use
\[
    \bar G_i-G_{-i}
    =
    -\bar G_iR_iG_{-i}.
\]
The regularizer contribution is bounded by
\[
\begin{aligned}
    \left|
        \frac1n
        x_i^\top\bar G_iR_i^\rho G_{-i}x_i
    \right|
    &\le
    \frac1n
    \norm{x_i}_2^2
    \norm{\bar G_i}_{\op}
    \norm{R_i^\rho}_{\op}
    \norm{G_{-i}}_{\op}
    \le
    \frac{O(1)}{n\sqrt n}
    |s_i^-|\norm{x_i}_2^2,
\end{aligned}
\]
where we used \eqref{eq:bound_R_rho}. Hence
\[
    \left\|
        \frac1n
        x_i^\top\bar G_iR_i^\rho G_{-i}x_i
    \right\|_{L^k}
    =
    O_k(n^{-1/2}).
\]
For the loss contribution,
\[
\begin{aligned}
    \left|
        \frac1n
        x_i^\top\bar G_iR_i^LG_{-i}x_i
    \right|
    &\le
    \frac1{n^2}
    \norm{X_{-i}^\top\bar G_i x_i}_2\,
    \norm{\Gamma_i}_F\,
    \norm{X_{-i}^\top G_{-i}x_i}_\infty .
\end{aligned}
\]
By Lemma~\ref{lem:curvature-resolvent},
\[
    \norm{X_{-i}^\top\bar G_i x_i}_2
    \le
    \norm{\bar G_iX_{-i}}_{\op}\norm{x_i}_2
    =
    O(\sqrt n)\norm{x_i}_2.
\]
Using Lemmas~\ref{lem:moments_of_x_i},
\ref{lem:gamma-frobenius-bound}, and
\ref{lem:cross-leverage-poincare}, we get
\[
    \left\|
        \frac1n
        x_i^\top\bar G_iR_i^LG_{-i}x_i
    \right\|_{L^k}
    =
    O_k\!\left(\frac{\log n}{\sqrt n}\right).
\]
Thus
\[
    \Norm{\lambda_i-\gamma_i}_{L^k}
    =
    O_k\!\left(\frac{\log n}{\sqrt n}\right).
\]
Since \(\Norm{s_i^\zeta}_{L^{2k}}=O_k(1)\) by
Lemma~\ref{lem:def_zeta}, H\"older's inequality gives
\[
    \Norm{s_i-s_i^\zeta}_{L^k}
    =
    O_k\!\left(\frac{\log n}{\sqrt n}\right).
\]
\end{proof}

From now on, in the scalar leave-one-out approximation, we use the averaged
first-order displacement
\[
    \delta_i^{(0)}\hat\theta
    :=
    -\frac{s_i^\zeta}{n}G_{-i}x_i.
\]

\begin{lemma}[First-order bounds with deterministic averaged score]
\label{lem:first-order-averaged-score-bounds}
Under Assumptions~\ref{ass:1dim}--\ref{ass:7third-order-curvature}, for every
fixed \(k\ge1\),
\[
    \sup_{i\in[n]}
    \left\|
        \norm{
            \delta_i\hat\theta-\delta_i^{(0)}\hat\theta
        }_2
    \right\|_{L^k}
    =
    O_k\!\left(\frac{\log n}{n}\right),
\]
and
\[
    \sup_{i\in[n]}
    \left\|
        \norm{
            X_{-i}^\top
            (\delta_i\hat\theta-\delta_i^{(0)}\hat\theta)
        }_2
    \right\|_{L^k}
    =
    O_k\!\left(\frac{\log n}{\sqrt n}\right).
\]
\end{lemma}

\begin{proof}
Write
\[
    \delta_i\hat\theta-\delta_i^{(0)}\hat\theta
    =
    \delta_i\hat\theta-\delta_{i,\mathrm{tr}}^{(0)}\hat\theta
    -
    \frac{s_i-s_i^\zeta}{n}G_{-i}x_i.
\]
By Lemma~\ref{lem:score-xi-approximation} and
Lemma~\ref{lem:moments_of_x_i},
\[
    \left\|
        \frac{s_i-s_i^\zeta}{n}G_{-i}x_i
    \right\|_{L^k}
    =
    O_k\!\left(\frac{\log n}{n}\right).
\]
Combining this with Theorem~\ref{thm:loo-remainder-tilde} gives the first
bound.

For the projected bound,
\[
    \norm{
        X_{-i}^\top
        \frac{s_i-s_i^\zeta}{n}G_{-i}x_i
    }_2
    \le
    \frac{|s_i-s_i^\zeta|}{n}
    \norm{X_{-i}^\top G_{-i}x_i}_2.
\]
Lemma~\ref{lem:score-xi-approximation} and
Lemma~\ref{lem:cross-leverage-poincare} give
\[
    \left\|
        \norm{
            X_{-i}^\top
            \frac{s_i-s_i^\zeta}{n}G_{-i}x_i
        }_2
    \right\|_{L^k}
    =
    O_k\!\left(\frac{\log n}{\sqrt n}\right).
\]
The projected estimate follows again from
Theorem~\ref{thm:loo-remainder-tilde}.
\end{proof}
\begin{corollary}[First-order bounds with deterministic averaged score]
\label{cor:first-order-averaged-score-bounds}
Under Assumptions~\ref{ass:1dim}--\ref{ass:7third-order-curvature}, for every
fixed \(k\ge1\),
\[
    \sup_{i\in[n]}
    \left\|
        \norm{X_{-i}^\top\delta_i^{(0)}\hat\theta}_2
    \right\|_{L^k}
    =
    O_k(1),
\qquad
\text{and}
\qquad
    \sup_{i\in[n]}
    \left\|
        \norm{X_{-i}^\top\delta_i^{(0)}\hat\theta}_\infty
    \right\|_{L^k}
    =
    O_k\!\left(\frac{\log n}{\sqrt n}\right).
\]
Consequently,
\[
    \sup_{i\in[n]}
    \left\|
        \norm{X_{-i}^\top\delta_i\hat\theta}_\infty
    \right\|_{L^k}
    =
    O_k\!\left(\frac{\log n}{\sqrt n}\right).
\]
\end{corollary}
\begin{proof}
    Let us bound:
    \[
    \norm{X_{-i}^\top\delta_i^{(0)}\hat\theta}_2
    \le
    \frac{|s_i^\zeta|}{n}
    \norm{X_{-i}^\top G_{-i}x_i}_2,
\]
and
\[
    \norm{X_{-i}^\top\delta_i^{(0)}\hat\theta}_\infty
    \le
    \frac{|s_i^\zeta|}{n}
    \norm{X_{-i}^\top G_{-i}x_i}_\infty.
\]
The \(L^k\)-bounds follow from Lemma~\ref{lem:def_zeta} and
Lemma~\ref{lem:cross-leverage-poincare}. The final bound for
\(X_{-i}^\top\delta_i\hat\theta\) follows by adding the projected error bound
proved above.

\end{proof}

\subsection{Second-order expansion}
\label{subsec:second-order-loo}

In this subsection we work under
Assumptions~\ref{ass:1dim}--\ref{ass:7third-order-curvature} and under the
supplementary condition \(\nabla^3\rho\equiv0\). Thus \(\nabla^2\rho\) is
constant and \(R_i^\rho=0\).

Define the diagonal matrix \(\tilde \Gamma_i\), indexed by \(j\ne i\), by
\[
    (\tilde \Gamma_i)_j
    :=
    \frac12
    L'''(x_j^\top\hat\theta_{-i})\,x_j^\top\delta_i^{(0)}\hat\theta
    =
    -\frac{s_i^\zeta}{2n}
    L'''(x_j^\top\hat\theta_{-i})\,x_j^\top G_{-i}x_i,
\]
and set
\[
    \tilde R_i
    :=
    \frac1nX_{-i}\tilde \Gamma_iX_{-i}^\top.
\]
The second-order correction is
\[
    \delta_i^{(1)}\hat\theta
    :=
    \frac{s_i^\zeta}{n}G_{-i}\tilde R_iG_{-i}x_i.
\]

\begin{lemma}[Second-order curvature remainder with deterministic averaged score]
\label{lem:second-order-curvature-remainders-xi}
Let \(E_i:=\Gamma_i-\tilde \Gamma_i\), equivalently
\[
    R_i^L-\tilde R_i
    =
    \frac1nX_{-i}E_iX_{-i}^\top.
\]
Then, for every fixed \(k\ge1\),
\[
    \sup_{i\in[n]}
    \left\|
        \norm{E_i}_F
    \right\|_{L^k}
    =
    O_k\!\left(\frac{\log n}{\sqrt n}\right).
\]
Moreover,
\[
    \sup_{i\in[n]}
    \left\|
        \norm{\tilde\Gamma_i}_F
    \right\|_{L^k}
    =
    O_k(1),
\]
and
\[
    \sup_{i\in[n]}
    \Norm{\norm{\Gamma_i}_{\op}}_{L^k}+\sup_{i\in[n]}
    \left\|
        \norm{\tilde\Gamma_i}_{\op}
    \right\|_{L^k}
    =
    O_k\!\left(\frac{\log n}{\sqrt n}\right).
\]
\end{lemma}

\begin{proof}
By Taylor's formula and the Lipschitz bound on \(L'''\),
\[
    |(E_i)_j|
    \le
    O(1)|x_j^\top(\delta_i\hat\theta-\delta_i^{(0)}\hat\theta)|
    +
    O(1)|x_j^\top\delta_i\hat\theta|^2.
\]
Therefore
\[
    \norm{E_i}_F
    \le
    O(1)\norm{
        X_{-i}^\top(\delta_i\hat\theta-\delta_i^{(0)}\hat\theta)
    }_2
    +
    O(1)
    \norm{X_{-i}^\top\delta_i\hat\theta}_\infty
    \norm{X_{-i}^\top\delta_i\hat\theta}_2 .
\]
The first term is \(O_k((\log n)/\sqrt n)\) by
Lemma~\ref{lem:first-order-averaged-score-bounds}. The second has the same
order by Lemma~\ref{lem:first-order-averaged-score-bounds},
Corollary~\ref{cor:first-order-averaged-score-bounds}, and H\"older's inequality. Hence
\[
    \sup_{i\in[n]}
    \left\|
        \norm{E_i}_F
    \right\|_{L^k}
    =
    O_k\!\left(\frac{\log n}{\sqrt n}\right).
\]

The Frobenius bound follows from
\[
    \norm{\tilde\Gamma_i}_F
    \le
    \frac{O(1)|s_i^\zeta|}{n}
    \norm{X_{-i}^\top G_{-i}x_i}_2,
\]
and the operator bounds follows from
\[
    \norm{\Gamma_i}_{\op}
    \le
    O(1)\norm{X_{-i}^\top\delta_i\hat\theta}_\infty
    \qquad \text{and}\qquad
    \norm{\tilde\Gamma_i}_{\op}
    \le
    \frac{O(1)|s_i^\zeta|}{n}
    \norm{X_{-i}^\top G_{-i}x_i}_\infty.
\]
We conclude using Lemmas~\ref{lem:first-order-averaged-score-bounds},~\ref{lem:def_zeta},~\ref{lem:cross-leverage-poincare} and Corollary~\ref{cor:first-order-averaged-score-bounds}.
\end{proof}
\begin{theorem}[Projected second-order expansion with deterministic averaged score]
\label{thm:second-order-loo-xi}
Under Assumptions~\ref{ass:1dim}--\ref{ass:7third-order-curvature} and
\(\nabla^3\rho\equiv0\), for every fixed \(k\ge1\) and every deterministic
\(u\in\R^p\) with \(\norm{u}_2\le1\),
\[
    \sup_{i\in[n]}
    \left\|
        u^\top
        \left(
            \delta_i\hat\theta
            -
            \delta_i^{(0)}\hat\theta
            -
            \delta_i^{(1)}\hat\theta
        \right)
    \right\|_{L^k}
    =
    O_k\!\left(\frac{(\log n)^2}{n^{3/2}}\right).
\]
Furthermore,
\[
    \sup_{i\in[n]}
    \left\|
        \norm{
            \delta_i\hat\theta
            -
            \delta_i^{(0)}\hat\theta
            -
            \delta_i^{(1)}\hat\theta
        }_2
    \right\|_{L^k}
    =
    O_k\!\left(\frac{\log n}{n}\right).
\]
\end{theorem}

\begin{proof}
Since \(\nabla^3\rho\equiv0\), \(R_i=R_i^L\). Starting from
\[
    \delta_i\hat\theta
    =
    -\frac{s_i}{n}\bar G_i x_i,
    \qquad
    \delta_i^{(0)}\hat\theta
    =
    -\frac{s_i^\zeta}{n}G_{-i}x_i,
\]
and using
\[
    \bar G_i-G_{-i}
    =
    -\bar G_iR_i^LG_{-i},
\]
we get
\[
\begin{aligned}
    \delta_i\hat\theta
    -
    \delta_i^{(0)}\hat\theta
    -
    \delta_i^{(1)}\hat\theta
    &=
    -\frac{s_i-s_i^\zeta}{n}G_{-i}x_i                         \\
    &\quad
    +
    \frac{s_i}{n}\bar G_i(R_i^L-\tilde R_i)G_{-i}x_i          \\
    &\quad
    +
    \frac{s_i}{n}(\bar G_i-G_{-i})\tilde R_iG_{-i}x_i         \\
    &\quad
    +
    \frac{s_i-s_i^\zeta}{n}G_{-i}\tilde R_iG_{-i}x_i .
\end{aligned}
\]
We bound the deterministic projection of these four terms.

First,
\[
    \left|
        \frac{s_i-s_i^\zeta}{n}u^\top G_{-i}x_i
    \right|
    \le
    \frac{|s_i-s_i^\zeta|}{n}|u^\top G_{-i}x_i|.
\]
By Lemma~\ref{lem:score-xi-approximation} and
Lemma~\ref{lem:moments_of_x_i},
\[
    \left\|
        \frac{s_i-s_i^\zeta}{n}u^\top G_{-i}x_i
    \right\|_{L^k}
    =
    O_k\!\left(\frac{\log n}{n^{3/2}}\right).
\]

Second, since \(R_i^L-\tilde R_i=n^{-1}X_{-i}E_iX_{-i}^\top\),
\[
\begin{aligned}
    &\left|
        \frac{s_i}{n}
        u^\top\bar G_i(R_i^L-\tilde R_i)G_{-i}x_i
    \right|
    \le
    \frac{|s_i|}{n^2}
    \norm{X_{-i}^\top\bar G_i u}_2
    \norm{E_i}_F
    \norm{X_{-i}^\top G_{-i}x_i}_\infty 
    =O_k\!\left(\frac{(\log n)^2}{n^{3/2}}\right),
\end{aligned}
\]
using $\norm{X_{-i}^\top\bar G_i u}_2=O(\sqrt n)$ and 
Lemmas~\ref{lem:score-moments},
\ref{lem:second-order-curvature-remainders-xi},~\ref{lem:cross-leverage-poincare}.

For the third term, use
\[
    \bar G_i-G_{-i}
    =
    -\frac1n\bar G_iX_{-i}\Gamma_iX_{-i}^\top G_{-i},
    \qquad
    \tilde R_i
    =
    \frac1nX_{-i}\tilde\Gamma_iX_{-i}^\top.
\]
Then
\[
\begin{aligned}
    &\left|
        \frac{s_i}{n}
        u^\top(\bar G_i-G_{-i})\tilde R_iG_{-i}x_i
    \right|                                                  \\
    &\qquad\le
    \frac{|s_i|}{n^3}
    \norm{X_{-i}^\top\bar G_i u}_2
    \norm{\Gamma_i}_{\op}
    \norm{X_{-i}^\top G_{-i}X_{-i}}_{\op}
    \norm{\tilde\Gamma_i}_F
    \norm{X_{-i}^\top G_{-i}x_i}_\infty .
\end{aligned}
\]
By Lemmas~\ref{lem:score-moments},
\ref{lem:curvature-resolvent},
\ref{lem:second-order-curvature-remainders-xi}, and
\ref{lem:cross-leverage-poincare}, this term is
\[
    O_k\!\left(\frac{(\log n)^2}{n^{3/2}}\right)
\]
in \(L^k\).

For the fourth term,
\[
\begin{aligned}
    \left|
        \frac{s_i-s_i^\zeta}{n}
        u^\top G_{-i}\tilde R_iG_{-i}x_i
    \right|
    &\le
    \frac{|s_i-s_i^\zeta|}{n^2}
    \norm{X_{-i}^\top G_{-i}u}_2
    \norm{\tilde\Gamma_i}_{\op}
    \norm{X_{-i}^\top G_{-i}x_i}_2 .
\end{aligned}
\]
Using
\[
    \norm{X_{-i}^\top G_{-i}u}_2=O(\sqrt n),
    \qquad
    \Norm{\norm{X_{-i}^\top G_{-i}x_i}_2}_{L^k}=O_k(n),
\]
together with Lemmas~\ref{lem:score-xi-approximation} and
\ref{lem:second-order-curvature-remainders-xi}, we get
\[
    \left\|
        \frac{s_i-s_i^\zeta}{n}
        u^\top G_{-i}\tilde R_iG_{-i}x_i
    \right\|_{L^k}
    =
    O_k\!\left(\frac{(\log n)^2}{n^{3/2}}\right).
\]
This proves the projected estimate.

The vector estimate follows from the same decomposition. The first term is
bounded by
\[
    \left\|
        \frac{s_i-s_i^\zeta}{n}|G_{-i}x_i|_2
    \right\|_{L^k}
    =
    O_k\!\left(\frac{\log n}{n}\right),
\]
and the other three terms are
\(O_k((\log n)^2/n^{3/2})\) in vector norm by the same estimates as above,
replacing \(\norm{X_{-i}^\top\bar G_i u}_2\) by
\(\norm{\bar G_iX_{-i}}_{\op}\) that has the same order. Hence the vector remainder is
\(O_k((\log n)/n)\).
\end{proof}

\begin{remark}[Consequence for the Wasserstein bound]
\label{rem:csqce_wasserstein_bound_scd_order_leaveone_out_expension}
Define
\[
\begin{aligned}
    \mathfrak d_i(X)
    &:=
    \sqrt n\,u^\top
    \left(
        \delta_i^{(0)}\hat\theta(X)
        +
        \delta_i^{(1)}\hat\theta(X)
    \right)                                      \\
    &=
    -\frac{s_i^\zeta}{\sqrt n}u^\top G_{-i}x_i
    +
    \frac{s_i^\zeta}{n^{3/2}}
    u^\top G_{-i}X_{-i}\tilde\Gamma_iX_{-i}^\top G_{-i}x_i .
\end{aligned}
\]
By Theorem~\ref{thm:second-order-loo-xi},
\[
    \sup_{i\in[n]}
    \Norm{
        \delta_i f_n(X)-\mathfrak d_i(X)
    }_{L^4}
    =
    O\!\left(\frac{(\log n)^2}{n}\right).
\]
The same estimate holds with \(X\) replaced by any \(X^A\), since \(X^A\) has
the same distribution as \(X\). Consequently,
\[
\begin{aligned}
    &\left|
        c_n^x
        -
        \sup_{\substack{i,j\in[n]\\i\ne j}}
        \left\|
            \mathfrak d_i(X)
            -
            \E_j\big[\mathfrak d_i(X)\big]
        \right\|_{L^4}
    \right|
    =
    O\!\left(\frac{(\log n)^2}{n}\right).
\end{aligned}
\]
Moreover,
\[
    \sup_{j\in[n]}
    \Norm{\delta_j f_n(X)}_{L^4}
    +
    \sup_{j\in[n]}
    \Norm{\mathfrak d_j(X)}_{L^4}
    =
    O\!\left(\frac{\log n}{\sqrt n}\right).
\]
Using conditional Cauchy--Schwarz and Jensen's inequality, the two preceding
bounds imply, uniformly in \(j\),
\[
\begin{aligned}
    \left\|
        \Cov_j\!\left(
            \delta_j f_n(X),
            \delta_j f_n(X^{A_j})
        \right)
        -
        \Cov_j\!\left(
            \mathfrak d_j(X),
            \mathfrak d_j(X^{A_j})
        \right)
    \right\|_{L^2}
    =
    O\!\left(\frac{(\log n)^3}{n^{3/2}}\right).
\end{aligned}
\]
Therefore
\[
\begin{aligned}
    &\left|
        \omega_c
        -
        \sup_{j\in[n]}
        \sqrt{
            \Var\!\left(
                \Cov_j\!\left(
                    \mathfrak d_j(X),
                    \mathfrak d_j(X^{A_j})
                \right)
            \right)
        }
    \right|
    \le
    O\!\left(\frac{(\log n)^3}{n^{3/2}}\right).
\end{aligned}
\]
\end{remark}



\section{Sensitivity analysis for the covariance terms}
\label{sec:sensitivity-covariances}

\subsection{Notation and preliminary bounds}
\label{subsec:covariance-preliminaries}

We now prove the estimate needed to control \(c_n^x\). The estimates are first
written in a generic form, without the leave-one-out index. They apply to the
leave-one-out quantities by replacing
\[
    X,\hat\theta,G,\tilde\Gamma_x,x,s_x^\zeta
    \quad\text{with}\quad
    X_{-i},\hat\theta_{-i},G_{-i},\tilde\Gamma_i,x_i,s_i^\zeta .
\]
This replacement does not change the estimates: the columns remain independent,
the normalization is still \(1/n\), the curvature lower bounds are unchanged,
and all moment and Poincar\'e constants are uniform. Throughout this section
\(r\ge1\) is fixed, and we write \(\log n\) instead of \(\log(en)\), which is
equivalent for our asymptotic bounds.

Let \(x\in\R^p\) be an auxiliary column, independent of \(X\), satisfying the
same columnwise assumptions as the columns of \(X\). 
We fix a deterministic scalar \(\gamma_x\) satisfying
\[
    0\le \gamma_x=O(1),
\]
corresponding in the leave-one-out application to
\[
    \gamma_i
    =
    \E\!\left[
        \frac1n x_i^\top G_{-i}x_i
    \right].
\]
With this convention, the map
\[
    z\mapsto z+\gamma_xL'(z)
\]
has derivative \(1+\gamma_xL''(z)\ge1\), and hence is a bijection of
\(\R\) onto \(\R\). Thus \(\zeta_x:\R\to\R\) satisfying
\[
    \zeta_x(t)+\gamma_xL'(\zeta_x(t))=t,
\]
 is well-defined and \(1\)-Lipschitz.
Further set
\[
    s_x^\zeta
    :=
    L'\!\left(
        \zeta_x(x^\top\hat\theta)
    \right).
\]
Under the replacement
\[
    (X,\hat\theta,G,x,\gamma_x)
    \mapsto
    (X_{-i},\hat\theta_{-i},G_{-i},x_i,\gamma_i),
\]
this is precisely
\[
    s_i^\zeta
    =
    L'\!\left(
        \zeta_i(x_i^\top\hat\theta_{-i})
    \right).
\]

Throughout the generic estimates, set
\[
    G
    :=
    \left(\nabla^2\Psi_X(\hat\theta)\right)^{-1},
    \qquad
    s_j:=L'(x_j^\top\hat\theta),
    \qquad j\in[n].
\]
Here \(s_j\) denotes the usual train score of the background sample \(X\); it is
not replaced by an averaged score.

We introduce the diagonal matrix \(\tilde\Gamma_x\), indexed by the columns of
\(X\), by
\[
    (\tilde\Gamma_x)_\ell
    :=
    -\frac{s_x^\zeta}{2n}
    L'''(x_\ell^\top\hat\theta)\,x_\ell^\top Gx .
\]
With \(X=X_{-i}\) and \(x=x_i\), this is exactly the matrix
\(\tilde\Gamma_i\) used in the second-order leave-one-out expansion.


We shall repeatedly use the following deterministic consequences of the
curvature argument in Lemma~\ref{lem:curvature-resolvent}:
\[
    \norm{G}_{\op}=O(1),
    \qquad
    \norm{GX}_{\op}
    =
    \norm{X^\top G}_{\op}
    =
    O(\sqrt n),
    \qquad
    \norm{X^\top GX}_{\op}=O(n).
\]
Similarly, the generic version of Lemma~\ref{lem:cross-leverage-poincare}
gives, for every fixed \(r\ge1\),
\[
    \Norm{\norm{Gx}_2}_{L^r}
    =
    O_r(\sqrt n),
    \qquad
    \Norm{\norm{X^\top Gx}_2}_{L^r}
    =
    O_r(n),
    \qquad
    \Norm{\norm{X^\top Gx}_\infty}_{L^r}
    =
    O_r(\sqrt n\log n).
\]
Moreover, by Lemma~\ref{lem:def_zeta},
\[
    \Norm{s_x^\zeta}_{L^r}=O_r(1).
\]

\begin{lemma}[Generic leverage bounds]
\label{lem:generic-leverage-bounds}
Under Assumptions~\ref{ass:1dim}--\ref{ass:7third-order-curvature} and
\(\nabla^3\rho\equiv0\), for every fixed integer \(r\ge1\), and every
deterministic \(u\in\R^p\) with \(\norm{u}_2\le1\),
\[
    \Norm{u^\top Gx}_{L^r}
    =
    O_r(1),
    \qquad
    \Norm{\norm{X^\top Gu}_2}_{L^r}
    =
    O_r(\sqrt n).
\]
\end{lemma}

\begin{proof}
Conditionally on \(X\), \(Gu\) is deterministic and
\(\norm{Gu}_2=O(1)\). Lemma~\ref{lem:moments_of_x_i} gives
\[
    \Norm{u^\top Gx}_{L_x^r}=O_r(1),
\]
and the first estimate follows by taking the \(L^r\)-norm over \(X\). The
second estimate is deterministic:
\[
    \norm{X^\top Gu}_2
    \le
    \norm{X^\top G}_{\op}\norm{u}_2
    =
    O(\sqrt n).
\]
\end{proof}

\begin{lemma}[Bounds involving \(\tilde\Gamma_x\)]
\label{lem:generic-tilde-gamma-bounds}
Under Assumptions~\ref{ass:1dim}--\ref{ass:7third-order-curvature} and
\(\nabla^3\rho\equiv0\), for every fixed integer \(r\ge1\),
\[
    \Norm{\norm{\tilde\Gamma_x}_{\op}}_{L^r}
    =
    O_r\!\left(\frac{\log n}{\sqrt n}\right),
\qquad
    \Norm{
        \norm{\tilde\Gamma_x X^\top Gx}_2
    }_{L^r}
    =
    O_r(\sqrt n\log n),
\]
and
\[
    \Norm{
        \norm{X^\top GX\tilde\Gamma_x X^\top Gx}_2
    }_{L^r}
    =
    O_r(n^{3/2}\log n).
\]
\end{lemma}

\begin{proof}
By the definition of \(\tilde\Gamma_x\) and the boundedness of \(L'''\),
\[
    \norm{\tilde\Gamma_x}_{\op}
    \le
    \frac{O(1)|s_x^\zeta|}{n}
    \norm{X^\top Gx}_\infty .
\]
Using \(\Norm{s_x^\zeta}_{L^r}=O_r(1)\) and the generic cross-leverage bound,
we obtain
\[
    \Norm{\norm{\tilde\Gamma_x}_{\op}}_{L^r}
    =
    O_r\!\left(\frac{\log n}{\sqrt n}\right).
\]
Next,
\[
    \norm{\tilde\Gamma_x X^\top Gx}_2
    \le
    \norm{\tilde\Gamma_x}_{\op}\norm{X^\top Gx}_2,
\]
and the second estimate follows from H\"older's inequality and
\(\Norm{\norm{X^\top Gx}_2}_{L^r}=O_r(n)\). Finally,
\[
    \norm{X^\top GX\tilde\Gamma_x X^\top Gx}_2
    \le
    \norm{X^\top GX}_{\op}
    \norm{\tilde\Gamma_x X^\top Gx}_2,
\]
and the deterministic bound \(\norm{X^\top GX}_{\op}=O(n)\) gives the last
estimate.
\end{proof}

\begin{lemma}[Uniform background drift]
\label{lem:uniform-background-drift}
Under Assumptions~\ref{ass:1dim}--\ref{ass:7third-order-curvature} and
\(\nabla^3\rho\equiv0\), define, for \(j\in[n]\),
\[
    \Delta_j
    :=
    \diag_{\ell\ne j}
    \left(
        L''(x_\ell^\top\hat\theta)
        -
        L''(x_\ell^\top\hat\theta_{-j})
    \right).
\]
Then, for every fixed integer \(r\ge1\),
\[
    \left\|
        \sup_{j\in[n]}
        \norm{\Delta_j}_{\op}
    \right\|_{L^r}
    =
    O_r\!\left(
        \frac{(\log n)^6}{\sqrt n}
    \right).
\]
\end{lemma}

\begin{proof}
By the Lipschitz property of \(L''\),
\[
    \norm{\Delta_j}_{\op}
    \le
    O(1)
    \norm{
        X_{-j}^\top(\hat\theta-\hat\theta_{-j})
    }_\infty .
\]
Using the train-score first-order term
\[
    \delta_{j,\mathrm{tr}}^{(0)}\hat\theta
    :=
    -\frac{s_j}{n}G_{-j}x_j,
\]
we get
\[
\begin{aligned}
    \sup_{j\in[n]}\norm{\Delta_j}_{\op}
    &\le
    O(1)
    \sup_{j\in[n]}
    \norm{X_{-j}^\top\delta_{j,\mathrm{tr}}^{(0)}\hat\theta}_\infty 
    +
    O(1)
    \sup_{j\in[n]}
    \norm{
        X_{-j}^\top
        \left(
            \delta_j\hat\theta-\delta_{j,\mathrm{tr}}^{(0)}\hat\theta
        \right)
    }_2 .
\end{aligned}
\]

For the leading term,
\[
    \norm{X_{-j}^\top\delta_{j,\mathrm{tr}}^{(0)}\hat\theta}_\infty
    \le
    \frac{|s_j|}{n}
    \norm{X_{-j}^\top G_{-j}x_j}_\infty .
\]
The sequence-level score bound from Lemma~\ref{lem:uniform-prediction-moments}
gives
\[
    \forall q\ge1:\qquad
    \sup_{j\in[n]}\Norm{s_j}_{L^q}=O(q^2).
\]
Thus Lemma~\ref{lem:max-polynomial-moment-growth} gives
\[
    \left\|
        \sup_{j\in[n]} |s_j|
    \right\|_{L^r}
    =
    O_r((\log n)^2).
\]
Moreover, conditionally on \(X_{-j}\), the variables
\((x_\ell^\top G_{-j}x_j)_{\ell\ne j}\) are linear forms in \(x_j\), with
coefficients of Euclidean norm at most \(O(\sqrt n)\). Therefore,
\[
    \forall q\ge1:\qquad
    \sup_{\substack{j,\ell\in[n]\\ \ell\ne j}}
    \Norm{x_\ell^\top G_{-j}x_j}_{L^q}
    \le
    O(q\sqrt n).
\]
Applying Lemma~\ref{lem:max-polynomial-moment-growth} to the \(O(n^2)\) family
\((x_\ell^\top G_{-j}x_j)_{\ell\ne j}\), we obtain
\[
    \left\|
        \sup_{j\in[n]}
        \norm{X_{-j}^\top G_{-j}x_j}_\infty
    \right\|_{L^r}
    =
    O_r(\sqrt n\log n).
\]
Hence
\[
    \left\|
        \sup_{j\in[n]}
        \norm{X_{-j}^\top\delta_{j,\mathrm{tr}}^{(0)}\hat\theta}_\infty
    \right\|_{L^r}
    =
    O_r\!\left(
        \frac{(\log n)^3}{\sqrt n}
    \right).
\]

It remains to control the first-order remainder uniformly in \(j\). Under
\(\nabla^3\rho\equiv0\), the regularizer drift vanishes, and the proof of
Theorem~\ref{thm:loo-remainder-tilde} gives
\[
\begin{aligned}
    &\norm{
        X_{-j}^\top
        \left(
            \delta_j\hat\theta-\delta_{j,\mathrm{tr}}^{(0)}\hat\theta
        \right)
    }_2
    \le
    \frac{|s_j|}{n^2}
    \norm{X_{-j}^\top\bar G_jX_{-j}}_{\op}
    \norm{\Gamma_j}_F
    \norm{X_{-j}^\top G_{-j}x_j}_\infty .
\end{aligned}
\]
By Lemma~\ref{lem:curvature-resolvent},
\[
    \norm{X_{-j}^\top\bar G_jX_{-j}}_{\op}=O(n).
\]
Moreover, as in the proof of the Frobenius bound for \(\Gamma_j\),
\[
    \norm{\Gamma_j}_F
    \le
    O(1)\norm{X_{-j}^\top\delta_j\hat\theta}_2
    \le
    \frac{O(1)|s_j|}{n}
    \norm{X_{-j}^\top\bar G_j}_{\op}
    \norm{x_j}_2
    \le
    O(1)\frac{|s_j|}{\sqrt n}\norm{x_j}_2 .
\]
Using
\[
    \left\|
        \sup_{j\in[n]}|s_j|
    \right\|_{L^r}
    =
    O_r((\log n)^2),
\]
and
\[
    \left\|
        \sup_{j\in[n]}\norm{x_j}_2
    \right\|_{L^r}
    +
    \left\|
        \sup_{j\in[n]}
        \norm{X_{-j}^\top G_{-j}x_j}_\infty
    \right\|_{L^r}
    =
    O_r(\sqrt n\log n),
\]
H\"older's inequality yields
\[
\begin{aligned}
    &\left\|
        \sup_{j\in[n]}
        \norm{
            X_{-j}^\top
            \left(
                \delta_j\hat\theta-\delta_{j,\mathrm{tr}}^{(0)}\hat\theta
            \right)
        }_2
    \right\|_{L^r} 
    =
    O_r\!\left(
        \frac{(\log n)^6}{\sqrt n}
    \right).
\end{aligned}
\]
Combining the leading term and the first-order remainder gives the claim.
\end{proof}

\begin{lemma}[Fixed-direction and auxiliary-column full-resolvent leverage]
\label{lem:fixed-direction-full-resolvent-leverage}
Under Assumptions~\ref{ass:1dim}--\ref{ass:7third-order-curvature} and
\(\nabla^3\rho\equiv0\), let \(u\in\R^p\) be deterministic with
\(\norm{u}_2\le1\). Then, for every fixed integer \(r\ge1\),
\[
    \Norm{\norm{X^\top Gu}_\infty}_{L^r}
    =
    O_r\!\left((\log n)^6\right)
    \qquad \text{and}\qquad
    \Norm{\norm{X^\top Gx}_\infty}_{L^r}
    =
    O_r(\sqrt n\log n).
\]
\end{lemma}

\begin{proof}
We first prove the fixed-direction bound. Fix \(j\in[n]\). We use the
decomposition
\[
    G^{-1}
    =
    G_{-j}^{-1}
    +
    B_j^{\mathrm{bg}}
    +
    B_j^{\mathrm{loc}},
\]
where
\[
    B_j^{\mathrm{loc}}
    =
    \frac1nL''(x_j^\top\hat\theta)x_jx_j^\top
    =
    a_jx_jx_j^\top,
    \qquad
    a_j:=\frac1nL''(x_j^\top\hat\theta)\ge\frac{\kappa_L}{n},
\]
and
\[
    B_j^{\mathrm{bg}}
    =
    \frac1nX_{-j}\Delta_jX_{-j}^\top.
\]
Set
\[
    \widetilde G_j
    :=
    \left(
        G_{-j}^{-1}
        +
        B_j^{\mathrm{bg}}
    \right)^{-1}.
\]
Since \(\nabla^3\rho\equiv0\),
\[
    \widetilde G_j^{-1}
    =
    \nabla^2\Psi_X^{-j}(\hat\theta)
    \succeq
    \kappa I_p
    +
    \frac{\kappa_L}{n}X_{-j}X_{-j}^\top .
\]
By the Sherman--Morrison formula,
\[
    G
    =
    \widetilde G_j
    -
    \frac{
        a_j\widetilde G_jx_jx_j^\top\widetilde G_j
    }{
        1+a_jx_j^\top\widetilde G_jx_j
    }.
\]
Therefore
\[
    x_j^\top Gu
    =
    \frac{x_j^\top\widetilde G_ju}
    {1+a_jx_j^\top\widetilde G_jx_j}.
\]
Using the resolvent identity
\[
    \widetilde G_j-G_{-j}
    =
    -\widetilde G_jB_j^{\mathrm{bg}}G_{-j},
\]
we obtain
\[
    x_j^\top Gu
    =
    \frac{x_j^\top G_{-j}u}
    {1+a_jx_j^\top\widetilde G_jx_j}
    -
    \frac{
        x_j^\top\widetilde G_jB_j^{\mathrm{bg}}G_{-j}u
    }{
        1+a_jx_j^\top\widetilde G_jx_j
    }.
\]
Since the denominator is at least one,
\[
    |x_j^\top Gu|
    \le
    |x_j^\top G_{-j}u|
    +
    \frac{
        |x_j^\top\widetilde G_jB_j^{\mathrm{bg}}G_{-j}u|
    }{
        1+a_jx_j^\top\widetilde G_jx_j
    }.
\]

For the first term, conditionally on \(X_{-j}\), \(G_{-j}u\) is deterministic
and has Euclidean norm \(O(1)\). Hence, for every integer \(q\ge1\),
\[
    \sup_{j\in[n]}
    \Norm{x_j^\top G_{-j}u}_{L^q}
    =
    O(q).
\]
Lemma~\ref{lem:max-polynomial-moment-growth} gives
\[
    \left\|
        \sup_{j\in[n]} |x_j^\top G_{-j}u|
    \right\|_{L^r}
    =
    O_r(\log n).
\]

For the background correction, the curvature lower bound on
\(\widetilde G_j^{-1}\) gives
\[
    \widetilde G_jX_{-j}X_{-j}^\top\widetilde G_j
    \preceq
    \frac{n}{\kappa_L}\widetilde G_j.
\]
Thus
\[
    \norm{X_{-j}^\top\widetilde G_jx_j}_2^2
    =
    x_j^\top\widetilde G_jX_{-j}X_{-j}^\top\widetilde G_jx_j
    \le
    \frac{n}{\kappa_L}x_j^\top\widetilde G_jx_j.
\]
Since \(a_j\ge\kappa_L/n\), for \(t:=x_j^\top\widetilde G_jx_j\ge0\),
\[
    \frac{\norm{X_{-j}^\top\widetilde G_jx_j}_2}
    {1+a_jx_j^\top\widetilde G_jx_j}
    \le
    \sqrt{\frac{n}{\kappa_L}}
    \frac{\sqrt t}{1+(\kappa_L/n)t}
    \le
    O(n).
\]
Therefore
\[
\begin{aligned}
    &\frac{
        |x_j^\top\widetilde G_jB_j^{\mathrm{bg}}G_{-j}u|
    }{
        1+a_jx_j^\top\widetilde G_jx_j
    }
    =
    \frac1n
    \frac{
        |x_j^\top\widetilde G_jX_{-j}
        \Delta_jX_{-j}^\top G_{-j}u|
    }{
        1+a_jx_j^\top\widetilde G_jx_j
    }
    \le
    O(1)\,
    \norm{\Delta_j}_{\op}
    \norm{X_{-j}^\top G_{-j}u}_2 .
\end{aligned}
\]
Moreover,
\[
    \norm{X_{-j}^\top G_{-j}u}_2
    \le
    \norm{X_{-j}^\top G_{-j}}_{\op}\norm{u}_2
    =
    O(\sqrt n).
\]
Hence
\[
    \sup_{j\in[n]}
    \frac{
        |x_j^\top\widetilde G_jB_j^{\mathrm{bg}}G_{-j}u|
    }{
        1+a_jx_j^\top\widetilde G_jx_j
    }
    \le
    O(\sqrt n)
    \sup_{j\in[n]}\norm{\Delta_j}_{\op}.
\]
Lemma~\ref{lem:uniform-background-drift} gives
\[
    \left\|
        \sup_{j\in[n]}
        \frac{
            |x_j^\top\widetilde G_jB_j^{\mathrm{bg}}G_{-j}u|
        }{
            1+a_jx_j^\top\widetilde G_jx_j
        }
    \right\|_{L^r}
    =
    O_r((\log n)^6).
\]
Combining the leave-one-out term and the background correction proves
\[
    \Norm{\norm{X^\top Gu}_\infty}_{L^r}
    =
    O_r((\log n)^6).
\]

We now prove the auxiliary-column estimate. Since \(x\) is independent of
\(X\), we condition on \(X\). Then \(G\) and the vectors
\((Gx_j)_{j\in[n]}\) are fixed, and
\[
    (X^\top Gx)_j
    =
    x_j^\top Gx
    =
    x^\top Gx_j .
\]
By the curvature bound,
\[
    \max_{j\in[n]}\norm{Gx_j}_2
    \le
    \norm{GX}_{\op}
    =
    O(\sqrt n).
\]
Therefore, conditionally on \(X\), Lemma~\ref{lem:moments_of_x_i} gives, for
every integer \(q\ge1\),
\[
    \max_{j\in[n]}
    \Norm{x^\top Gx_j}_{L_x^q}
    \le
    O(q\sqrt n).
\]
Applying Lemma~\ref{lem:max-polynomial-moment-growth} conditionally on \(X\), we obtain
\[
    \Norm{
        \norm{X^\top Gx}_\infty
    }_{L_x^r}
    =
    \Norm{
        \max_{j\in[n]}|x^\top Gx_j|
    }_{L_x^r}
    \le
    O_r(\sqrt n\log n).
\]
The right-hand side is deterministic, so taking the \(L^r\)-norm over the
background variables proves
\[
    \Norm{\norm{X^\top Gx}_\infty}_{L^r}
    =
    O_r(\sqrt n\log n).
\]
\end{proof}

\subsection{Contracted active-background term}
\label{subsec:contracted-active-background}

It remains to bound the contribution to \(\omega_c\). 
The notation \(x\in\R^p\) denotes the active column, independent of the background
matrix \(X=(x_1,\ldots,x_n)\), and
\[
    G
    =
    \left(\nabla^2\Psi_X(\hat\theta)\right)^{-1}.
\]
The deterministic averaged leverage associated with the active column is denoted
by \(\gamma_x\), and we define
\[
    s_x^\zeta
    :=
    L'\!\left(
        \zeta_x(x^\top\hat\theta)
    \right),
    \qquad
    \zeta_x(t)+\gamma_xL'(\zeta_x(t))=t.
\]

We use a check accent for quantities evaluated at the checked background
\(X^{A_j}\):
\[
    \check G:=G(X^{A_j}),
    \qquad
    \check{\hat\theta}:=\hat\theta(X^{A_j}),
    \qquad
    \check s_x^\zeta:=s_x^\zeta(X^{A_j}),
    \qquad
    \check{\tilde\Gamma}_x:=\tilde\Gamma_x(X^{A_j}).
\]
The active column \(x\) is the same in the checked and unchecked expressions;
only the background is checked.

Define
\[
    y:=Gu,
    \qquad
    \check y:=\check G u.
\]
Let \(D\) be the diagonal matrix, indexed by the columns of \(X\), with entries
\[
    D_\ell
    :=
    L'''(x_\ell^\top\hat\theta)\,x_\ell^\top y,
\]
and set
\[
    Y
    :=
    GXDX^\top G.
\]
The checked matrices \(\check D\) and \(\check Y\) are defined analogously.
Recall that
\[
    \tilde\Gamma_x
    =\diag_{\ell\in[n]}-\frac{s_x^\zeta}{2n}
    L'''(x_\ell^\top\hat\theta)\,x_\ell^\top Gx,
\]
Therefore the explicit approximation can be written as
\[
    \mathfrak d_x
    =
    \mathfrak d_x^{(0)}
    +
    \mathfrak d_x^{(1)},
\]
where
\[
    \mathfrak d_x^{(0)}
    =
    -\frac{s_x^\zeta}{\sqrt n}\,y^\top x,
    \qquad
    \mathfrak d_x^{(1)}
    =
    -\frac{(s_x^\zeta)^2}{2n^{5/2}}\,x^\top Yx .
\]
The checked statistics
\(\check{\mathfrak d}_x,\check{\mathfrak d}_x^{(0)},\check{\mathfrak d}_x^{(1)}\)
are obtained by replacing \(s_x^\zeta,y,Y\) with
\(\check s_x^\zeta,\check y,\check Y\).

In this subsection, \(\E_x\), \(\Var_x\), and \(\Cov_x\) denote conditional
expectation, variance, and covariance with respect to the active column \(x\)
only, conditionally on all background variables. 
For a background-measurable matrix \(B\), and for fixed \(q\ge1\), write
\[
    \mathcal Q_q(B)
    :=
    \Norm{x^\top Bx}_{L_x^q}.
\]
We also write
\[
    \Sigma_x:=\E_x[xx^\top]
\]
for the second-moment matrix of the active column. By
Assumptions~\ref{ass:poincare-data} and~\ref{ass:3bounded-means},
\[
    \norm{\Sigma_x}_{\op}=O(1).
\]


For background-measurable vectors \(a,b\in\R^p\) and symmetric matrices
\(B,C\in\R^{p\times p}\), introduce the active contractions
\[
\begin{aligned}
    \mathcal C^{00}(a,b)
    &:=
    \Cov_x\!\left(
        s_x^\zeta\,a^\top x,
        \check s_x^\zeta\,b^\top x
    \right),                                                   \\
    \mathcal C^{01}(a,C)
    &:=
    \Cov_x\!\left(
        s_x^\zeta\,a^\top x,
        (\check s_x^\zeta)^2\,x^\top Cx
    \right),                                                   \\
    \mathcal C^{10}(B,b)
    &:=
    \Cov_x\!\left(
        (s_x^\zeta)^2\,x^\top Bx,
        \check s_x^\zeta\,b^\top x
    \right),                                                   \\
    \mathcal C^{11}(B,C)
    &:=
    \Cov_x\!\left(
        (s_x^\zeta)^2\,x^\top Bx,
        (\check s_x^\zeta)^2\,x^\top Cx
    \right).
\end{aligned}
\]
Then

\begin{align}\label{borne_Cov_x}
    \Cov_x\!\left(
        \mathfrak d_x,
        \check{\mathfrak d}_x
    \right)
    &=
    \frac1n\mathcal C^{00}(y,\check y)
    +
    \frac1{2n^3}\mathcal C^{01}(y,\check Y)
    +
    \frac1{2n^3}\mathcal C^{10}(Y,\check y)
    +
    \frac1{4n^5}\mathcal C^{11}(Y,\check Y).
\end{align}

\begin{lemma}[Square-summed derivatives of the active score]
\label{lem:square-summed-active-score-derivatives}
Under Assumptions~\ref{ass:1dim}--\ref{ass:7third-order-curvature}, for every
fixed integer \(r\ge1\),
\[
    \left\|
        \left(
            \sum_{\ell=1}^n
            \left(
                \norm{\mathbb D_\ell[s_x^\zeta]}^*
            \right)^2
        \right)^{1/2}
    \right\|_{L^r}
    =
    O_r(1).
\]
The same estimate holds for the checked score:
\[
    \left\|
        \left(
            \sum_{\ell=1}^n
            \left(
                \norm{\mathbb D'_\ell[\check s_x^\zeta]}^*
            \right)^2
        \right)^{1/2}
    \right\|_{L^r}
    =
    O_r(1).
\]
\end{lemma}

\begin{proof}
By Lemma~\ref{lem:derivative-zeta}, in the generic background notation,
\[
    \mathbb D_\ell[s_x^\zeta][h]
    =
    \frac{
        L''(\zeta_x(x^\top\hat\theta))
    }{
        1+\gamma_xL''(\zeta_x(x^\top\hat\theta))
    }
    x^\top\mathbb D_\ell[\hat\theta][h].
\]
The prefactor is uniformly bounded. Hence
\[
    \norm{\mathbb D_\ell[s_x^\zeta]}^*
    \le
    O(1)\norm{x^\top\mathbb D_\ell[\hat\theta]}^*.
\]
Applying Lemma~\ref{lem:square-summed-background-sensitivity} conditionally on
the active column \(x\), with \(z=x\), gives
\[
\begin{aligned}
    \left\|
        \left(
            \sum_{\ell=1}^n
            \left(
                \norm{x^\top\mathbb D_\ell[\hat\theta]}^*
            \right)^2
        \right)^{1/2}
    \right\|_{L^r}
    &\le
    \frac{O_r(1)}{\sqrt n}
    \Norm{\norm{x}_2}_{L^{4r}}
    =
    O_r(1),
\end{aligned}
\]
using Lemma~\ref{lem:moments_of_x_i}. The checked estimate is identical.
\end{proof}
\begin{lemma}[Background matrix bounds]
\label{lem:background-matrix-bounds_bounded}
Under Assumptions~\ref{ass:1dim}--\ref{ass:7third-order-curvature} and
\(\nabla^3\rho\equiv0\), for every fixed integers \(q,r\ge1\),
\[
    \Norm{\norm{y}_2}_{L^r}=O_r(1),
    \qquad
    \Norm{\mathcal Q_q(Y)}_{L^r}
    =
    O_{q,r}\!\left(n^2(\log n)^6\right).
\]
The same estimates hold for \(\check y\) and \(\check Y\).
\end{lemma}

\begin{proof}
Recall that
\[
    y=Gu,
    \qquad
    D=\diag_{m\in[n]}
    \left(
        L'''(x_m^\top\hat\theta)\,x_m^\top y
    \right),
    \qquad
    Y=GXDX^\top G.
\]
Since \(\norm{G}_{\op}=O(1)\) and \(\norm{u}_2\le1\),
\[
    \norm{y}_2=O(1).
\]

Set
\[
    z:=X^\top Gx.
\]
Then
\[
    x^\top Yx=z^\top Dz.
\]
By the boundedness of \(L'''\),
\[
    \norm{D}_{\op}
    \le
    O(1)\norm{X^\top y}_\infty .
\]
Moreover, Lemma~\ref{lem:fixed-direction-full-resolvent-leverage} gives
\[
    \Norm{\norm{X^\top y}_\infty}_{L^r}
    =
    O_r((\log n)^6),
\]
while the generic cross-leverage estimate gives, conditionally on the
background,
\[
    \Norm{\norm{X^\top Gx}_2}_{L_x^{2q}}
    =
    O_q(n).
\]
Therefore
\[
\begin{aligned}
    \mathcal Q_q(Y)
    &=
    \Norm{z^\top Dz}_{L_x^q}
    \le
    \norm{D}_{\op}\,
    \Norm{\norm{z}_2^2}_{L_x^q}                         \\
    &\le
    O_q(n^2)\norm{X^\top y}_\infty .
\end{aligned}
\]
Taking the \(L^r\)-norm over the background gives
\[
    \Norm{\mathcal Q_q(Y)}_{L^r}
    =
    O_{q,r}\!\left(n^2(\log n)^6\right).
\]
The checked estimate is identical, since the checked background has the same
law and satisfies the same bounds.
\end{proof}

\begin{lemma}[Differentiated active-score contractions]
\label{lem:differentiated-active-score-contractions}
Under Assumptions~\ref{ass:1dim}--\ref{ass:7third-order-curvature} and
\(\nabla^3\rho\equiv0\), for every fixed integer \(r\ge1\),
\begin{itemize}
    \item 
$\begin{aligned}
    &\left\|
        \left(
            \sum_{\ell=1}^n
            \left(
                \norm{
                    \Cov_x\!\left(
                        \mathbb D_\ell[s_x^\zeta]\,y^\top x,
                        \check s_x^\zeta\,\check y^\top x
                    \right)
                }^*
            \right)^2
        \right)^{1/2}
    \right\|_{L^r}
    =
    O_r\!\left(\frac1{\sqrt n}\right).
\end{aligned}$
\item
$\begin{aligned}
    &\left\|
        \left(
            \sum_{\ell=1}^n
            \left(
                \norm{
                    \Cov_x\!\left(
                        \mathbb D_\ell[s_x^\zeta]\,y^\top x,
                        (\check s_x^\zeta)^2\,x^\top\check Yx
                    \right)
                }^*
            \right)^2
        \right)^{1/2}
    \right\|_{L^r}
    =
O_r\!\left(n^{3/2}(\log n)^6\right).
\end{aligned}$
\item 
$\begin{aligned}
    &\left\|
        \left(
            \sum_{\ell=1}^n
            \left(
                \norm{
                    \Cov_x\!\left(
                        2s_x^\zeta\,\mathbb D_\ell[s_x^\zeta]\,x^\top Yx,
                        (\check s_x^\zeta)^2\,x^\top\check Yx
                    \right)
                }^*
            \right)^2
        \right)^{1/2}
    \right\|_{L^r}
=
O_r\!\left(n^{7/2}(\log n)^{12}\right).
\end{aligned}$
\end{itemize}
The same bound holds with the checked and unchecked roles interchanged, in particular for the score derivatives appearing in \(\mathcal C^{10}(Y,\check y)\).
\end{lemma}

\begin{proof}
We prove the unchecked estimates; the checked estimates are identical. By
Lemma~\ref{lem:derivative-zeta},
\[
    \mathbb D_\ell[s_x^\zeta][h]
    =
    \frac{
        L''(\zeta_x(x^\top\hat\theta))
    }{
        1+\gamma_xL''(\zeta_x(x^\top\hat\theta))
    }
    x^\top\mathbb D_\ell[\hat\theta][h].
\]
The prefactor is bounded.

For the \(00\)-term, for every \(\norm{h}_2\le1\),
\[
\begin{aligned}
    &\Cov_x\!\left(
        \mathbb D_\ell[s_x^\zeta][h]\,y^\top x,
        \check s_x^\zeta\,\check y^\top x
    \right)                                                \\
    &\quad =
    \left(
        \E_x\!\left[
            \frac{
                L''(\zeta_x(x^\top\hat\theta))
            }{
                1+\gamma_xL''(\zeta_x(x^\top\hat\theta))
            }
            x\,(y^\top x)
            \left(
                \check s_x^\zeta\,\check y^\top x
                -
                \E_x[\check s_x^\zeta\,\check y^\top x]
            \right)
        \right]
    \right)^\top
    \mathbb D_\ell[\hat\theta][h].
\end{aligned}
\]
Consequently, after taking the supremum over \(\norm{h}_2\le1\) and summing in
\(\ell\), Lemma~\ref{lem:square-summed-background-sensitivity} gives
\[
\begin{aligned}
    &\left\|
        \left(
            \sum_{\ell=1}^n
            \left(
                \norm{
                    \Cov_x\!\left(
                        \mathbb D_\ell[s_x^\zeta]\,y^\top x,
                        \check s_x^\zeta\,\check y^\top x
                    \right)
                }^*
            \right)^2
        \right)^{1/2}
    \right\|_{L^r}                                      \\
    &\qquad\le
    \frac{O_r(1)}{\sqrt n}
    \left\|
        \norm{
            \E_x\!\left[
                \frac{
                    L''(\zeta_x(x^\top\hat\theta))
                }{
                    1+\gamma_xL''(\zeta_x(x^\top\hat\theta))
                }
                x\,(y^\top x)
                \left(
                    \check s_x^\zeta\,\check y^\top x
                    -
                    \E_x[\check s_x^\zeta\,\check y^\top x]
                \right)
            \right]
        }_2
    \right\|_{L^{4r}} 
\end{aligned}
\]
We bound the vector inside the last display by duality. For every
background-measurable \(v\in\R^p\) with \(\norm{v}_2\le1\), conditional
Cauchy--Schwarz and the active linear moment bounds give
\[
\begin{aligned}
    &\left|
        \E_x\!\left[
            (v^\top x)(y^\top x)
            \left(
                \check s_x^\zeta\,\check y^\top x
                -
                \E_x[\check s_x^\zeta\,\check y^\top x]
            \right)
        \right]
    \right|
    \le
    O(1)\norm{y}_2\norm{\check y}_2
    \Norm{\check s_x^\zeta}_{L_x^4}.
\end{aligned}
\]
Taking the supremum over \(\norm{v}_2\le1\), then the \(L^{4r}\)-norm over the
background, yields
\[
    \left\|
        \norm{
            \E_x\!\left[
                \frac{
                    L''(\zeta_x(x^\top\hat\theta))
                }{
                    1+\gamma_xL''(\zeta_x(x^\top\hat\theta))
                }
                x\,(y^\top x)
                \left(
                    \check s_x^\zeta\,\check y^\top x
                    -
                    \E_x[\check s_x^\zeta\,\check y^\top x]
                \right)
            \right]
        }_2
    \right\|_{L^{4r}}
    =
    O_r(1).
\]
Consequently the \(00\)-term is \(O_r(n^{-1/2})\).
By conditional Cauchy--Schwarz, Lemmas~\ref{lem:moments_of_x_i},~~\ref{lem:def_zeta} and~\ref{lem:background-matrix-bounds_bounded}.
The checked derivative gives the same estimate.

For the one-quadratic active slot, the same duality argument gives
\[
    \norm{
        \E_x\!\left[
            \frac{
                L''(\zeta_x(x^\top\hat\theta))
            }{
                1+\gamma_xL''(\zeta_x(x^\top\hat\theta))
            }
            x\,(y^\top x)
            \left(
                (\check s_x^\zeta)^2x^\top\check Yx
                -
                \E_x[(\check s_x^\zeta)^2x^\top\check Yx]
            \right)
        \right]
    }_2
    \le
    O(1)\norm{y}_2
    \mathcal Q_4(\check Y)
    \Norm{\check s_x^\zeta}_{L_x^8}^2 .
\]
Applying Lemmas~\ref{lem:square-summed-background-sensitivity} and~\ref{lem:background-matrix-bounds_bounded} as above gives
\[
    \left\|
        \left(
            \sum_{\ell=1}^n
            \left(
                \norm{
                    \Cov_x\!\left(
                        \mathbb D_\ell[s_x^\zeta]\,y^\top x,
                        (\check s_x^\zeta)^2\,x^\top\check Yx
                    \right)
                }^*
            \right)^2
        \right)^{1/2}
    \right\|_{L^r}
    \leq O_r((\log n)^6n^{3/2}).
\]
The same argument gives the corresponding bound for
\(\mathcal C^{10}(Y,\check y)\).

For the two-quadratic active slot, duality gives
\[
    \norm{
        \E_x\!\left[
            x\, (x^\top Yx)
            \left(
                (\check s_x^\zeta)^2x^\top\check Yx
                -
                \E_x[(\check s_x^\zeta)^2x^\top\check Yx]
            \right)
        \right]
    }_2
    \le
    O(1)\mathcal Q_8(Y)\mathcal Q_8(\check Y)
    \Norm{\check s_x^\zeta}_{L_x^{16}}^2 .
\]
By Lemma~\ref{lem:background-matrix-bounds_bounded},
\[
    \Norm{\mathcal Q_8(Y)}_{L^{8r}}
    +
    \Norm{\mathcal Q_8(\check Y)}_{L^{8r}}
    =
    O_r\!\left(n^2(\log n)^6\right).
\]
Hence
\[
\begin{aligned}
    &\left\|
        \left(
            \sum_{\ell=1}^n
            \left(
                \norm{
                    \Cov_x\!\left(
                        2s_x^\zeta\,\mathbb D_\ell[s_x^\zeta]\,
                        x^\top Yx,
                        (\check s_x^\zeta)^2x^\top\check Yx
                    \right)
                }^*
            \right)^2
        \right)^{1/2}
    \right\|_{L^r}
    =
    O_r\!\left(n^{7/2}(\log n)^{12}\right).
\end{aligned}
\]

This proves all the stated estimates.
\end{proof}

To bound the terms appearing when differentiating the terms $y, \check y, Y, \check Y$, we make use of the following preliminary Lemma.
\begin{lemma}[Active-column contractions]
\label{lem:active-column-contractions-detailed}
Under Assumptions~\ref{ass:1dim}--\ref{ass:7third-order-curvature}, for every
fixed integer \(r\ge1\), there exists a nonnegative background-measurable
random variable \(R\), with
\[
    \Norm{R}_{L^r}=O_r(1),
\]
such that, for all background-measurable vectors \(a,b\in\R^p\) and symmetric
matrices \(B,C\in\R^{p\times p}\),
\begin{align*}
    &|\mathcal C^{00}(a,b)|
    \le
    R\norm{a}_2\norm{b}_2,
&&
    |\mathcal C^{01}(a,C)|
    \le
    R\norm{a}_2\,\mathcal Q_4(C),
\\
    &|\mathcal C^{10}(B,b)|
    \le
    R\,\mathcal Q_4(B)\norm{b}_2,
&&
    |\mathcal C^{11}(B,C)|
    \le
    R\,\mathcal Q_4(B)\mathcal Q_4(C).
\end{align*}
\end{lemma}

\begin{proof}
Set
\[
    R
    :=
    C
    \left(
        1+\Norm{s_x^\zeta}_{L_x^8}
        +
        \Norm{\check s_x^\zeta}_{L_x^8}
    \right)^4.
\]
By Lemma~\ref{lem:def_zeta}, and the same argument applied to the checked
background,
\[
    \Norm{R}_{L^r}=O_r(1).
\]
For the linear-linear contraction, conditional Cauchy--Schwarz gives
\[
\begin{aligned}
    |\mathcal C^{00}(a,b)|
    &\le
    \Norm{s_x^\zeta\,a^\top x}_{L_x^2}
    \Norm{\check s_x^\zeta\,b^\top x}_{L_x^2}        \\
    &\le
    \Norm{s_x^\zeta}_{L_x^4}
    \Norm{\check s_x^\zeta}_{L_x^4}
    \Norm{a^\top x}_{L_x^4}
    \Norm{b^\top x}_{L_x^4}
    \le
    R\norm{a}_2\norm{b}_2 .
\end{aligned}
\]
For the mixed contractions,
\[
\begin{aligned}
    |\mathcal C^{01}(a,C)|
    &\le
    \Norm{s_x^\zeta\,a^\top x}_{L_x^2}
    \Norm{(\check s_x^\zeta)^2x^\top Cx}_{L_x^2}
    \le
    R\norm{a}_2\,\mathcal Q_4(C),
\end{aligned}
\]
and the proof of the bound for \(\mathcal C^{10}\) is identical. Finally,
\[
\begin{aligned}
    |\mathcal C^{11}(B,C)|
    &\le
    \Norm{(s_x^\zeta)^2x^\top Bx}_{L_x^2}
    \Norm{(\check s_x^\zeta)^2x^\top Cx}_{L_x^2}
    \le
    R\,\mathcal Q_4(B)\mathcal Q_4(C).
\end{aligned}
\]
This proves the lemma.
\end{proof}
\begin{lemma}[Square-summed derivatives of \(y\)]
\label{lem:square-summed-y-derivative}
Under Assumptions~\ref{ass:1dim}--\ref{ass:7third-order-curvature} and
\(\nabla^3\rho\equiv0\), for every fixed integer \(r\ge1\), and for every background-measurable vector \(v=v(X)\),
\[
    \left\|
        \left(
            \sum_{\ell=1}^n
            \left(
                \norm{v^\top\mathbb D_\ell[y]}^*
            \right)^2
        \right)^{1/2}
    \right\|_{L^r}
    \le
    \frac{O_r((\log n)^6)}{\sqrt n}
    \Norm{\norm{v}_2}_{L^{4r}}.
\]
The same estimates hold for the checked variables, with
\(\mathbb D'_\ell\) in place of \(\mathbb D_\ell\).
\end{lemma}

\begin{proof}
Since \(y=Gu\), for \(\norm{h}_2\le1\),
\[
    \mathbb D_\ell[y][h]
    =
    -G
    \left(
        \mathcal L_\ell[h]
        +
        \mathcal B_\ell[h]
    \right)y,
\]
where
\[
\begin{aligned}
    \mathcal L_\ell[h]
    &:=
    \frac1n
    L''(x_\ell^\top\hat\theta)
    \left(
        hx_\ell^\top+x_\ell h^\top
    \right)
    +
    \frac1n
    L'''(x_\ell^\top\hat\theta)
    (h^\top\hat\theta)x_\ell x_\ell^\top,
\end{aligned}
\]
and
\[
\begin{aligned}
    \mathcal B_\ell[h]
    &:=
    \frac1n
    \sum_{m=1}^n
    L'''(x_m^\top\hat\theta)
    \left(
        x_m^\top\mathbb D_\ell[\hat\theta][h]
    \right)
    x_m x_m^\top
    +
    \nabla^3\rho(\hat\theta)
    \big[
        \mathbb D_\ell[\hat\theta][h]
    \big].
\end{aligned}
\]
Under \(\nabla^3\rho\equiv0\), the last term vanishes.

For the local part,
\[
\begin{aligned}
    \left|
        v^\top G\mathcal L_\ell[h]y
    \right|
    &\le
    \frac{O(1)}{n}
    \big[
        \norm{v}_2\,|x_\ell^\top y|
        +
        |x_\ell^\top Gv|\,\norm{y}_2
    +
        \norm{\hat\theta}_2
        |x_\ell^\top Gv|\,
        |x_\ell^\top y|
    \big].
\end{aligned}
\]
Therefore
\[
\begin{aligned}
    &\left(
        \sum_{\ell=1}^n
        \left(
            \norm{v^\top G\mathcal L_\ell y}^*
        \right)^2
    \right)^{1/2}                                           \\
    &\qquad\qquad\le
    \frac{O(1)}{n}
    \left[
        \norm{v}_2\norm{X^\top y}_2
        +
        \norm{y}_2\norm{X^\top Gv}_2
        +
        \norm{\hat\theta}_2
        \norm{X^\top y}_\infty
        \norm{X^\top Gv}_2
    \right].
\end{aligned}
\]
Using
\[
    \norm{y}_2=O(1),
    \qquad
    \norm{X^\top y}_2=O(\sqrt n),
    \qquad
    \norm{X^\top Gv}_2\le O(\sqrt n)\norm{v}_2,
\]
and
\[
    \Norm{\norm{X^\top y}_\infty}_{L^r}
    =
    O_r((\log n)^6)
\]
by Lemma~\ref{lem:fixed-direction-full-resolvent-leverage}, we get
\[
    \left\|
        \left(
            \sum_{\ell=1}^n
            \left(
                \norm{v^\top G\mathcal L_\ell y}^*
            \right)^2
        \right)^{1/2}
    \right\|_{L^r}
    \le
    \frac{O_r((\log n)^6)}{\sqrt n}
    \Norm{\norm{v}_2}_{L^{4r}} .
\]

For the background part, using the displayed expression of
\(\mathcal B_\ell[h]\),
\[
\begin{aligned}
    v^\top G\mathcal B_\ell[h]y
    &=
    \frac1n
    \sum_{m=1}^n
    (x_m^\top Gv)
    L'''(x_m^\top\hat\theta)
    (x_m^\top\mathbb D_\ell[\hat\theta][h])
    (x_m^\top y).
\end{aligned}
\]
Using
\[
    \mathbb D_\ell[\hat\theta][h]
    =
    -\frac1nG
    \left[
        L''(x_\ell^\top\hat\theta)(h^\top\hat\theta)x_\ell
        +
        s_\ell h
    \right],
    \qquad
    s_\ell:=L'(x_\ell^\top\hat\theta),
\]
and writing \(F=\diag_m(L'''(x_m^\top\hat\theta))\), we obtain
\[
\begin{aligned}
    v^\top G\mathcal B_\ell[h]y
    &=
    -\frac{L''(x_\ell^\top\hat\theta)(h^\top\hat\theta)}{n^2}
    (X^\top Gv\odot X^\top y)^\top
    F
    X^\top Gx_\ell                                      \\
    &\quad
    -\frac{s_\ell}{n^2}
    (X^\top Gv\odot X^\top y)^\top
    F
    X^\top Gh .
\end{aligned}
\]
Hence
\[
\begin{aligned}
    &\left(
        \sum_{\ell=1}^n
        \left(
            \norm{
                v^\top G\mathcal B_\ell y
            }^*
        \right)^2
    \right)^{1/2}                                      \\
    &\quad\le
    \frac{O(1)}{n^2}
    \left[
        \norm{\hat\theta}_2
        \norm{
            X^\top GX
            F
            (X^\top Gv\odot X^\top y)
        }_2
        +
        \left(\sum_{\ell=1}^n s_\ell^2\right)^{1/2}
        \norm{
            GX
            F
            (X^\top Gv\odot X^\top y)
        }_2
    \right].
\end{aligned}
\]
Since \(L'''\) is bounded,
\[
    \norm{
        F
        (X^\top Gv\odot X^\top y)
    }_2
    \le
    O(1)\norm{X^\top y}_\infty\norm{X^\top Gv}_2
    \le
    O(\sqrt n)\norm{X^\top y}_\infty\norm{v}_2.
\]
Therefore
\[
    \norm{
        X^\top GX
        F
        (X^\top Gv\odot X^\top y)
    }_2
    \le
    O(n^{3/2})\norm{X^\top y}_\infty\norm{v}_2,
\]
and
\[
    \norm{
        GX
        F
        (X^\top Gv\odot X^\top y)
    }_2
    \le
    O(n)\norm{X^\top y}_\infty\norm{v}_2.
\]
Using
\[
    \Norm{\norm{X^\top y}_\infty}_{L^r}
    =
    O_r((\log n)^6),
    \qquad
    \Norm{\norm{\hat\theta}_2}_{L^r}=O_r(1),
    \qquad
    \Norm{
        \left(\sum_{\ell=1}^n s_\ell^2\right)^{1/2}
    }_{L^r}
    =
    O_r(\sqrt n),
\]
we obtain
\[
    \left\|
        \left(
            \sum_{\ell=1}^n
            \left(
                \norm{
                    v^\top G\mathcal B_\ell y
                }^*
            \right)^2
        \right)^{1/2}
    \right\|_{L^r}
    \le
    \frac{O_r((\log n)^6)}{\sqrt n}
    \Norm{\norm{v}_2}_{L^{4r}} .
\]
This proves the first square-summed derivative estimate.

\end{proof}

\begin{lemma}[Derivative of the minimizer]
\label{lem:generic-derivative-minimizer}
Under Assumptions~\ref{ass:1dim}--\ref{ass:5regularity}, for every fixed
integer \(r\ge1\),
\[
    \sup_{j\in[n]}
    \Norm{
        \norm{\mathbb D_j[\hat\theta]}_2^*
    }_{L^r}
    =
    O_r\!\left(\frac1{\sqrt n}\right),
\]
and
\[
    \sup_{j\in[n]}
    \Norm{
        \norm{X^\top\mathbb D_j[\hat\theta]}_2^*
    }_{L^r}
    =
    O_r(1).
\]
Consequently,
\[
    \left\|
        \left(
            \sum_{j=1}^n
            \left(
                \norm{\mathbb D_j[\hat\theta]}_2^*
            \right)^2
        \right)^{1/2}
    \right\|_{L^r}
    =
    O_r(1),
\]
and
\[
    \left\|
        \left(
            \sum_{j=1}^n
            \left(
                \norm{X^\top\mathbb D_j[\hat\theta]}_2^*
            \right)^2
        \right)^{1/2}
    \right\|_{L^r}
    =
    O_r(\sqrt n).
\]
\end{lemma}

\begin{proof}
For \(\norm{h}_2\le1\), the derivative identity gives
\[
    \mathbb D_j[\hat\theta][h]
    =
    -\frac1nG
    \left[
        L''(x_j^\top\hat\theta)(h^\top\hat\theta)x_j
        +
        s_j h
    \right].
\]
Therefore
\[
    \norm{\mathbb D_j[\hat\theta][h]}_2
    \le
    \frac{O(1)}{n}
    \left(
        \norm{\hat\theta}_2\norm{x_j}_2
        +
        |s_j|
    \right).
\]
The first estimate follows from Theorem~\ref{thm:direct-hat-bound},
Lemma~\ref{lem:moments_of_x_i}, and Lemma~\ref{lem:uniform-prediction-moments}.

Multiplying the same identity by \(X^\top\), we get
\[
    X^\top\mathbb D_j[\hat\theta][h]
    =
    -\frac1n
    \left[
        L''(x_j^\top\hat\theta)(h^\top\hat\theta)X^\top Gx_j
        +
        s_j X^\top Gh
    \right].
\]
Using
\[
    \norm{X^\top Gh}_2
    \le
    O(\sqrt n)\norm{h}_2,
    \qquad
    \norm{X^\top Gx_j}_2
    \le
    \norm{X^\top G}_{\op}\norm{x_j}_2
    =
    O(\sqrt n)\norm{x_j}_2,
\]
we obtain
\[
    \norm{X^\top\mathbb D_j[\hat\theta]}_2^*
    \le
    \frac{O(1)}{\sqrt n}
    \left(
        \norm{\hat\theta}_2\norm{x_j}_2
        +
        |s_j|
    \right).
\]
The second estimate follows again from Theorem~\ref{thm:direct-hat-bound},
Lemma~\ref{lem:moments_of_x_i}, and
Lemma~\ref{lem:uniform-prediction-moments}.

Finally, by Minkowski's inequality in \(L^r(\ell_2)\),
\[
\begin{aligned}
    \left\|
        \left(
            \sum_{j=1}^n
            \left(
                \norm{\mathbb D_j[\hat\theta]}_2^*
            \right)^2
        \right)^{1/2}
    \right\|_{L^r}
    &\le
    \left(
        \sum_{j=1}^n
        \Norm{
            \norm{\mathbb D_j[\hat\theta]}_2^*
        }_{L^r}^2
    \right)^{1/2}
    =
    O_r(1),
\end{aligned}
\]
and similarly
\[
\begin{aligned}
    \left\|
        \left(
            \sum_{j=1}^n
            \left(
                \norm{X^\top\mathbb D_j[\hat\theta]}_2^*
            \right)^2
        \right)^{1/2}
    \right\|_{L^r}
    &\le
    \left(
        \sum_{j=1}^n
        \Norm{
            \norm{X^\top\mathbb D_j[\hat\theta]}_2^*
        }_{L^r}^2
    \right)^{1/2}
    =
    O_r(\sqrt n).
\end{aligned}
\]
\end{proof}

\begin{lemma}[Background-resolvent active quadratic estimate]
\label{lem:background-resolvent-active-quadratic}
Under Assumptions~\ref{ass:1dim}--\ref{ass:7third-order-curvature} and
\(\nabla^3\rho\equiv0\), for every fixed integers \(q,r\ge1\),
\[
\begin{aligned}
    &\left\|
        \left(
            \sum_{\ell=1}^n
            \left[
                \sup_{\norm{h}_2\le1}
                \Norm{
                    x^\top GX D X^\top G\mathcal B_\ell[h]Gx
                }_{L_x^q}
            \right]^2
        \right)^{1/2}
    \right\|_{L^r}
    =
    O_{q,r}\!\left(n^{3/2}(\log n)^6\right).
\end{aligned}
\]
where for \(\ell\in[n]\) and \(h\in\mathbb R^p\), we denoted,
\[
    \mathcal B_\ell[h]
    :=
    \frac1nX\Lambda_\ell(h)X^\top,
\qquad \text{with}\ 
    \Lambda_\ell(h)
    := \diag_{m\in [n]}
    L'''(x_m^\top\hat\theta)\,
    x_m^\top\mathbb D_\ell[\hat\theta][h].
\]
\end{lemma}

\begin{proof}
Let us recall the notations:
\begin{align*}
    y=Gu,
    \qquad
    D=\diag_{m\in[n]}
    \left(
        L'''(x_m^\top\hat\theta)\,x_m^\top y
    \right),
\end{align*}
we have
\[
    x^\top GX D X^\top G\mathcal B_\ell[h]Gx
    =
    x^\top A_\ell[h]x,
\]
where
\[
    A_\ell[h]
    :=
    \frac1n
    GXD X^\top GX\Lambda_\ell(h)X^\top G .
\]
The conditional quadratic-form bound, applied conditionally on the background
and to the symmetric part of \(A_\ell[h]\), gives
\[
    \Norm{x^\top A_\ell[h]x}_{L_x^q}
    \le
    O_q\left(
        \norm{A_\ell[h]}_F
        +
        \left|
            \tr(A_\ell[h]\Sigma_x)
        \right|
    \right).
\]
It is therefore enough to bound the Frobenius and trace contributions.

\begin{itemize}
    \item For the Frobenius contribution, using
\[
    \norm{GX}_{\op}=O(\sqrt n),
    \qquad
    \norm{X^\top GX}_{\op}=O(n),
    \qquad
    \norm{X^\top G}_{\op}=O(\sqrt n),
\]
we get, uniformly over \(\norm{h}_2\le1\),
\[
\begin{aligned}
    \norm{A_\ell[h]}_F
    &\le
    \frac1n
    \norm{GX}_{\op}
    \norm{D}_{\op}
    \norm{X^\top GX}_{\op}
    \norm{\Lambda_\ell(h)}_F
    \norm{X^\top G}_{\op}
    \le
    O(n)\norm{D}_{\op}\norm{\Lambda_\ell(h)}_F .
\end{aligned}
\]
Since \(L'''\) is bounded,
\[
    \norm{\Lambda_\ell(h)}_F
    \le
    O(1)\norm{X^\top\mathbb D_\ell[\hat\theta][h]}_2.
\]
Thus
\[
\begin{aligned}
    &\left\|
        \left(
            \sum_{\ell=1}^n
            \left[
                \sup_{\norm{h}_2\le1}
                \norm{A_\ell[h]}_F
            \right]^2
        \right)^{1/2}
    \right\|_{L^r}
    \le
    O(n)
    \left\|
        \norm{D}_{\op}
        \left(
            \sum_{\ell=1}^n
            \left(
                \norm{X^\top\mathbb D_\ell[\hat\theta]}_2^*
            \right)^2
        \right)^{1/2}
    \right\|_{L^r}.
\end{aligned}
\]
By Lemmas~\ref{lem:fixed-direction-full-resolvent-leverage} and
\ref{lem:generic-derivative-minimizer},
\[
    \Norm{\norm{D}_{\op}}_{L^{2r}}
    =
    O_r((\log n)^6),
\]
and
\[
    \left\|
        \left(
            \sum_{\ell=1}^n
            \left(
                \norm{X^\top\mathbb D_\ell[\hat\theta]}_2^*
            \right)^2
        \right)^{1/2}
    \right\|_{L^{2r}}
    =
    O_r(\sqrt n).
\]
Therefore
\[
\begin{aligned}
    &\left\|
        \left(
            \sum_{\ell=1}^n
            \left[
                \sup_{\norm{h}_2\le1}
                \norm{A_\ell[h]}_F
            \right]^2
        \right)^{1/2}
    \right\|_{L^r}
    =
    O_r\!\left(n^{3/2}(\log n)^6\right).
\end{aligned}
\]
\item We now treat the trace contribution. By cyclicity of the trace,
\[
\begin{aligned}
    \tr(A_\ell[h]\Sigma_x)
    &=
    \frac1n
    \tr\!\left(
        \Lambda_\ell(h)
        X^\top G\Sigma_xGXD X^\top GX
    \right).
\end{aligned}
\]
Let
\[
    F:=\diag_{m\in[n]}
    \left(
        L'''(x_m^\top\hat\theta)
    \right),
\qquad \text{and}\qquad 
    c
    :=
    \diag\!\left(
        X^\top G\Sigma_xGXD X^\top GX
    \right)
    \in\mathbb R^n.
\]
Since $\diag(\Lambda_\ell(h))
    =
    F X^\top\mathbb D_\ell[\hat\theta][h]$, we have
\[
    \tr(A_\ell[h]\Sigma_x)
    =
    \frac1n
    (XFc)^\top
    \mathbb D_\ell[\hat\theta][h].
\]
Using the explicit derivative formula
\[
    \mathbb D_\ell[\hat\theta][h]
    =
    -\frac1nG
    \left[
        L''(x_\ell^\top\hat\theta)
        (h^\top\hat\theta)x_\ell
        +
        s_\ell h
    \right],
\]
we obtain, uniformly over \(\norm{h}_2\le1\),
\[
\begin{aligned}
    \left|
        (XFc)^\top\mathbb D_\ell[\hat\theta][h]
    \right|
    &\le
    \frac{O(1)}{n}
    \left(
        \norm{\hat\theta}_2
        |e_\ell^\top X^\top GXFc|
        +
        |s_\ell|\norm{GXFc}_2
    \right).
\end{aligned}
\]
Therefore
\[
\begin{aligned}
    &\left(
        \sum_{\ell=1}^n
        \left[
            \sup_{\norm{h}_2\le1}
            \left|
                \tr(A_\ell[h]\Sigma_x)
            \right|
        \right]^2
    \right)^{1/2}
    \le
    \frac{O(1)}{n^2}
    \left(
        \norm{\hat\theta}_2
        \norm{X^\top GXFc}_2
        +
        \norm{s}_2\norm{GXFc}_2
    \right).
\end{aligned}
\]
Since \(|F|_{\op}\leq |L'''|_{\infty}\leq O(1)\),
\[
    \norm{GXFc}_2
    \le
    \norm{GX}_{\op}\norm{c}_2
    =
    O(\sqrt n)\norm{c}_2,
\qquad
\text{and}
\qquad
    \norm{X^\top GXFc}_2
    =
    O(n)\norm{c}_2.
\]
Moreover,
\[
    \Norm{\norm{\hat\theta}_2}_{L^{2r}}=O_r(1),
    \qquad
    \Norm{\norm{s}_2}_{L^{2r}}=O_r(\sqrt n).
\]
Thus
\[
\begin{aligned}
    &\left\|
        \left(
            \sum_{\ell=1}^n
            \left[
                \sup_{\norm{h}_2\le1}
                \left|
                    \tr(A_\ell[h]\Sigma_x)
                \right|
            \right]^2
        \right)^{1/2}
    \right\|_{L^r} 
    \le
    \frac{O_r(1)}{n}
    \Norm{\norm{c}_2}_{L^{2r}} .
\end{aligned}
\]
It remains to bound \(c\). Since
\[
    \norm{c}_2
    \le
    \norm{
        X^\top G\Sigma_xGXD X^\top GX
    }_F
    \le
    \norm{X^\top G\Sigma_xGX}_{\op}
    \norm{D}_F
    \norm{X^\top GX}_{\op}
    =O(n^{5/2}),
\]
since
\[
    \norm{\Sigma_x}_{\op}=O(1),
    \qquad
    \norm{X^\top G\Sigma_xGX}_{\op}=O(n),
    \qquad
    \norm{X^\top GX}_{\op}=O(n),
\]
and
\[
    \norm{D}_F
    \le
    O(1)\norm{X^\top y}_2
    =
    O(\sqrt n).
\]

Consequently,
\[
\begin{aligned}
    &\left\|
        \left(
            \sum_{\ell=1}^n
            \left[
                \sup_{\norm{h}_2\le1}
                \left|
                    \tr(A_\ell[h]\Sigma_x)
                \right|
            \right]^2
        \right)^{1/2}
    \right\|_{L^r}
    =
    O_r(n^{3/2}).
\end{aligned}
\]
\end{itemize}

Combining the Frobenius and trace bounds in the conditional quadratic-form
estimate yields
\[
\begin{aligned}
    &\left\|
        \left(
            \sum_{\ell=1}^n
            \left[
                \sup_{\norm{h}_2\le1}
                \Norm{
                    x^\top GX D X^\top G\mathcal B_\ell[h]Gx
                }_{L_x^q}
            \right]^2
        \right)^{1/2}
    \right\|_{L^r}
    =
    O_{q,r}\!\left(n^{3/2}(\log n)^6\right).
\end{aligned}
\]
\end{proof}

\begin{lemma}[Active quadratic norm of the background matrix]
\label{lem:active-quadratic-background-bounds}
Under Assumptions~\ref{ass:1dim}--\ref{ass:7third-order-curvature} and
\(\nabla^3\rho\equiv0\), for every fixed integers \(q,r\ge1\),
\[
    \left\|
        \left(
            \sum_{\ell=1}^n
            \left(
                \mathcal Q_q(\mathbb D_\ell[Y])^*
            \right)^2
        \right)^{1/2}
    \right\|_{L^r}
    =
    O_{q,r}\!\left(n^{3/2}(\log n)^7\right),
\]
where, for a matrix-valued statistic \(M\), we denote
\[
    \mathcal Q_q(\mathbb D_\ell[M])^*
    :=
    \sup_{\norm{h}_2\le1}
    \Norm{x^\top\mathbb D_\ell[M][h]x}_{L_x^q}.
\]
The same estimates hold for \(\check Y\), with \(\mathbb D'_\ell\) in place of
\(\mathbb D_\ell\).
\end{lemma}
\begin{proof}
Set
\[
    z:=X^\top Gx,
    \qquad
    a:=X^\top y,
    \qquad
    D=\diag_{m\in[n]}
    \left(
        L'''(x_m^\top\hat\theta)a_m
    \right).
\]
Then
\[
    x^\top Yx
    =
    x^\top GXDX^\top Gx
    =
    z^\top Dz.
\]
For \(\norm{h}_2\le1\),
\[
    \mathbb D_\ell[z][h]
    =
    e_\ell h^\top Gx
    +
    X^\top\mathbb D_\ell[G][h]x,
\]
and therefore
\[
    \mathbb D_\ell[x^\top Yx][h]
    =
    2z^\top D e_\ell h^\top Gx
    +
    z^\top\mathbb D_\ell[D][h]z
    +
    2z^\top D X^\top\mathbb D_\ell[G][h]x .
\]
We bound these three terms.

First,
\[
\begin{aligned}
    \left\|
        \left(
            \sum_{\ell=1}^n
            \left[
                \sup_{\norm{h}_2\le1}
                \Norm{
                    D_\ell z_\ell h^\top Gx
                }_{L_x^q}
            \right]^2
        \right)^{1/2}
    \right\|_{L^r}                                      
    &\le
    O_q
    \left\|
        \Norm{\norm{Gx}_2}_{L_x^{2q}}
        \left(
            \sum_{\ell=1}^n
            D_\ell^2\norm{Gx_\ell}_2^2
        \right)^{1/2}
    \right\|_{L^r}\\
    &\le
    O_q
    \left\|
        \Norm{\norm{Gx}_2}_{L_x^{2q}}
        \norm{D}_{\op}\norm{GX}_F
    \right\|_{L^r}                                      
    =
    O_{q,r}\!\left(n^{3/2}(\log n)^6\right),
\end{aligned}
\]
using
\[
    \Norm{\norm{Gx}_2}_{L_x^{2q}}=O_q(\sqrt n),
    \qquad
    \Norm{\norm{D}_{\op}}_{L^{3r}}=O_r((\log n)^6),
    \qquad
    \Norm{\norm{GX}_F}_{L^{3r}}=O_r(n).
\]

We now handle the term \(z^\top\mathbb D_\ell[D]z\). For \(m\in[n]\),
\[
\begin{aligned}
    \mathbb D_\ell[D_m][h]
    &=
    \mathbb D_\ell
    \left[
        L'''(x_m^\top\hat\theta)
    \right][h]a_m
    +
    L'''(x_m^\top\hat\theta)
    \left(
        \mathbf 1_{m=\ell}h^\top y
        +
        x_m^\top\mathbb D_\ell[y][h]
    \right),
\end{aligned}
\]
with
\[
    \left|
        \mathbb D_\ell
        \left[
            L'''(x_m^\top\hat\theta)
        \right][h]
    \right|
    \le
    O(1)
    \left|
        \mathbf 1_{m=\ell}h^\top\hat\theta
        +
        x_m^\top\mathbb D_\ell[\hat\theta][h]
    \right|.
\]
Thus \(z^\top\mathbb D_\ell[D][h]z\) is the sum of the following three
contributions:
\[
\begin{aligned}
    T_\ell^{D,\mathrm{loc}}[h]
    &:=
    O(1)
    \left(
        h^\top\hat\theta\,a_\ell
        +
        h^\top y
    \right)
    z_\ell^2,                                                \\
    T_\ell^{D,\theta}[h]
    &:=
    O(1)
    \sum_{m=1}^n
    L'''(x_m^\top\hat\theta)
    a_m z_m^2\,
    x_m^\top\mathbb D_\ell[\hat\theta][h],                    \\
    T_\ell^{D,y}[h]
    &:=
    \sum_{m=1}^n
    L'''(x_m^\top\hat\theta)
    z_m^2\,
    x_m^\top\mathbb D_\ell[y][h].
\end{aligned}
\]

For the local contribution,
\[
\begin{aligned}
    &\left\|
        \left(
            \sum_{\ell=1}^n
            \left[
                \sup_{\norm{h}_2\le1}
                \Norm{T_\ell^{D,\mathrm{loc}}[h]}_{L_x^q}
            \right]^2
        \right)^{1/2}
    \right\|_{L^r}                                      \\
    &\quad\le
    O_q
    \left\|
        \left(
            \norm{\hat\theta}_2\norm{a}_\infty+\norm{y}_2
        \right)
        \left(
            \sum_{\ell=1}^n
            \Norm{|z_\ell|^2}_{L_x^q}^2
        \right)^{1/2}
    \right\|_{L^r}
    \le
    O_q
    \left\|
        \left(
            \norm{\hat\theta}_2\norm{a}_\infty+\norm{y}_2
        \right)
        \left(
            \sum_{\ell=1}^n
            \norm{Gx_\ell}_2^4
        \right)^{1/2}
    \right\|_{L^r}                                      \\
    &\quad\le
    O_q
    \left\|
        \left(
            \norm{\hat\theta}_2\norm{a}_\infty+\norm{y}_2
        \right)
        \norm{GX}_{\op}\norm{GX}_F
    \right\|_{L^r} 
    =
    O_{q,r}\!\left(n^{3/2}(\log n)^6\right).
\end{aligned}
\]

For the \(\hat\theta\)-background contribution, set
\[
    w_m
    :=
    L'''(x_m^\top\hat\theta)a_m z_m^2,
    \qquad
    w=(w_m)_{m\in[n]}.
\]
Then
\[
    T_\ell^{D,\theta}[h]
    =
    O(1)\sum_{m=1}^n w_m\,x_m^\top\mathbb D_\ell[\hat\theta][h].
\]
Using
\[
    \mathbb D_\ell[\hat\theta][h]
    =
    -\frac1nG
    \left[
        L''(x_\ell^\top\hat\theta)(h^\top\hat\theta)x_\ell
        +
        s_\ell h
    \right],
\]
we have
\[
\begin{aligned}
    &\left\|
        \left(
            \sum_{\ell=1}^n
            \left[
                \sup_{\norm{h}_2\le1}
                \Norm{T_\ell^{D,\theta}[h]}_{L_x^q}
            \right]^2
        \right)^{1/2}
    \right\|_{L^r}                                      \\
    &\quad\le
    \frac{O(1)}{n}
    \left\|
        \norm{\hat\theta}_2
        \Norm{\norm{X^\top GXw}_2}_{L_x^q}
        +
        \norm{s}_2
        \Norm{\norm{GXw}_2}_{L_x^q}
    \right\|_{L^r}.
\end{aligned}
\]
Moreover,
\[
\begin{aligned}
    \left\|
        \Norm{\norm{w}_2}_{L_x^q}
    \right\|_{L^{3r}}
    &\le
    O_q
    \left\|
        \norm{a}_\infty
        \left(
            \sum_{m=1}^n
            \Norm{|z_m|^2}_{L_x^q}^2
        \right)^{1/2}
    \right\|_{L^{3r}}                                  \\
    &\le
    O_q
    \left\|
        \norm{a}_\infty
        \norm{GX}_{\op}\norm{GX}_F
    \right\|_{L^{3r}}=
    O_{q,r}\!\left(n^{3/2}(\log n)^6\right).
\end{aligned}
\]
Hence, using
\[
    \norm{X^\top GX}_{\op}=O(n),
    \qquad
    \norm{GX}_{\op}=O(\sqrt n),
    \qquad
    \Norm{\norm{\hat\theta}_2}_{L^{3r}}=O_r(1),
    \qquad
    \Norm{\norm{s}_2}_{L^{3r}}=O_r(\sqrt n),
\]
we obtain
\[
\begin{aligned}
    &\left\|
        \left(
            \sum_{\ell=1}^n
            \left[
                \sup_{\norm{h}_2\le1}
                \Norm{T_\ell^{D,\theta}[h]}_{L_x^q}
            \right]^2
        \right)^{1/2}
    \right\|_{L^r}
    =
    O_{q,r}\!\left(n^{3/2}(\log n)^6\right).
\end{aligned}
\]

For the \(y\)-background contribution, set
\[
    v_m:=L'''(x_m^\top\hat\theta)z_m^2,
    \qquad
    v=(v_m)_{m\in[n]}.
\]
Then
\[
    T_\ell^{D,y}[h]
    =
    \sum_{m=1}^n v_m\,x_m^\top\mathbb D_\ell[y][h].
\]
Since
\[
    \mathbb D_\ell[y][h]
    =
    -G\left(\mathcal L_\ell[h]+\mathcal B_\ell[h]\right)y,
\]
the local part satisfies
\[
\begin{aligned}
    &\left\|
        \left(
            \sum_{\ell=1}^n
            \left[
                \sup_{\norm{h}_2\le1}
                \Norm{
                    \sum_{m=1}^n
                    v_m x_m^\top G\mathcal L_\ell[h]y
                }_{L_x^q}
            \right]^2
        \right)^{1/2}
    \right\|_{L^r}                                      \\
    &\quad\le
    \frac{O(1)}{n}
    \left\|
        \Norm{\norm{GXv}_2}_{L_x^q}\norm{X^\top y}_2
        +
        \norm{y}_2\Norm{\norm{X^\top GXv}_2}_{L_x^q}
        +
        \norm{\hat\theta}_2\norm{a}_\infty
        \Norm{\norm{X^\top GXv}_2}_{L_x^q}
    \right\|_{L^r}.
\end{aligned}
\]
Also,
\[
\begin{aligned}
    \left\|
        \Norm{\norm{v}_2}_{L_x^q}
    \right\|_{L^{3r}}
    &\le
    O_q
    \left\|
        \left(
            \sum_{m=1}^n
            \Norm{|z_m|^2}_{L_x^q}^2
        \right)^{1/2}
    \right\|_{L^{3r}}                                  \\
    &\le
    O_q
    \left\|
        \norm{GX}_{\op}\norm{GX}_F
    \right\|_{L^{3r}}
    =
    O_{q,r}(n^{3/2}).
\end{aligned}
\]
Therefore the local \(y\)-background contribution is
\[
\begin{aligned}
    &\left\|
        \left(
            \sum_{\ell=1}^n
            \left[
                \sup_{\norm{h}_2\le1}
                \Norm{
                    \sum_{m=1}^n
                    v_m x_m^\top G\mathcal L_\ell[h]y
                }_{L_x^q}
            \right]^2
        \right)^{1/2}
    \right\|_{L^r}
    =
    O_{q,r}\!\left(n^{3/2}(\log n)^6\right).
\end{aligned}
\]
For the background part of \(\mathbb D_\ell[y]\), we argue conditionally on the
active column \(x\). Then the coefficient vector
\[
    v_m=L'''(x_m^\top\hat\theta)z_m^2
\]
is frozen, and the background-part computation in the proof of
Lemma~\ref{lem:square-summed-y-derivative} applies with \(v\) in place of the
deterministic coefficient vector. It gives, conditionally on \(x\),
\[
\begin{aligned}
    &\left(
        \sum_{\ell=1}^n
        \left[
            \sup_{\norm{h}_2\le1}
            \left|
                \sum_{m=1}^n
                v_m x_m^\top G\mathcal B_\ell[h]y
            \right|
        \right]^2
    \right)^{1/2}                                      \\
    &\qquad\le
    \frac{O(1)}{n^2}
    \left[
        \norm{\hat\theta}_2
        \norm{X^\top GX}_{\op}
        \norm{X^\top y}_\infty
        \norm{X^\top GXv}_2
        +
        \norm{s}_2
        \norm{GX}_{\op}
        \norm{X^\top y}_\infty
        \norm{X^\top GXv}_2
    \right].
\end{aligned}
\]
Using
\[
    \norm{X^\top GXv}_2
    \le
    \norm{X^\top GX}_{\op}\norm{v}_2
    =
    O(n)\norm{v}_2,
\]
together with
\[
    \left\|
        \Norm{\norm{v}_2}_{L_x^q}
    \right\|_{L^{3r}}
    =
    O_{q,r}(n^{3/2}),
\]
and the bounds
\[
    \Norm{\norm{X^\top y}_\infty}_{L^{3r}}
    =
    O_r((\log n)^6),
    \qquad
    \Norm{\norm{\hat\theta}_2}_{L^{3r}}=O_r(1),
    \qquad
    \Norm{\norm{s}_2}_{L^{3r}}=O_r(\sqrt n),
\]
we obtain
\[
\begin{aligned}
    &\left\|
        \left(
            \sum_{\ell=1}^n
            \left[
                \sup_{\norm{h}_2\le1}
                \Norm{
                    \sum_{m=1}^n
                    v_m x_m^\top G\mathcal B_\ell[h]y
                }_{L_x^q}
            \right]^2
        \right)^{1/2}
    \right\|_{L^r}
    =
    O_{q,r}\!\left(n^{3/2}(\log n)^6\right).
\end{aligned}
\]
Thus
\[
\begin{aligned}
    &\left\|
        \left(
            \sum_{\ell=1}^n
            \left[
                \sup_{\norm{h}_2\le1}
                \Norm{z^\top\mathbb D_\ell[D][h]z}_{L_x^q}
            \right]^2
        \right)^{1/2}
    \right\|_{L^r}
    =
    O_{q,r}\!\left(n^{3/2}(\log n)^6\right).
\end{aligned}
\]

It remains to bound the resolvent contribution
\[
    2z^\top D X^\top\mathbb D_\ell[G][h]x
    =
    -2z^\top D X^\top G\mathcal L_\ell[h]Gx
    -
    2z^\top D X^\top G\mathcal B_\ell[h]Gx .
\]
For the local part,
\[
\begin{aligned}
    &\left\|
        \left(
            \sum_{\ell=1}^n
            \left[
                \sup_{\norm{h}_2\le1}
                \Norm{
                    z^\top D X^\top G\mathcal L_\ell[h]Gx
                }_{L_x^q}
            \right]^2
        \right)^{1/2}
    \right\|_{L^r}                                      \\
    &\quad\le
    \frac{O(1)}{n}
    \left\|
        \Norm{
            \norm{GXDz}_2\norm{z}_2
            +
            \norm{X^\top GXDz}_2\norm{Gx}_2
            +
            \norm{\hat\theta}_2
            \norm{X^\top GXDz}_2\norm{z}_\infty
        }_{L_x^q}
    \right\|_{L^r}                                      \\
    &\quad\le
    O_{q,r}\!\left(n^{3/2}(\log n)^7\right).
\end{aligned}
\]
Indeed, we used
\[
    \norm{GXDz}_2
    \le
    \norm{GX}_{\op}\norm{D}_{\op}\norm{z}_2,
    \qquad
    \norm{X^\top GXDz}_2
    \le
    \norm{X^\top GX}_{\op}\norm{D}_{\op}\norm{z}_2,
\]
together with
\[
    \Norm{\norm{z}_2}_{L_x^{3q}}=O_q(n),
    \qquad
    \Norm{\norm{z}_\infty}_{L_x^{3q}}=O_q(\sqrt n\log n),
    \qquad
    \Norm{\norm{Gx}_2}_{L_x^{3q}}=O_q(\sqrt n).
\]

For the background part, Lemma~\ref{lem:background-resolvent-active-quadratic}
gives directly
\[
\begin{aligned}
    &\left\|
        \left(
            \sum_{\ell=1}^n
            \left[
                \sup_{\norm{h}_2\le1}
                \Norm{
                    z^\top D X^\top G\mathcal B_\ell[h]Gx
                }_{L_x^q}
            \right]^2
        \right)^{1/2}
    \right\|_{L^r}
    =
    O_{q,r}\!\left(n^{3/2}(\log n)^6\right).
\end{aligned}
\]

Combining the three displayed bounds yields
\[
\begin{aligned}
    &\left\|
        \left(
            \sum_{\ell=1}^n
            \left(
                \mathcal Q_q(\mathbb D_\ell[Y])^*
            \right)^2
        \right)^{1/2}
    \right\|_{L^r}
    =
    O_{q,r}\!\left(n^{3/2}(\log n)^7\right).
\end{aligned}
\]

 The checked estimates are
identical, since the checked background satisfies the same assumptions and has
the same distribution.
\end{proof}

\begin{lemma}[Background gradients of the active contractions]
\label{lem:background-gradients-active-contractions}
Under Assumptions~\ref{ass:1dim}--\ref{ass:7third-order-curvature} and
\(\nabla^3\rho\equiv0\), for every fixed integer \(r\ge1\),
\[
\begin{aligned}
    &\left\|
        \left(
            \sum_{\ell=1}^n
            \left(
                \norm{
                    \mathbb D_\ell[
                        \mathcal C^{00}(y,\check y)
                    ]
                }^*
            \right)^2
            +
            \sum_{\ell=1}^n
            \left(
                \norm{
                    \mathbb D'_\ell[
                        \mathcal C^{00}(y,\check y)
                    ]
                }^*
            \right)^2
        \right)^{1/2}
    \right\|_{L^r}
    =
    O_r\!\left(\frac{(\log n)^6}{\sqrt n}\right),
\end{aligned}
\]
\[
\begin{aligned}
    &\left\|
        \left(
            \sum_{\ell=1}^n
            \left(
                \norm{
                    \mathbb D_\ell[
                        \mathcal C^{01}(y,\check Y)
                    ]
                }^*
            \right)^2
            +
            \sum_{\ell=1}^n
            \left(
                \norm{
                    \mathbb D'_\ell[
                        \mathcal C^{01}(y,\check Y)
                    ]
                }^*
            \right)^2
        \right)^{1/2}
    \right\|_{L^r}
    =
    O_r\!\left(n^{3/2}(\log n)^{12}\right),
\end{aligned}
\]
and the same bound holds with \(\mathcal C^{01}(y,\check Y)\) replaced by
\(\mathcal C^{10}(Y,\check y)\). Finally,
\[
\begin{aligned}
    &\left\|
        \left(
            \sum_{\ell=1}^n
            \left(
                \norm{
                    \mathbb D_\ell[
                        \mathcal C^{11}(Y,\check Y)
                    ]
                }^*
            \right)^2
            +
            \sum_{\ell=1}^n
            \left(
                \norm{
                    \mathbb D'_\ell[
                        \mathcal C^{11}(Y,\check Y)
                    ]
                }^*
            \right)^2
        \right)^{1/2}
    \right\|_{L^r}
    =
    O_r\!\left(n^{7/2}(\log n)^{13}\right).
    \end{aligned}
\]
\end{lemma}

\begin{proof}
We will skip some of the terms appearing from the derivations $\mathbb D_\ell'$ when they are symmetric to those appearing through derivations $\mathbb D_\ell$ (for the terms $\mathcal C^{00}$ and $\mathcal C^{11}$).
We use when useful the active-contraction bounds given by Lemma~\ref{lem:active-column-contractions-detailed}:
\[
    |\mathcal C^{00}(a,b)|
    \le
    R\norm{a}_2\norm{b}_2,
\]
\[
    |\mathcal C^{01}(a,C)|
    \le
    R\norm{a}_2\mathcal Q_4(C),
    \qquad
    |\mathcal C^{10}(B,b)|
    \le
    R\mathcal Q_4(B)\norm{b}_2,
\]
and
\[
    |\mathcal C^{11}(B,C)|
    \le
    R\mathcal Q_4(B)\mathcal Q_4(C),
    \qquad
    \Norm{R}_{L^r}=O_r(1).
\]

For \(\mathcal C^{00}(y,\check y)\), the product rule gives
\[
\begin{aligned}
    \mathbb D_\ell[
        \mathcal C^{00}(y,\check y)
    ][h]
    &=
    \Cov_x\!\left(
        \mathbb D_\ell[s_x^\zeta][h]\,y^\top x,
        \check s_x^\zeta\,\check y^\top x
    \right)
    +
    \mathcal C^{00}(\mathbb D_\ell[y][h],\check y).
\end{aligned}
\]
By Lemma~\ref{lem:differentiated-active-score-contractions},
\[
\begin{aligned}
    &\left\|
        \left(
            \sum_{\ell=1}^n
            \left(
                \norm{
                    \Cov_x\!\left(
                        \mathbb D_\ell[s_x^\zeta]\,y^\top x,
                        \check s_x^\zeta\,\check y^\top x
                    \right)
                }^*
            \right)^2
            +
            \sum_{\ell=1}^n
            \left(
                \norm{
                    \Cov_x\!\left(
                        s_x^\zeta\,y^\top x,
                        \mathbb D'_\ell[\check s_x^\zeta]\,
                        \check y^\top x
                    \right)
                }^*
            \right)^2
        \right)^{1/2}
    \right\|_{L^r}
    \le
    O_r(n^{-1/2}).
\end{aligned}
\]
Moreover, by duality applied to \(\mathcal C^{00}\), there exist
background-measurable vectors \(v_{00}\) and \(\check v_{00}\) such that
\[
    \mathcal C^{00}(a,\check y)=\check v_{00}^\top a,
    \qquad
    \mathcal C^{00}(y,b)= v_{00}^\top b,
\]
with
\[
    \norm{\check v_{00}}_2\le R\norm{\check y}_2,
    \qquad
    \norm{v_{00}}_2\le R\norm{y}_2.
\]
Therefore, by Lemma~\ref{lem:square-summed-y-derivative} and
Lemma~\ref{lem:background-matrix-bounds_bounded},
\[
\begin{aligned}
    \left\|
        \left(
            \sum_{\ell=1}^n
            \left(
                \norm{
                    \mathcal C^{00}(\mathbb D_\ell[y],\check y)
                }^*
            \right)^2
        \right)^{1/2}
    \right\|_{L^r}                                      
    &\le
    \frac{O_r((\log n)^6)}{\sqrt n}
    \Norm{\norm{\check v_{00}}_2}_{L^{4r}}                     \\
    &\le
    \frac{O_r((\log n)^6)}{\sqrt n}
    \Norm{R\norm{\check y}_2}_{L^{4r}}
    =
    O_r\!\left(\frac{(\log n)^6}{\sqrt n}\right),
\end{aligned}
\]
and similarly
\[
\begin{aligned}
    &\left\|
        \left(
            \sum_{\ell=1}^n
            \left(
                \norm{
                    \mathcal C^{00}(y,\mathbb D'_\ell[\check y])
                }^*
            \right)^2
        \right)^{1/2}
    \right\|_{L^r}
    =
    O_r\!\left(\frac{(\log n)^6}{\sqrt n}\right).
\end{aligned}
\]
This proves the \(00\)-bound.

We now treat \(\mathcal C^{01}(y,\check Y)\). The product rule gives
\[
\begin{aligned}
    \mathbb D_\ell[
        \mathcal C^{01}(y,\check Y)
    ][h]
    &=
    \Cov_x\!\left(
        \mathbb D_\ell[s_x^\zeta][h]\,y^\top x,
        (\check s_x^\zeta)^2x^\top\check Yx
    \right)                                             \\
    &\quad+
    \mathcal C^{01}(\mathbb D_\ell[y][h],\check Y).
\end{aligned}
\]
By Lemma~\ref{lem:differentiated-active-score-contractions},
\[
\begin{aligned}
    &\left\|
        \left(
            \sum_{\ell=1}^n
            \left(
                \norm{
                    \Cov_x\!\left(
                        \mathbb D_\ell[s_x^\zeta]\,y^\top x,
                        (\check s_x^\zeta)^2x^\top\check Yx
                    \right)
                }^*
            \right)^2
        \right)^{1/2}
    \right\|_{L^r}
    \le
    O_r(n^{3/2}(\log n)^6).
\end{aligned}
\]
By duality applied to \(\mathcal C^{01}\), there exists \(\check v_{01}\) such that
\[
    \mathcal C^{01}(a,\check Y)=\check v_{01}^\top a,
    \qquad
    \norm{\check v_{01}}_2\le R\mathcal Q_4(\check Y).
\]
Thus, by Lemmas~\ref{lem:square-summed-y-derivative} and
\ref{lem:background-matrix-bounds_bounded},
\[
\begin{aligned}
    &\left\|
        \left(
            \sum_{\ell=1}^n
            \left(
                \norm{
                    \mathcal C^{01}(\mathbb D_\ell[y],\check Y)
                }^*
            \right)^2
        \right)^{1/2}
    \right\|_{L^r}\le
    \frac{O_r((\log n)^6)}{\sqrt n}
    \Norm{R\mathcal Q_4(\check Y)}_{L^{4r}} 
    \le
    O_r\!\left(n^{3/2}(\log n)^{12}\right).
\end{aligned}
\]
For the checked matrix derivative, the active contraction bound gives
\[
\begin{aligned}
    &\left(
        \sum_{\ell=1}^n
        \left(
            \norm{
                \Cov_x\!\left(
                    s_x^\zeta\,y^\top x,
                    (\check s_x^\zeta)^2
                    x^\top\mathbb D'_\ell[\check Y]x
                \right)
            }^*
        \right)^2
    \right)^{1/2} 
    \le
    R\norm{y}_2
    \left(
        \sum_{\ell=1}^n
        \left(
            \mathcal Q_4(\mathbb D'_\ell[\check Y])^*
        \right)^2
    \right)^{1/2}.
\end{aligned}
\]
Consequently, by
Lemmas~\ref{lem:active-quadratic-background-bounds} and~\ref{lem:background-matrix-bounds_bounded},
\[
\begin{aligned}
    \left\|
        \left(
            \sum_{\ell=1}^n
            \left(
                \norm{
                    \Cov_x\!\left(
                        s_x^\zeta\,y^\top x,
                        (\check s_x^\zeta)^2
                        x^\top\mathbb D'_\ell[\check Y]x
                    \right)
                }^*
            \right)^2
        \right)^{1/2}
    \right\|_{L^r}                                     
    &\le
    \left\|
        R\norm{y}_2
        \left(
            \sum_{\ell=1}^n
            \left(
                \mathcal Q_4(\mathbb D'_\ell[\check Y])^*
            \right)^2
        \right)^{1/2}
    \right\|_{L^r}                                      \\
    &\le
    O_r\!\left(n^{3/2}(\log n)^7\right).
\end{aligned}
\]
Combining the three estimates yields
\[
\begin{aligned}
    &\left\|
        \left(
            \sum_{\ell=1}^n
            \left(
                \norm{
                    \mathbb D_\ell[
                        \mathcal C^{01}(y,\check Y)
                    ]
                }^*
            \right)^2
            +
            \sum_{\ell=1}^n
            \left(
                \norm{
                    \mathbb D'_\ell[
                        \mathcal C^{01}(y,\check Y)
                    ]
                }^*
            \right)^2
        \right)^{1/2}
    \right\|_{L^r}
    =
    O_r\!\left(n^{3/2}(\log n)^{12}\right).
\end{aligned}
\]
The proof for \(\mathcal C^{10}(Y,\check y)\) is identical, with checked and
unchecked quantities interchanged.

It remains to treat \(\mathcal C^{11}(Y,\check Y)\). The differentiated-score
terms are controlled by Lemma~\ref{lem:differentiated-active-score-contractions}:
\[
\begin{aligned}
    &\left\|
        \left(
            \sum_{\ell=1}^n
            \left(
                \norm{
                    \Cov_x\!\left(
                        2s_x^\zeta
                        \mathbb D_\ell[s_x^\zeta]\,x^\top Yx,
                        (\check s_x^\zeta)^2x^\top\check Yx
                    \right)
                }^*
            \right)^2
        \right)^{1/2}
    \right\|_{L^r}
    \le
    O_r(n^{7/2}(\log n)^{12}),
\end{aligned}
\]
and the checked differentiated-score term is controlled in the same way.
For the unchecked matrix derivative,
\[
\begin{aligned}
    &\left(
        \sum_{\ell=1}^n
        \left(
            \norm{
                \Cov_x\!\left(
                    (s_x^\zeta)^2x^\top\mathbb D_\ell[Y]x,
                    (\check s_x^\zeta)^2x^\top\check Yx
                \right)
            }^*
        \right)^2
    \right)^{1/2} 
    \le
    R\mathcal Q_4(\check Y)
    \left(
        \sum_{\ell=1}^n
        \left(
            \mathcal Q_4(\mathbb D_\ell[Y])^*
        \right)^2
    \right)^{1/2}.
\end{aligned}
\]
Therefore, by Lemma~\ref{lem:background-matrix-bounds_bounded} and\ref{lem:active-quadratic-background-bounds},
\[
\begin{aligned}
    \left\|
        \left(
            \sum_{\ell=1}^n
            \left(
                \norm{
                    \Cov_x\!\left(
                        (s_x^\zeta)^2x^\top\mathbb D_\ell[Y]x,
                        (\check s_x^\zeta)^2x^\top\check Yx
                    \right)
                }^*
            \right)^2
        \right)^{1/2}
    \right\|_{L^r}                                      
    &\le
    \left\|
        R\mathcal Q_4(\check Y)
        \left(
            \sum_{\ell=1}^n
            \left(
                \mathcal Q_4(\mathbb D_\ell[Y])^*
            \right)^2
        \right)^{1/2}
    \right\|_{L^r}                                      \\
    &\le
    O_r\!\left(n^{7/2}(\log n)^{13}\right).
\end{aligned}
\]
The checked matrix derivative satisfies the same bound:
\[
\begin{aligned}
    &\left\|
        \left(
            \sum_{\ell=1}^n
            \left(
                \norm{
                    \Cov_x\!\left(
                        (s_x^\zeta)^2x^\top Yx,
                        (\check s_x^\zeta)^2
                        x^\top\mathbb D'_\ell[\check Y]x
                    \right)
                }^*
            \right)^2
        \right)^{1/2}
    \right\|_{L^r}
    \le
    O_r\!\left(n^{7/2}(\log n)^{13}\right).
\end{aligned}
\]
Combining the differentiated-score and matrix-derivative estimates gives
\[
\begin{aligned}
    &\left\|
        \left(
            \sum_{\ell=1}^n
            \left(
                \norm{
                    \mathbb D_\ell[
                        \mathcal C^{11}(Y,\check Y)
                    ]
                }^*
            \right)^2
            +
            \sum_{\ell=1}^n
            \left(
                \norm{
                    \mathbb D'_\ell[
                        \mathcal C^{11}(Y,\check Y)
                    ]
                }^*
            \right)^2
        \right)^{1/2}
    \right\|_{L^r}
    =
    O_r\!\left(n^{7/2}(\log n)^{13}\right).
\end{aligned}
\]
This completes the proof.
\end{proof}

Combining Lemma~\ref{lem:background-gradients-active-contractions} with
\eqref{borne_Cov_x} and
Remark~\ref{rem:csqce_wasserstein_bound_scd_order_leaveone_out_expension}, we
get the following bound.

\begin{corollary}[Bound on \(\omega_c\)]
\label{cor:bound-cn-bar-x}
Under Assumptions~\ref{ass:1dim}--\ref{ass:7third-order-curvature} and
\(\nabla^3\rho\equiv0\),
\[
    \omega_c
    =
    O\!\left(
        \frac{(\log n)^{13}}{n^{3/2}}
    \right).
\]
\end{corollary}

\begin{proof}
Fix the active index \(j\). In the generic notation of this subsection,
\[
    \Phi_j
    :=
    \Cov_j\!\left(
        \mathfrak d_j(X),
        \mathfrak d_j(X^{A_j})
    \right)
\]
has the same form as $\Cov_x(\mathfrak d_x,\check{\mathfrak d}_x)$. 
By \eqref{borne_Cov_x},
\[
    \Phi_j
    =
    \frac1n\mathcal C^{00}(y,\check y)
    +
    \frac1{2n^3}\mathcal C^{01}(y,\check Y)
    +
    \frac1{2n^3}\mathcal C^{10}(Y,\check y)
    +
    \frac1{4n^5}\mathcal C^{11}(Y,\check Y).
\]

We apply the tensorized Poincaré inequality to \(\Phi_j\) as a function of the
background variables and their independent copies. For coordinates that are
shared by \(X\) and \(X^{A_j}\), the derivative is the sum of the unchecked and
checked derivatives; hence it is bounded by the square-sum of the two
contributions controlled in
Lemma~\ref{lem:background-gradients-active-contractions}. Therefore
\[
\begin{aligned}
    \sqrt{\Var(\Phi_j)}
    &\le
    O(1)
    \left\|
        \left(
            \sum_{\ell}
            \left(
                \norm{\mathbb D_\ell[\Phi_j]}^*
            \right)^2
            +
            \sum_{\ell}
            \left(
                \norm{\mathbb D'_\ell[\Phi_j]}^*
            \right)^2
        \right)^{1/2}
    \right\|_{L^2}.
\end{aligned}
\]
Using Lemma~\ref{lem:background-gradients-active-contractions} and the
decomposition above, we obtain
\[
\begin{aligned}
    \sqrt{\Var(\Phi_j)}
    &\le
    \frac1n
    O\!\left(\frac{(\log n)^6}{\sqrt n}\right)
    +
    \frac1{n^3}
    O\!\left(n^{3/2}(\log n)^{12}\right) 
    +
    \frac1{n^5}
    O\!\left(n^{7/2}(\log n)^{13}\right)               \\
    &=
    O\!\left(
        \frac{(\log n)^{13}}{n^{3/2}}
    \right),
\end{aligned}
\]
uniformly in \(j\). Hence
\[
    \sup_{j\in[n]}
    \sqrt{
        \Var\!\left(
            \Cov_j\!\left(
                \mathfrak d_j(X),
                \mathfrak d_j(X^{A_j})
            \right)
        \right)
    }
    =
    O\!\left(
        \frac{(\log n)^{13}}{n^{3/2}}
    \right).
\]
This bound transfers to $\omega_c$ since, by
Remark~\ref{rem:csqce_wasserstein_bound_scd_order_leaveone_out_expension},
replacing \(\delta_j f_n\) by \(\mathfrak d_j\) changes the corresponding
\(\omega_c\) quantity by only and order $O\!\left(\frac{(\log n)^3}{n^{3/2}}\right)$.
\end{proof}



\appendix
\section{Proof of Theorem~\ref{the:conditional-loo-perturbative}}
\label{subsec:proof_of_clt-perturbative}

We start with the covariance decomposition that explains the role of the sets $[i-1]$ and justifies the identity $\mathbb E[\sum_{i=1}^n c_i] = 1$.

\begin{lemma}[Ordered covariance decomposition]
\label{lem:cov-decomp}
Let $Z=(z_1,\ldots,z_n)$ have independent entries. Let $Z'$ be an
independent copy of $Z$. Then, for any square-integrable statistics
$U=U(Z)$ and $V=V(Z)$,
\[
    \Cov(U,V)
    =
    \sum_{i=1}^n
    \E\left[
        \Cov_i\!\left(U(Z),V(Z^{[i-1]})\right)
    \right].
\]
\end{lemma}

\begin{proof}
Since $Z'$ has the same law as $Z$ and is independent of $Z$,
\[
    \Cov(U,V)
    =
    \E\left[U(Z)\bigl(V(Z)-V(Z')\bigr)\right]
    =
    \sum_{i=1}^n
    \E\left[U(Z)\left(
        V(Z^{[i-1]})-V(Z^{[i]})
    \right)\right],
\]
thanks to a telescoping identity.

Fix $i$. Conditionally on all variables except $z_i$ and $z_i'$, the
random variable $V(Z^{[i]})$ is obtained from $V(Z^{[i-1]})$ by replacing
$z_i$ by an independent copy. Moreover, conditionally on
$Z_{-i},Z'_{[i-1]}$, the variables $U(Z)$ and $V(Z^{[i]})$ depend on the
independent variables $z_i$ and $z_i'$, respectively. Therefore
\[
\begin{aligned}
    &\E\left[
        U(Z)\left(
            V(Z^{[i-1]})-V(Z^{[i]})
        \right)
    \right] \\
    &\qquad =
    \E\left[
        U(Z)V(Z^{[i-1]})
    \right]
    -
    \E\left[
        \E_i[U(Z)]\,\E_i[V(Z^{[i-1]})]
    \right] \\
    &\qquad =
    \E\left[
        \Cov_i\!\left(U(Z),V(Z^{[i-1]})\right)
    \right].
\end{aligned}
\]
Summing over $i$ gives the result.
\end{proof}

\begin{lemma}[Normal Stein-factor bounds; see~\cite{Chatterjee2008}, Lemma~4.2]
\label{lem:Stein_equation_solution}
Let $N\sim\mathcal N(0,1)$ and let $h:\mathbb R\to\mathbb R$ be
$1$-Lipschitz. The Stein equation
\[
    f_h'(x)-x f_h(x)=h(x)-\E h(N)
\]
admits the solution
\[
    f_h(x)
    =
    e^{x^2/2}
    \int_{-\infty}^x
    \bigl(h(t)-\E h(N)\bigr)e^{-t^2/2}\,dt .
\]
Moreover,
\[
    \|f_h'\|_\infty \le \sqrt{\frac{2}{\pi}},
    \qquad
    |f_h'(x)-f_h'(y)|\le 2|x-y|,
    \quad x,y\in\mathbb R.
\]
\end{lemma}

And here is the proof of the theorem.
\begin{proof}[Proof of Theorem~\ref{the:conditional-loo-perturbative}]
Set
\[
    G:=g_n(Z)
    \qquad \text{and}\qquad
    G^{[i-1]}:=g_n(Z^{[i-1]}).
\]
Recall the notation
\[
    \Delta_i
    :=
    G-\E_i[G],
    \qquad
    \Delta_i^{[i-1]}
    :=
    G^{[i-1]}-\E_i[G^{[i-1]}]
\]
and
\[
    c_i
    :=
    \Cov_i\!\left(G,G^{[i-1]}\right)
    =
    \E_i\left[
        \Delta_i\Delta_i^{[i-1]}
    \right].
\]
Note first that, by Lemma~\ref{lem:cov-decomp} applied with $U=V=G$, we get
\[
    \sum_{i=1}^n \E[c_i]
    =
    \Var(G)
    =
    1.
\]

Let $h$ be $1$-Lipschitz and let $f:=f_h$ be the solution of the Stein
equation from Lemma~\ref{lem:Stein_equation_solution}. Since $\E G=0$, the
Stein equation gives
\begin{align}
\label{eq:pf_1}
    \E[h(G)]-\E[h(N)]
    =
    \E[f'(G)]-\E[Gf(G)].
\end{align}
Using Lemma~\ref{lem:cov-decomp} with $U=f(G)$ and $V=G$, we obtain
\begin{align}
\label{eq:pf_1bis}
    \E[Gf(G)]
    &=
    \Cov(f(G),G) \nonumber \\
    &=
    \sum_{i=1}^n
    \E\left[
        \Cov_i\!\left(f(G),G^{[i-1]}\right)
    \right]
    =
    \sum_{i=1}^n
    \E\left[
        \E_i\left[
            \bigl(f(G)-\E_i[f(G)]\bigr)
            \Delta_i^{[i-1]}
        \right]
    \right].
\end{align}

Let us introduce the random variable $R_i\in\mathbb R$ satisfying
\begin{align}
\label{eq:intro_R_i}
    f(G)
    =
    f(\E_i[G])+f'(\E_i[G])\Delta_i+R_i.
\end{align}
By Lemma~\ref{lem:Stein_equation_solution}, if, say, $\E_i[G]\le G$, then
\begin{align}
\label{eq:borne_R_i}
    |R_i|
    =
    \left|\int_{\E_i[G]}^G
        \bigl(f'(u)-f'(\E_i[G])\bigr)\,du
    \right|
    \le
    2\int_{\E_i[G]}^G (u-\E_i[G])\,du
    \le
    |\Delta_i|^2.
\end{align}
The case $G\le \E_i[G]$ is identical. Since $\E_i[\Delta_i]=0$,
\eqref{eq:intro_R_i} implies
\[
    f(G)-\E_i[f(G)]
    =
    f'(\E_i[G])\Delta_i+R_i-\E_iR_i.
\]
Consequently,
\begin{align}
\label{eq:pf_2}
    \E_i\left[
        \bigl(f(G)-\E_i[f(G)]\bigr)
        \Delta_i^{[i-1]}
    \right]
    =
    f'(\E_i[G])c_i
    +
    \E_i\left[
        (R_i-\E_iR_i)\Delta_i^{[i-1]}
    \right].
\end{align}
One can bound the remainder thanks to~\eqref{eq:borne_R_i}:
\begin{align}
\label{eq:remainder-bound}
    &\left|
    \E\left[
        \E_i\left[
            (R_i-\E_iR_i)\Delta_i^{[i-1]}
        \right]
    \right]
    \right| \nonumber \\
    &\qquad \le
    \E\left[
        |\Delta_i|^2|\Delta_i^{[i-1]}|
    \right]
    +
    \E\left[
        (\E_i|\Delta_i|^2)|\Delta_i^{[i-1]}|
    \right]
    \le
    2\|\Delta_i\|_{L^3}^3,
\end{align}
thanks to H\"older's inequality, Jensen's inequality, and the fact that
$\Delta_i^{[i-1]}$ has the same distribution as $\Delta_i$.

Combining \eqref{eq:pf_1}, \eqref{eq:pf_1bis}, \eqref{eq:pf_2}, and
\eqref{eq:remainder-bound}, we get
\begin{align*}
    \left|
    \E[h(G)]-\E[h(N)]
    \right|
    &\le
    \left|
        \E[f'(G)]
        -
        \sum_{i=1}^n
        \E[f'(\E_i[G])c_i]
    \right|
    +
    2\sum_{i=1}^n
    \|\Delta_i\|_{L^3}^3.
\end{align*}
Now
\[
\begin{aligned}
    \E[f'(G)]
    -
    \sum_{i=1}^n
    \E[f'(\E_i[G])c_i]
    &=
    \E\left[
        f'(G)\left(1-\sum_{i=1}^n c_i\right)
    \right]
    +
    \sum_{i=1}^n
    \E\left[
        c_i\bigl(f'(G)-f'(\E_i[G])\bigr)
    \right].
\end{aligned}
\]
Therefore, using $\|f'\|_\infty\le\sqrt{2/\pi}$,
\begin{align}
\label{eq:pf_4}
    \left|
    \E[h(G)]-\E[h(N)]
    \right|
    &\le
    \sqrt{\frac{2}{\pi}}
    \E\left|
        1-\sum_{i=1}^n c_i
    \right| \nonumber \\
    &\quad+
    \sum_{i=1}^n
    \left|
    \E\left[
        c_i\bigl(f'(G)-f'(\E_i[G])\bigr)
    \right]
    \right|
    +
    2\sum_{i=1}^n
    \|\Delta_i\|_{L^3}^3.
\end{align}

It remains to bound the second term. Since $f'$ is $2$-Lipschitz,
\[
    \left|
    \E\left[
        c_i\bigl(f'(G)-f'(\E_i[G])\bigr)
    \right]
    \right|
    \le
    2\E\left[
        |\Delta_i||c_i|
    \right].
\]
Moreover, by conditional Cauchy--Schwarz,
\[
    |c_i|
    =
    \left|
    \E_i\left[
        \Delta_i\Delta_i^{[i-1]}
    \right]
    \right|
    \le
    \sqrt{
        \E_i[\Delta_i^2] \E_i[(\Delta_i^{[i-1]})^2]
    }.
\]
Hence, by H\"older's inequality and Jensen's inequality,
\[
    \|c_i\|_{L^{3/2}}
    \le
    \|\Delta_i\|_{L^3}^2.
\]
Therefore,
\[
    \left|
    \E\left[
        c_i\bigl(f'(G)-f'(\E_i[G])\bigr)
    \right]
    \right|
    \le
    2\|\Delta_i\|_{L^3}^3.
\]
Plugging this into \eqref{eq:pf_4}, we obtain
\[
    \left|
    \E[h(G)]-\E[h(N)]
    \right|
    \le
    \sqrt{\frac{2}{\pi}}
    \E\left|
        1-\sum_{i=1}^n c_i
    \right|
    +
    4\sum_{i=1}^n
    \E|\Delta_i|^3 .
\]
Taking the supremum over all $1$-Lipschitz functions $h$ yields the result of the theorem.
\end{proof}

\section{Consequences of the Poincar\'e inequality}
\label{sec:preliminaries_data}

We say that a random vector $Z\in \mathbb R^m$ satisfies a Poincar\'e
inequality with constant $C_P$ if, for every sufficiently smooth mapping
$F:\mathbb R^m\to \mathbb R$,
\begin{align}\label{eq:poincare_C_P}
    \Var(F(Z))
    \le
    C_P\E \left[ \norm{\nabla_{Z}F(Z)}^2 \right].
\end{align}
We start with the standard tensorization property of the Poincar\'e inequality;
see, for instance, \cite{Gozlan2010Tensorization,Ledoux2001Concentration}.

\begin{proposition}[Poincar\'e inequality tensorization]
\label{prop:tensorization}
Let $Z_1,\ldots,Z_n\in \mathbb R^p$ be independent random vectors, and denote
$Z := (Z_1,\ldots,Z_n)\in \mathcal{M}_{p,n}$. Assume that each $Z_i$ satisfies
a Poincar\'e inequality with constant $C_P$. Then, for every sufficiently smooth
$F:\mathcal{M}_{p,n}\to \mathbb R$,
\[
    \Var(F(Z))
    \le
    C_P\E \left[  \norm{\nabla F(Z)}_F^2 \right]
    =
    C_P\sum_{i=1}^n \E \left[ \norm{\nabla_{i}F(Z)}^2 \right].
\]
\end{proposition}

\begin{proof}
By the variance decomposition, or equivalently the Efron--Stein inequality for
independent variables,
\[
    \Var(F(Z))
    \le
    \sum_{i=1}^n
    \E\left[
        \Var_{Z_i}\big(F(Z_1,\ldots,Z_n)\big)
    \right],
\]
where $\Var_{Z_i}$ denotes variance with respect to $Z_i$ only, all other
coordinates being fixed. For each $i$, conditionally on
$(Z_1,\ldots,Z_{i-1},Z_{i+1},\ldots,Z_n)$, the map
$z_i\mapsto F(Z_1,\ldots,Z_{i-1},z_i,Z_{i+1},\ldots,Z_n)$ is a function on
$\mathbb R^p$. Applying the Poincar\'e inequality for $Z_i$ gives
\[
    \Var_{Z_i}\big(F(Z)\big)
    \le
    C_P \E_{Z_i}\left[
        \norm{\nabla_i F(Z)}_2^2
    \right].
\]
Taking expectation with respect to the remaining variables and summing over
$i$ yields
\[
    \Var(F(Z))
    \le
    C_P\sum_{i=1}^n
    \E\left[
        \norm{\nabla_i F(Z)}_2^2
    \right]
    =
    C_P\E\left[
        \norm{\nabla F(Z)}_F^2
    \right].
\]
\end{proof}
The tensorization also works for higher moments than the variance.
\begin{lemma}[Tensorized \(L^k\) Poincaré inequality]
\label{lem:tensorized-Lk-poincare}
Let \(Z=(Z_1,\ldots,Z_n)\) have independent blocks. Assume that each block
\(Z_\ell\) satisfies a Poincaré inequality with constant \(C_P\). Then, for
every integer \(k\ge2\) and every sufficiently smooth scalar function \(F\),
\[
    \Norm{F(Z)-\E F(Z)}_{L^k}
    \le
    C\sqrt{C_P}\,k\,
    \left\|
        \left(
            \sum_{\ell=1}^n
            \left(
                \norm{\mathbb D_\ell[F(Z)]}_2^{*}
            \right)^2
        \right)^{1/2}
    \right\|_{L^k},
\]
where \(C>0\) is universal. In particular, for fixed \(k\), the constant is
\(O_{k,P}(1)\).
\end{lemma}

\begin{proof}
By Proposition~\ref{prop:tensorization}, the product vector \(Z\) satisfies
the Poincaré inequality
\[
    \Var(H(Z))
    \le
    C_P
    \E\left[
        \sum_{\ell=1}^n
        \left(
            \norm{\mathbb D_\ell[H(Z)]}_2^{*}
        \right)^2
    \right]
\]
for every sufficiently smooth scalar \(H\). Set
\[
    Y:=F(Z)-\E F(Z),
    \qquad
    S:=
    \left(
        \sum_{\ell=1}^n
        \left(
            \norm{\mathbb D_\ell[F(Z)]}_2^{*}
        \right)^2
    \right)^{1/2}.
\]
For \(q\ge2\), write
\[
    a_q:=\Norm{Y}_{L^q},
    \qquad
    b_q:=\Norm{S}_{L^q}.
\]
Applying the tensorized Poincaré inequality to \(H=|Y|^{q/2}\), with the usual
smooth approximation if necessary, gives
\[
    \Var\!\left(|Y|^{q/2}\right)
    \le
    C_P\frac{q^2}{4}
    \E\left[
        |Y|^{q-2}S^2
    \right].
\]
By Hölder's inequality,
\[
    \E\left[
        |Y|^{q-2}S^2
    \right]
    \le
    a_q^{q-2}b_q^2.
\]
Since
\[
    \E|Y|^q
    =
    \Var\!\left(|Y|^{q/2}\right)
    +
    \left(\E|Y|^{q/2}\right)^2,
\]
we obtain
\[
    a_q^q
    \le
    C_P\frac{q^2}{4}a_q^{q-2}b_q^2
    +
    a_{q/2}^q.
\]
If \(a_q=0\), there is nothing to prove. Otherwise, using \(a_{q/2}\le a_q\),
\[
    a_q^2
    \le
    C_P\frac{q^2}{4}b_q^2
    +
    a_{q/2}^2.
\]
Iterating this inequality until the exponent lies in \((1,2]\), and using the
ordinary tensorized Poincaré inequality at exponent \(2\), gives
\[
    a_q
    \le
    C\sqrt{C_P}\,q\,b_q
\]
for a universal constant \(C\). Taking \(q=k\) proves the claim.
\end{proof}

\begin{lemma}[Linear moments under Poincar\'e]
\label{lem:linear-moments-poincare}
Let $m\in \mathbb N$ and let $Z\in \mathbb R^m$ be a random vector satisfying a
Poincar\'e inequality with constant $C_P$. Then, for any $k\in \mathbb N$,
\begin{align*}
    \sup_{\|u\|\leq 1}\Norm{u^\top (Z - \mathbb E[Z]) }_{L^{k}}
    \le
    6 \sqrt{C_P}\, k.
\end{align*}
\end{lemma}

This lemma allows us to bound the operator norm of the $k$th-order centered
moment tensor $M_k[Z]\in(\mathbb R^m)^{\otimes k}$ of $Z$, defined, for all
$u_1,\ldots,u_k\in \mathbb R^m$, by
\begin{align*}
    M_k[Z]\cdot (u_1,\ldots, u_k)
    =
    \mathbb E \left[ u_1^\top \tilde Z \cdots u_k^\top \tilde Z \right],
\end{align*}
where $\tilde Z := Z - \mathbb E[Z]$. Indeed, H\"older's inequality and
Lemma~\ref{lem:linear-moments-poincare} give the following corollary.

\begin{corollary}\label{cor:bound-moments}
Let $m\in \mathbb N$ and let $Z\in \mathbb R^m$ be a random vector satisfying a
Poincar\'e inequality with constant $C_P$. Then, for any $k\in \mathbb N$,
\begin{align*}
    \norm{M_k[Z]}_{\mathrm{op}}
    :=
    \sup_{\|u_1\|,\ldots,\|u_k\|\le1}
    \left|M_k[Z]\cdot (u_1,\ldots,u_k)\right|
    \le
    \left(6 \sqrt{C_P}\, k\right)^k.
\end{align*}
Equivalently,
\[
    \norm{M_k[Z]}_{\mathrm{op}}^{1/k}
    \le
    6 \sqrt{C_P}\, k.
\]
\end{corollary}

Lemma~\ref{lem:linear-moments-poincare} relies on the following classical result, taken from \cite[Proposition~4.1]{BobkovLedoux1997PoincareTalagrand}.

\begin{lemma}[Bobkov--Ledoux exponential integrability bound]
\label{lem:bobkov_ledoux}
Let $m\in \mathbb N$ and let $Z\in \mathbb R^m$ be a random vector satisfying a
Poincar\'e inequality with constant $C_P$. Let $\kappa>0$. For any bounded
$\kappa$-Lipschitz mapping $f:\mathbb R^m\to \mathbb R$,
\begin{align*}
    \forall 0\le \lambda < \frac{2}{\sqrt{C_P}\kappa}:\qquad
    \E \left[
        \exp\left(\lambda \left\vert f(Z) - \mathbb E[f(Z)]\right\vert\right)
    \right]
    \le
    \frac{4/\sqrt{C_P}+2\lambda\kappa}{2/\sqrt{C_P}-\lambda\kappa}.
\end{align*}
\end{lemma}

\begin{proof}[Proof of Lemma~\ref{lem:linear-moments-poincare}]
Consider a deterministic vector $u\in\R^m$ satisfying $\norm{u}_2 \leq 1$. If
$u=0$, there is nothing to prove. Set
\[
    Y
    :=
    u^\top Z-\E[u^\top Z].
\]
To apply Bobkov--Ledoux's exponential integrability bound, we introduce a
quantity $M>0$ and the mapping $T_M(t):=(-M)\vee(t\wedge M)$, which allows us
to truncate the linear form. Define
\[
    Y^{(M)}
    :=
    T_M(u^\top Z)-\E[T_M(u^\top Z)].
\]
Then $Y^{(M)}$ is bounded, centered, and $\norm{u}_2$-Lipschitz as a function of
$Z$. Lemma~\ref{lem:bobkov_ledoux} applied with
$\lambda=1/(\sqrt{C_P}\norm{u}_2)$ yields
\begin{align}\label{eq:exp_exp_Ym}
    \E \left[
        \exp\!\left(\frac{|Y^{(M)}|}{\sqrt{C_P}\norm{u}_2}\right)
    \right]
    \le
    6.
\end{align}
By the Poincar\'e inequality applied to linear functions, $u^\top Z$ has finite
variance. Hence $T_M(u^\top Z)\to u^\top Z$ in $L^1$, so $Y^{(M)}$ converges to
$Y$ in probability and along a subsequence almost surely. Fatou's lemma allows
us to pass to the limit in~\eqref{eq:exp_exp_Ym}, with $Y$ in place of
$Y^{(M)}$. Therefore, using $t^k\le k!e^t$ for $t\ge0$,
\[
    \forall k\in \mathbb N:\qquad
    \E|Y|^k
    \le
    6k!\,(\sqrt{C_P}\norm{u}_2)^k.
\]
Hence, since $k!\leq k^k$,
\[
    \Norm{u^\top(Z-\E[Z]) }_{L^{k}}
    \le
    6^{1/k}\sqrt{C_P}\,k\norm{u}_2.
\]
Taking the supremum over $\norm{u}_2\le1$ gives the claim.
\end{proof}

\begin{lemma}[Euclidean moments of a random vector]
\label{lem:column-euclidean-moments}
Let $m\in \mathbb N$ and let $Z\in \mathbb R^m$ be a random vector. Then, for
any $k\geq 2$,
\begin{align*}
    \Norm{\norm{Z}_2 }_{L^{k}}
    \leq
    \sqrt m \sup_{\|u\|\leq 1} \Norm{u^\top Z  }_{L^k}.
\end{align*}
\end{lemma}

\begin{proof}
Let \((e_a)_{a\le m}\) be an orthonormal basis of \(\R^m\). For \(k\ge2\),
Minkowski's inequality gives
\[
    \left\|
        \norm{Z}_2
    \right\|_{L^k}
    =
    \left\|
        \left(
            \sum_{a=1}^m |e_a^\top Z|^2
        \right)^{1/2}
    \right\|_{L^k}
    \le
    \left(
        \sum_{a=1}^m
        \Norm{e_a^\top Z  }_{L^k}^2
    \right)^{1/2}.
\]
The stated bound follows by taking the supremum over all unit vectors.
\end{proof}

The next lemma shows that maxima of a family of random variables satisfying
similar linear moment-growth bounds grow only logarithmically with the size of
the family. Note that this type of uniform control is not available from
fixed-order moment assumptions alone.


\begin{lemma}[Maxima under polynomial moment growth]
\label{lem:max-polynomial-moment-growth}
Let \(\alpha>0\). Assume that \(Z_1,\ldots,Z_m\) are real-valued random
variables and that there exists \(C>0\) such that, for every integer \(q\ge1\),
\[
    \max_{\ell\in[m]}
    \Norm{Z_\ell}_{L^q}
    \le
    Cq^\alpha .
\]
Then, for every integer \(r\ge1\),
\[
    \left\|
        \max_{\ell\in[m]}|Z_\ell|
    \right\|_{L^r}
    \le
    e\,C\,\left(\left\lceil r\vee\log(em)\right\rceil\right)^\alpha .
\]
\end{lemma}

\begin{proof}
Let
\[
    q:=\left\lceil r\vee\log(em)\right\rceil .
\]
Then \(q\ge r\), and
\[
\begin{aligned}
    \left\|
        \max_{\ell\in[m]}|Z_\ell|
    \right\|_{L^r}
    &\le
    \left\|
        \max_{\ell\in[m]}|Z_\ell|
    \right\|_{L^q}                                      \\
    &\le
    \left(
        \sum_{\ell=1}^m \Norm{Z_\ell}_{L^q}^q
    \right)^{1/q}
    \le
    Cq^\alpha m^{1/q}.
\end{aligned}
\]
Since \(q\ge\log(em)\), we have \(m^{1/q}\le e\), which proves the claim.
\end{proof}
\begin{lemma}[Conditional \(L^k\) Poincar\'e bound]
\label{lem:conditional-Lk-poincare}
Let \(Z=(Z_1,\ldots,Z_n)\) be a vector of independent blocks. Assume that the
block \(Z_j\) satisfies a Poincar\'e inequality with constant \(C_P\). Then, for
every integer \(k\ge2\) and every sufficiently smooth scalar function \(F\),
\[
    \left\|
        F(Z)-\E_j[F(Z)]
    \right\|_{L^k}
    \le
    C_{k,P}
    \left\|
        \norm{\mathbb D_j[F]}_2^{*}
    \right\|_{L^k},
\]
where \(C_{k,P}\) depends only on \(k\) and \(C_P\). More precisely, one can take
\(C_{k,P}=O(k\sqrt{C_P})\).
\end{lemma}

\begin{proof}
It suffices to prove the corresponding conditional estimate. Fix all blocks
except \(Z_j\), and write
\[
    Y:=F(Z)-\E_j[F(Z)],
    \qquad
    G:=\norm{\mathbb D_j[F]}_2^{*}.
\]
For \(q\ge2\), set
\[
    a_q:=\Norm{Y}_{L_j^q},
    \qquad
    b_q:=\Norm{G}_{L_j^q}.
\]
The Poincar\'e inequality applied conditionally to \(|Y|^{q/2}\) gives
\[
    \Var_j\!\left(|Y|^{q/2}\right)
    \le
    C_P\frac{q^2}{4}
    \E_j\!\left[
        |Y|^{q-2}G^2
    \right].
\]
By H\"older's inequality,
\[
    \E_j\!\left[
        |Y|^{q-2}G^2
    \right]
    \le
    a_q^{q-2}b_q^2.
\]
Hence
\[
    \Var_j\!\left(|Y|^{q/2}\right)
    \le
    C_P\frac{q^2}{4}a_q^{q-2}b_q^2.
\]
Since
\[
    \E_j|Y|^q
    =
    \Var_j\!\left(|Y|^{q/2}\right)
    +
    \left(\E_j|Y|^{q/2}\right)^2,
\]
we obtain
\[
    a_q^q
    \le
    C_P\frac{q^2}{4}a_q^{q-2}b_q^2
    +
    a_{q/2}^q.
\]
If \(a_q=0\), there is nothing to prove. Otherwise, dividing by
\(a_q^{q-2}\) and using \(a_{q/2}\le a_q\), we get
\[
    a_q^2
    \le
    C_P\frac{q^2}{4}b_q^2
    +
    a_{q/2}^2.
\]
Iterating this inequality until the exponent lies in \((1,2]\), and using the
ordinary Poincar\'e inequality at exponent \(2\), gives
\[
    a_q
    \le
    C\sqrt{C_P}\,q\,b_q
\]
for a universal constant \(C\). Taking \(q=k\), we have shown that, almost
surely with respect to the remaining blocks,
\[
    \Norm{
        F(Z)-\E_j[F(Z)]
    }_{L_j^k}
    \le
    C\sqrt{C_P}\,k\,
    \Norm{
        \norm{\mathbb D_j[F]}_2^{*}
    }_{L_j^k}.
\]
Raising this inequality to the power \(k\) and integrating over the remaining
blocks yields
\[
    \left\|
        F(Z)-\E_j[F(Z)]
    \right\|_{L^k}
    \le
    C\sqrt{C_P}\,k\,
    \left\|
        \norm{\mathbb D_j[F]}_2^{*}
    \right\|_{L^k}.
\]
This proves the claim.
\end{proof}

\begin{lemma}[Operator norm of a Poincar\'e data matrix]
\label{lem:operator-norm-poincare-matrix}
Let \(Z=(z_1,\ldots,z_n)\in\mathcal M_{p,n}\) have independent columns. Assume
that each \(z_i\) satisfies a Poincar\'e inequality with constant \(C_P=O(1)\),
and that
\[
    \sup_{i\in[n]}\norm{\E z_i}_2=O(1).
\]
Then, for every fixed integer \(k\ge1\),
\[
    \Norm{\norm{Z}_{\op}}_{L^k}
    =
    O_k\!\left(\sqrt{p+n}\right).
\]
\end{lemma}
This lemma relies on the following result taken from~\cite{AdamczakLitvakPajorTomczak2011Covariance} (one could also use $\varepsilon$-nets to prove it)
\begin{lemma}[Empirical-covariance input]
\label{lem:empirical-covariance-input}
Let \(Y=(y_1,\ldots,y_n)\in\mathbb R^{p\times n}\) have independent centered
columns. Assume that, for some constant \(K=O(1)\),
\[
    \sup_{i\in[n]}\sup_{\|u\|_2\le1}\sup_{q\ge1}
    q^{-1}\Norm{u^\top y_i}_{L^q}
    \le K,
    \qquad
    \sup_{i\in[n]}\E\|y_i\|_2^2\le K^2p .
\]
Then there exists a constant \(C=C(K)\) such that
\[
    \PP\!\left(
        \|YY^\top\|_{\op}>C(p+n)
    \right)
    \le
    \frac14 .
\]
\end{lemma}
\begin{proof}[Proof of Lemma~\ref{lem:operator-norm-poincare-matrix}]
Write
\[
    m_i:=\E z_i,
    \qquad
    M:=(m_1,\ldots,m_n),
    \qquad
    Y:=Z-M=(y_1,\ldots,y_n).
\]
The deterministic mean part satisfies
\[
    \norm{M}_{\op}
    \le
    \norm{M}_F
    =
    \left(\sum_{i=1}^n\norm{m_i}_2^2\right)^{1/2}
    =
    O(\sqrt n)
    \le
    O\!\left(\sqrt{p+n}\right).
\]
It remains to control \(Y\).

For every \(u\in\R^p\) with \(\norm{u}_2\le1\), the Poincar\'e inequality
applied to the linear map \(x\mapsto u^\top x\) gives
\[
    \Var(u^\top z_i)\le C_P.
\]
Hence
\[
    \norm{
        \sum_{i=1}^n \E[y_i y_i^\top]
    }_{\op}
    \le
    C_P n.
\]
Moreover, Lemma~\ref{lem:linear-moments-poincare} gives the uniform
sub-exponential moment bound
\[
    \sup_{i\in[n]}\sup_{\norm{u}_2\le1}\sup_{q\ge1}
    \frac1q
    \Norm{u^\top y_i}_{L^q}
    =
    O(1).
\]
Furthermore,
\[
    \E\norm{y_i}_2^2
    =
    \tr\!\big(\Cov(z_i)\big)
    \le
    C_Pp,
\]
and Lemma~\ref{lem:linear-moments-poincare} gives
\[
    \sup_{i\in[n]}\sup_{\|u\|_2\le1}\sup_{q\ge1}
    q^{-1}\Norm{u^\top y_i}_{L^q}
    =
    O(1).
\]
Therefore Lemma~\ref{lem:empirical-covariance-input} yields, for a constant
\(C\) depending only on \(C_P\),
\[
    \PP\!\left(
        \norm{YY^\top}_{\op}
        >
        C(p+n)
    \right)
    \le
    \frac14 .
\]
\[
    \PP\!\left(
        \norm{YY^\top}_{\op}
        >
        C(p+n)
    \right)
    \le
    \frac14 .
\]
Equivalently, with
\[
    R:=\norm{Y}_{\op},
\]
we have
\[
    \PP\!\left(
        R
        \le
        C\sqrt{p+n}
    \right)
    \ge
    \frac34 .
\]

On the other hand, the map \(Y\mapsto\norm{Y}_{\op}\) is \(1\)-Lipschitz with
respect to the Frobenius norm. Since the columns of \(Y\) are independent and
each satisfies a Poincar\'e inequality with constant \(C_P\), tensorization and
the Bobkov--Ledoux moment consequence give, for every fixed integer \(k\ge1\),
\[
    \Norm{
        R-\E R
    }_{L^k}
    =
    O_k(1).
\]
Here one may apply Lemma~\ref{lem:bobkov_ledoux} first to the bounded
truncations \(R\wedge A\), and then let \(A\to\infty\).

We now combine this concentration around the mean with the preceding
positive-probability bound. Let
\[
    a:=C\sqrt{p+n}.
\]
If \(\E R>a+t\), then
\[
    \frac34
    \le
    \PP(R\le a)
    \le
    \PP(|R-\E R|\ge t)
    \le
    \frac{\Norm{R-\E R}_{L^k}^k}{t^k}.
\]
Since \(\Norm{R-\E R}_{L^k}=O_k(1)\), this implies \(t=O_k(1)\). Therefore
\[
    \E R
    =
    O_k\!\left(\sqrt{p+n}\right).
\]
Consequently,
\[
    \Norm{R}_{L^k}
    \le
    \E R+\Norm{R-\E R}_{L^k}
    =
    O_k\!\left(\sqrt{p+n}\right).
\]
Finally,
\[
    \Norm{\norm{Z}_{\op}}_{L^k}
    \le
    \norm{M}_{\op}
    +
    \Norm{\norm{Y}_{\op}}_{L^k}
    =
    O_k\!\left(\sqrt{p+n}\right).
\]
Under Assumption~\ref{ass:1dim}, \(p=O(n)\), and the final bound becomes
\(O_k(\sqrt n)\).
\end{proof}

\bibliographystyle{plain}
\bibliography{biblio}

\end{document}